\documentclass[a4paper,12pt]{amsart}
\usepackage{lipsum}
\usepackage{amssymb,amsthm}
\usepackage[foot]{amsaddr}
\usepackage{amsmath}
\usepackage{tikz}
\usetikzlibrary{quotes,angles}
\usepackage{graphicx}
\usepackage{subcaption}
\usepackage[shortlabels]{enumitem}
\usetikzlibrary{arrows}
\usepackage{float}
\usepackage{graphicx}
\usepackage{subcaption}
\usepackage{hyperref}
\usepackage[section]{placeins}
\graphicspath{{Figures/}}
\usetikzlibrary{decorations.pathmorphing}
\tikzset{snake it/.style={decorate, decoration=snake}}

\usepackage{soul}

\usepackage{lineno}

\makeatletter
\newcommand*{\rom}[1]{\expandafter\@slowromancap\romannumeral #1@}
\makeatother

\numberwithin{equation}{section}

\theoremstyle{plain}
\newtheorem{theorem}{Theorem}
\newtheorem{corollary}[theorem]{Corollary}
\numberwithin{theorem}{section}

\theoremstyle{definition}

\theoremstyle{remark}
\newtheorem{remark}[theorem]{Remark}
\theoremstyle{remark}
\newtheorem{discussion}[theorem]{Discussion}

\DeclareMathOperator*{\argmin}{arg\,min}
\DeclareMathOperator{\sinc}{sinc}

\usepackage{makecell}

\usepackage{fullpage}

\newcommand{\paren}[1]{\left(#1\right)}               

\newcommand{\tred}[1]{{\color{red}#1}}
\newcommand{\tblue}[1]{{\color{blue}#1}}

\newcommand{\vb}{\mathbf{b}}

\newcommand{\vu}{\mathbf{u}}

\newcommand{\vx}{\mathbf{x}}
\newcommand{\vy}{\mathbf{y}}

\newcommand{\be}{\begin{equation}}
\newcommand{\ee}{\end{equation}}
\newcommand{\bea}{\begin{eqnarray}}
\newcommand{\eea}{\end{eqnarray}}
\newcommand{\bean}{\begin{eqnarray*}}
\newcommand{\eean}{\end{eqnarray*}}

\newcommand{\bel}[1]{\begin{equation}\label{#1}}

\newcommand{\eel}[1]{{\label{#1}\end{equation}}}

\newcommand\irregularcircle[2]{
  \pgfextra {\pgfmathsetmacro\len{(#1)+rand*(#2)}}
  +(0:\len pt)
  \foreach \a in {10,20,...,350}{
    \pgfextra {\pgfmathsetmacro\len{(#1)+rand*(#2)}}
    -- +(\a:\len pt)
  } -- cycle
}

\newcommand{\vxi}{{\boldsymbol{\xi}}}
\newcommand{\vom}{{\boldsymbol{\omega}}}

\usepackage{datetime}

\title[short]{Generalized Abel equations and applications to translation invariant Radon transforms\\{\footnotesize\ddmmyyyydate\today~\currenttime}}
\author{James W. Webber\textsuperscript{$\dagger$}}
\address[James W. Webber (corresponding author) 
]{Department of Oncology and Gynecology, Brigham and Womens Hospital, 221 Longwood Ave. Boston, MA 02115}
\email[A1,A2]{jwebber5@bwh.harvard.edu\textsuperscript{$\dagger$} 
}

\providecommand{\keywords}[1]
{
  \small	
  \textbf{\textit{Keywords---}} #1
}

\begin{document}
\maketitle
\begin{abstract}
Generalized Abel equations have been employed in the recent literature to invert Radon transforms which arise in a number of important imaging applications, including Compton Scatter Tomography (CST), Ultrasound Reflection Tomography (URT), and X-ray CT. In this paper, we present novel injectivity results and inversion methods for Generalized Abel operators. We apply our theory to a new Radon transform, $\mathcal{R}_j$, of interest in URT, which integrates a square integrable function of compact support, $f$, over ellipsoid and hyperboloid surfaces with centers on a plane. Using our newly established theory on generalized Abel equations, we show that $\mathcal{R}_j$ is injective and provide an inversion method based on Neumann series.  In addition, using algebraic methods, we present image phantom reconstructions from $\mathcal{R}_jf$ data with added pseudo random noise.
\end{abstract}
 - Abel equations, Radon transforms, uniqueness of solution


\section{introduction} 
In this paper, we present novel injectivity results and inversion methods for a class of generalized Abel operators, which have important applications in URT, and to Radon transforms which are translation invariant.  Generalized Abel equations are a special class of Volterra equations, and have been applied in the recent literature to invert Radon transforms \cite{p13,p4,p9,p18,p19,p20,p21}. The literature also includes analysis of the stability of generalized Abel operators on Sobolev scale \cite{p30}. 

In \cite{p18}, the authors consider a circular Radon transform, $R$, which has applications in Thermoacoustic Tomography (TAT) and Photo Acoustic Tomography (PAT). Using classical Volterra integral equations theory, the authors prove that $R$ is injective on the domain of continuous functions which are compactly supported on an annulus. 

In \cite{p4}, the authors apply general theory on weakly singular Volterra equations to invert a class of rotation invariant Radon transforms in $n$-dimensions. They authors consider generalized Abel type equations in \cite[equation (3.9)]{p4}, and they provide inversion methods. The integral kernels they consider are smoother than those considered in this paper. The class of kernels we consider are more general than those of \cite{p4}, and our Abel equation theory is a generalization of the results of \cite[pages 515 and 516]{p4}. 

In \cite{p9}, the authors provide inversion methods for a class of generalized Abel equations, and apply them to conic Radon transforms.  They show, in \cite[Theorem 3.4]{p9}, that the generalized Abel equations have unique continuous solution, if certain conditions on the integral kernels are satisfied. This theory is also applied by the authors of \cite{p20} to spindle torus Radon transforms, which are of interest in CST. The uniqueness conditions of \cite[Theorem 3.4]{p9} are different to those presented here, which distinguishes our work from that of \cite{p9}. One key difference, which we discuss in more detail later in section \ref{section_volt} and remark \ref{remark_volt}, is that our integral kernels are zero everywhere on the diagonal, whereas \cite[Theorem 3.4]{p9} requires that the set of zeros on the diagonal be finite.

The literature considers the inversion properties of ellipsoid and hyperboloid Radon transforms \cite{p1,p4,p21,p22,p23,p24,p25,p26,p27,p28,p29}. In \cite{p28}, the authors present support theorems for the spherical Radon transform with centers on a general class of manifolds in $n$-dimensions. This theory applies to generalized functions (distributions). They use microlocal analysis to prove their theorems, and give some example applications, e.g., to sonar. 

In \cite{p26}, the authors present inversion formulae for a spherical means Radon transform, $M$, which integrates smooth functions over spheres with centers on a plane  in $n$-dimensions, when $n$ is even. The authors also provide a full characterization of the range of $M$. Their inversion formulae are of filtered backprojection type, and apply to smooth functions of compact support. All sphere centers on the plane, and all sphere radii are needed to invert $M$ in this way.

In \cite{p25}, the authors present a microlocal analysis of an ellipsoid Radon transform in three dimensions, which has applications in URT. They also provide approximate inversion formulae, and provide in depth microlocal analyses of their properties. A number of numerical experiments and worked examples are conducted to validate their theory, and reconstructions of characteristic functions are presented.


We introduce novel inversion methods for a class of generalized Abel operators. We show, under certain conditions on the operator kernel, that the generalized Abel equations have unique square integrable solution. We apply our theory to ellipsoid and hyperboloid Radon transforms, $\mathcal{R}_j$, which integrate square integrable functions of compact support over ellipsoids and hyperboloids of revolution with centers on a plane in $n$-dimensional space. The stability of $\mathcal{R}_j$ is analyzed in \cite[example 3.4]{p13}, from a microlocal perspective. Inversion formulae are not provided, however. We aim to address this here. We prove that $\mathcal{R}_j$ is injective, and provide an inversion method based on Neumann series. Specifically, after Fourier decomposition, we show that $\mathcal{R}_j$ can be reduced to a set of generalized Abel operators. The operator kernels are then shown to satisfy the specific conditions of our theorems to invert $\mathcal{R}_j$.  We also introduce a novel generalized Radon transform, $\mathcal{R}$, which defines the integrals of a square integrable function over the surfaces of revolution of continuous curves, which are more general than ellipses and hyperbola. In addition to our theory on $\mathcal{R}_j$, we prove injectivity results for $\mathcal{R}$.

The remainder of this paper is organized as follows. In section \ref{prelim}, we state some preliminary results which will be used in our theorems. In section \ref{section_volterra}, we review weakly singular Volterra equations, and present our first main theorem on generalized Abel equations, where we give sufficient conditions on the operator kernel so that there exists a unique, square integrable solution. In section \ref{application}, we apply our theory on generalized Abel equations to $n$-dimensional ellipsoid and hyperboloid Radon transforms, $\mathcal{R}_j$. 
We also give some generalizations of our theory, and provide inversion methods for a novel generalized Radon transform, $\mathcal{R}$, which defines the integrals of $f$ over a more general set of continuous surfaces in $n$-dimensions. In section \ref{results}, we present image reconstructions from $\mathcal{R}_jf$ data using algebraic methods, and discuss the image artifacts. We show that the artifacts can be suppressed effectively using Total Variation (TV) regularization.

\section{Preliminary results}
\label{prelim}
In this section, we state some preliminary results that will be used in our theorems. First, we state a classical result on the solution of Volterra equations of the second kind.

\begin{theorem}[Tricomi \cite{p17}]
\label{tricomi}
Set $j=0$ or $j=1$. Let $f \in L^2([a,b])$, where $-\infty < a < b < \infty$, and let $K \in L^2(T_j)$, where $T_j = \{(p,\omega) \in \mathbb{R}^2 : a\leq p\leq b, a + j(p-a) \leq \omega \leq p + j(b-p)\}$. Let $B_0 = [a,p]$, and let $B_1 = [p,b]$. Then, the Volterra equation of the second kind
\begin{equation}
g(p) = \int_{B_j} K(p,\omega) f(\omega) \mathrm{d}\omega + f(p)
\end{equation}
can be solved uniquely for $f $ if $g$ is known for $p\in[a,b]$.
\end{theorem}

\noindent We now state the general Leibniz rule.

\begin{theorem}[General Leibniz rule \cite{p15}]
\label{general_leibniz}
Let $f : \mathbb{R} \to \mathbb{R}$ and $g : \mathbb{R} \to \mathbb{R}$ be $k$-times differentiable functions. Then
\begin{equation}
\frac{\mathrm{d}^k}{\mathrm{d}x^k} \paren{ f(x) g(x) } = \sum_{i=0}^k\begin{pmatrix} k\\ i\end{pmatrix} f^{(k-i)}(x) g^{(i)}(x),
\end{equation}
where $g^{(i)}= \frac{\mathrm{d}^i}{\mathrm{d}x^i} g$, and $f^{(i)}= \frac{\mathrm{d}^i}{\mathrm{d}x^i} f$.
\end{theorem}

\noindent Finally, we have Fa\`{a} di Bruno's formula.

\begin{theorem}[Fa\`{a} di Bruno's formula \cite{p16}]
\label{di_bruno_thm}
Let $f : \mathbb{R} \to \mathbb{R}$ and $g : \mathbb{R} \to \mathbb{R}$ be $k$-times differentiable functions. Then
\begin{equation}
\frac{\mathrm{d}^k}{\mathrm{d}x^k}  f( g(x) )   = \sum_{\gamma \in \Pi} f^{(|\gamma|)}(g(x))\cdot \prod_{\tau \in \gamma} g^{(|\tau|)}(x),
\end{equation}
where $\Pi$ is the set of all partitions of $\{1,\ldots, k\}$, $g^{(i)}= \frac{\mathrm{d}^i}{\mathrm{d}x^i} g$, and $f^{(i)}= \frac{\mathrm{d}^i}{\mathrm{d}x^i} f$.
\end{theorem}

\section{Volterra equations}
\label{section_volterra}
In this section, we review weakly singular Volterra equations and prove our first main result which provides uniqueness of solution conditions for a class of generalized Abel equations.
\subsection{Weakly singular Volterra equations}
In this section, we review some results from the literature on weakly singular Volterra equations. We have the theorem \cite[page 23]{p7}:
\begin{theorem}
\label{weak_volt_thm}
Set $j=0$ or $j=1$. Let $f \in L^2([a,b])$, where $-\infty < a < b < \infty$, let $\beta \in [0,1)$, and let $K \in L^2(T_j)$, where $T_j = \{(p,\omega) \in \mathbb{R}^2 : a\leq p\leq b, a + j(p-a) \leq \omega \leq p + j(b-p)\}$. Further, let 
the following conditions on $K$ be satisfied:
\begin{enumerate}
\item $K$ and $\frac{\mathrm{d}K}{\mathrm{d}p}$ are continuous on $T_j$.
\item $K(p,p)$ is non-zero, for $p\in [a,b]$.
\end{enumerate}
Then, the weakly singular Volterra equation
\begin{equation}
\label{weak_volt}
g(p) = \int_{B_j(p)} \frac{K(p,\omega)}{ \left[(-1)^j\paren{p-\omega} \right]^{\beta}} f(\omega) \mathrm{d}\omega,
\end{equation}
where $B_0(p) = [a,p]$ and $B_1(p) = [p,b]$, can be solved uniquely for $f$ if $g(p)$ is known for all $p \in [a,b]$. 
\begin{proof}
The case $\beta = 0$ follows directly from theory on standard Volterra equations of the first kind. We now consider the case $\beta \in (0,1)$. Let $D_0(r) = [a,r]$ and $D_1(r) = [r,b]$, where $r\in [a,b]$. Then, from \eqref{weak_volt}, we have
\begin{equation}
\label{weak_volt_1}
\begin{split}
\int_{D_j(r)}\frac{g(p)}{\left[(-1)^j(r-p)\right]^{1-\beta}}\mathrm{d}p &= \int_{D_j(r)}\int_{B_j(p)} \frac{K(p,\omega)}{\left[(-1)^j(p-\omega)\right]^{\beta}\left[(-1)^j(r-p)\right]^{1-\beta}} f(\omega) \mathrm{d}\omega \mathrm{d}p\\
& = \int_{D_j(r)} L_{\beta}(r,\omega) f(\omega) \mathrm{d}\omega,
\end{split}
\end{equation}
where 
$$ L_{\beta}(r,\omega)  = (-1)^j\int_{\omega}^r \frac{K(p,\omega)}{\left[(-1)^j(p-\omega)\right]^{\beta}\left[(-1)^j(r-p)\right]^{1-\beta}} \mathrm{d} p.$$
We now substitute $p = \omega +j(r-\omega) + (-1)^j(r-\omega) t$, and 
\begin{equation}
\begin{split}
L_{\beta}(r, \omega) 
& = \int_{0}^1 \frac{K(\omega +j(r-\omega) + (-1)^j(r-\omega) t,\omega)}{t^{\beta}(1-t)^{1-\beta}} \mathrm{d}t.
\end{split}
\end{equation}
With this, we have
$$L_{\beta}(r,r) =\frac{\pi  K(r,r)}{\sin(\pi \beta)} = c_{\beta}(r),$$
where $\left|c_{\beta}(r)\right|\geq \epsilon$, for all $r \in [a,b]$, for some $\epsilon > 0 $, by conditions (1) and (2). That is, $c_{\beta}$ is bounded away from zero on $[a,b]$. Further, we have
\begin{equation}
\begin{split}
\left|\frac{\mathrm{d}L_{\beta}}{\mathrm{d}r}(r,\omega)\right| &= \left|\int_{0}^1 (-1)^j t \cdot \frac{\frac{\mathrm{d}K}{\mathrm{d}p}(\omega +j(r-\omega) + (-1)^j(r-\omega) t,\omega)}{t^{\beta}(1-t)^{1-\beta}} \mathrm{d}t\right|\\
& \leq \int_{0}^1 t \cdot \frac{\left| \frac{\mathrm{d}K}{\mathrm{d}p}(\omega +j(r-\omega) + (-1)^j(r-\omega) t,\omega)\right|}{t^{\beta}(1-t)^{1-\beta}} \mathrm{d}t\\
& \leq \frac{\pi  M}{\sin(\pi \beta)},
\end{split}
\end{equation}
for $(r,w) \in T_j$, where $\left| \frac{\mathrm{d}K}{\mathrm{d}p}\right| \leq M$ on $T_j$, by condition (1), noting that $T_j$ is compact.

Now 
\begin{equation}
\begin{split}
h(r) &= \frac{1}{c_{\beta}(r)}\paren{\frac{\mathrm{d}}{\mathrm{d}r} \int_{D_j(r)}\frac{g(p)}{\left[(-1)^j(r-p)\right]^{1-\beta}}\mathrm{d}p} \\
&= \frac{1}{c_{\beta}(r)} \int_{D_j(r)} \frac{\mathrm{d}L_{\beta}}{\mathrm{d}r}(r,\omega) f(\omega) \mathrm{d}\omega + (-1)^jf(r)
\end{split}
\end{equation}
we can convert \eqref{weak_volt_1} into a Volterra equation of the second kind with bounded kernel and solve uniquely for $f$ by Theorem \ref{tricomi}.
\end{proof}
\end{theorem}

\subsection{Generalized Abel equations}
\label{section_volt}
In this section, we analyze generalized Abel equations, and prove our first main theorem, which gives sufficient conditions on the integral kernels for uniqueness of solution. 

Set $j=0$ or $j=1$. Then, a \emph{generalized Abel} equation is of the form
\begin{equation}
\label{anti_volt}
g(p) = \int_{B_j(p)} \left[(-1)^j(p-\omega)\right]^{\alpha} K(p,\omega) f(\omega) \mathrm{d}\omega,
\end{equation}
where $K \in L^2(T_j)$, $T_j$
and $B_j(p)$ are defined as in Theorem \ref{weak_volt_thm}, $p \in [a,b]$, $\alpha = m-\beta > -1$ with $m\geq 0$, an integer, and $\beta \in [0,1)$. In \cite[page 43]{p7}, the authors consider the case $K \equiv 1$. In \cite{p8}, the $\alpha = 1/2$ and $K \equiv 1$ case is covered. Weakly singular Volterra equations are a special case of generalized Abel equations, when $m=0$.

We now state and prove our first main theorem. 
\begin{theorem}
\label{anti_volt_thm}
Set $j=0$ or $j=1$. Let $f \in L^2([a,b])$, let $T_j$
be defined as in Theorem \ref{weak_volt_thm}, and let the following conditions on $K$ be satisfied:
\begin{enumerate}
\item $K$ and the functions
$$G_i(p,\omega) = \left[(-1)^j(p-\omega)\right]^{i-1}K^{(i)}(p,\omega), \  \text{for}\  1\leq i \leq m+1,$$
are continuous on $T_j$, where $K^{(i)} = \frac{\mathrm{d}^{i} K}{\mathrm{d}p^{i}}$.
\item $K$ is non-zero on the diagonal, i.e., $K(p,p) \neq 0$ for $p\in [a,b]$.
\end{enumerate}
Then, \eqref{anti_volt} can be solved uniquely for $f$, if $g(p)$ is known for $p\in [a,b]$.
\end{theorem}

\begin{remark}
\label{remark_volt}
The proof uses similar ideas to \cite[page 10]{p7}. The basic idea is to differentiate both sides of \eqref{anti_volt} $m$ times to convert \eqref{anti_volt} into a weakly singular Volterra equation as in \eqref{weak_volt}. Then, Theorem \ref{weak_volt_thm} can be applied to obtain a solution. Note, the general Volterra theory of \cite{p9} cannot be applied to \eqref{anti_volt}, since the kernel vanishes everywhere on the diagonal. In \cite[Theorem 3.1]{p4}, the authors discuss an inversion method for \eqref{anti_volt}, for $K \in C^{\infty}([a,b]^2)$, and when $\alpha = q/2$, with $q \geq -1$ an integer. Our theory is more general, and does not require $K$ to be smooth on $[a,b]^2$, and we consider all real $\alpha > -1$.
\end{remark}

\begin{proof}
The $m=0$ case follows from Theorem \ref{weak_volt_thm}, so we now focus on the $m\geq 1$ case.
Let us define $H_{\alpha} \in C(T_j)$ by
\begin{equation}
H_{\alpha}(p, \omega) = \left[(-1)^j(p-\omega)\right]^{\alpha}K(p,w).
\end{equation}
Then, by Theorem \ref{general_leibniz}, for $1\leq k \leq m$, we have
\begin{equation}
\label{H_a}
\begin{split}
H_{\alpha}^{(k)}(p, \omega) &= \frac{\mathrm{d}^k H_{\alpha}}{\mathrm{d}p^k}(p, \omega) \\
&= \sum_{i=0}^k c_{i,k} \left[(-1)^j(p-\omega)\right]^{\alpha - (k-i)} K^{(i)}(p, \omega)\\
&= \frac{1}{\left[(-1)^j(p-\omega)\right]^{\beta}}\sum_{i=0}^k c_{i,k} \left[(-1)^j(p-\omega)\right]^{m - (k-i)} K^{(i)}(p, \omega)\\
&= \frac{\left[(-1)^j(p-\omega)\right]^{m - k}}{\left[(-1)^j(p-\omega)\right]^{\beta}}\left[c_{0,k} K(p,\omega) + \left[(-1)^j(p-\omega)\right]\sum_{i=1}^k c_{i,k} G_i(p, \omega)\right],
\end{split}
\end{equation}
where $c_{k,k} = 1$, and
\begin{equation}
c_{i,k} = (-1)^{j (k-i)}\begin{pmatrix} k \\ i \end{pmatrix} \prod_{q=0}^{k-1-i}(\alpha - q)
\end{equation}
for $0\leq i \leq k - 1$. Then, $H_{\alpha}^{(k)} \in C(T_j)$, for all $0\leq k\leq m-1$, by condition (1).

We have $H_{\alpha}(p,p) = 0$, for all $p \in [a,b]$, since $K \in C(T_j)$, by condition (1). Further, for $1\leq k \leq m - 1$, $H_{\alpha}^{(k)}(p,p) = 0$ on $[a,b]$ by \eqref{H_a}, as $K, G_i \in C(T_j)$, by condition (1). It follows that
\begin{equation}
\begin{split}
\label{anti_volt_1}
\frac{\mathrm{d}^mg}{\mathrm{d}p^m}(p) & = \int_{B_j} H_{\alpha}^{(m)}(p,\omega) f(\omega) \mathrm{d}\omega \\
& = \int_{B_j} \frac{1}{\left[(-1)^j(p-\omega)\right]^{\beta}}\left[\sum_{i=0}^m c_{i,m} \left[(-1)^j(p-\omega)\right]^{i} K^{(i)}(p, \omega) \right]  f(\omega) \mathrm{d}\omega \\
& = \int_{B_j} \frac{L_{\alpha}(p, \omega)}{\left[(-1)^j(p-\omega)\right]^{\beta}} f(\omega) \mathrm{d}\omega.
\end{split}
\end{equation}
It now remains to show that $L_{\alpha}$ satisfies the conditions of Theorem \ref{weak_volt_thm} to prove the result. 

We have, $L_{\alpha}(p,p) = c_{0,m} K(p,p) \neq 0$, for $p \in [a,b]$, by condition (2). Further,
$$L_{\alpha}(p, \omega) = c_{0,m} K(p,\omega) + \left[(-1)^j(p-\omega)\right]\sum_{i=1}^m c_{i,m} G_i(p, \omega) \in C(T_j),$$
and
\begin{equation}
\begin{split}
\frac{\mathrm{d}L_{\alpha}}{\mathrm{d}p}(p, \omega) & = c_{0,m} K^{(1)}(p,\omega) \\
&\hspace{-1cm}+ \sum_{i=1}^m c_{i,m} \left[(-1)^j i \cdot \left[(-1)^j(p-\omega)\right]^{i-1}K^{(i)}(p, \omega) + \left[(-1)^j(p-\omega)\right]^{i}K^{(i+1)}(p, \omega)\right] \\
& = c_{0,m} G_1(p,\omega) + \sum_{i=1}^m c_{i,m} \left[ (-1)^j i  \cdot G_i(p, \omega) + G_{i+1}(p, \omega)\right] \in C(T_j),
\end{split}
\end{equation}
by condition (1). 
Thus, the conditions of Theorem \ref{weak_volt_thm} are satisfied and this completes the proof.
\end{proof}

\noindent We now have some corollaries.

\begin{corollary}
\label{corr1}
Set $j=0$ or $j=1$. Let $\alpha = m-\beta > -1$ with $m\geq 0$, an integer, and $\beta \in [0,1)$, and let $T_j$ be defined as in Theorem \ref{anti_volt_thm}. Let $K^{(i)} \in C(T_j)$, for all $0\leq i \leq m+1$, and let $K(p,p) \neq 0$ for $p \in [a,b]$. Then, \eqref{anti_volt} has a unique solution $f \in L^2([a,b])$, if $g(p)$ is known for all $p\in [a,b]$.
\begin{proof}
The result follows immediately from Theorem \ref{anti_volt_thm}. 
\end{proof}
\end{corollary}

\begin{corollary}
\label{corr2}
Set $j=0$ or $j=1$. Let $\alpha = m-\beta > -1$ with $m\geq 0$, an integer, and $\beta \in [0,1)$, and let $T_j$
be defined as in Theorem \ref{anti_volt_thm}. Let $K(p,p) \neq 0$ for $p \in [a,b]$, and let $K$ have the form
\begin{equation}
K = \prod_{i=1}^l K_i,
\end{equation}
for some integer $ l\geq 2$, where $K_2^{(i)},\ldots,K_l^{(i)}\in C(T_j)$ 
, for all $0\leq i\leq m+1$. 
Let $K_1$ and the functions
$$G_{1,q}(p,\omega) = \left[(-1)^j(p-\omega)\right]^{q-1}K_1^{(q)}(p,\omega)$$
be continuous on $T_j$, for $1\leq q \leq m+1$. Then, \eqref{anti_volt} has a unique solution $f \in L^2([a,b])$.
\begin{proof}
First, it is clear that $K \in C(T_j)$, by assumption. Let $J = \prod_{i=2}^l K_i$. Then $J^{(i)} \in C(T_j)$ 
, for $0\leq i \leq m+1$, and, by Theorem \ref{general_leibniz}
\begin{equation}
K^{(i)}(p,\omega) = \sum_{q=0}^{i} \begin{pmatrix} i \\ q \end{pmatrix}K_1^{(q)}(p,\omega)J^{(i-q)}(p, \omega).
\end{equation}
Therefore
\begin{equation}
\begin{split}
G_i(p, \omega) & = \left[(-1)^j(p-\omega)\right]^{i-1}K^{(i)}(p,\omega) \\
& = K_1(p,\omega)J^{(i)}(p, \omega) + \sum_{q=1}^{i} \begin{pmatrix} i \\ q \end{pmatrix}\left[(-1)^q(p-\omega)\right]^{i-q}J^{(i-j)}(p, \omega) G_{1,q}(p,\omega),
\end{split}
\end{equation}
which is continuous on $T_j$, for $1\leq i \leq m+1$, by assumption on the $G_{1,q}$, and since $J^{(i)} \in C(T_j)$, for all $0\leq i \leq m+1$. Thus, $K$ satisfies the conditions of Theorem \ref{anti_volt_thm}. This finishes the proof.
\end{proof}
\end{corollary}

\section{Application to Radon transforms}
\label{application}
In this section, we explore the applications of generalized Abel equations to translation invariant Radon transforms.

\subsection{A translation invariant ellipsoid and hyperboloid Radon transform}

Let $\vx = \paren{x_1,\ldots,x_n}$ be our coordinate system, and let $f \in L^2_c\paren{\{0\leq x_n\leq \infty\}}$. In $n=2$ dimensions, we define an ellipse or hyperbola, with center $(y_1,0)\in\mathbb{R}^2$,  by the equation
\begin{equation}
\label{ellipse_equ}
(-1)^j(x_1-y_1)^2+sx_2^2=t,
\end{equation}
where $s>0$ is fixed, and $t>0$. When $j=0$, \eqref{ellipse_equ} is the equation of an ellipse, and when $j=1$, \eqref{ellipse_equ} is the equation for a hyperbola. We now define the elliptic and hyperbolic arc curves
\begin{equation}
\label{ellipse}
\mathcal{C}_j\paren{t,y_1}= \left\{\paren{x_1+y_1,\sqrt{\frac{t+(-1)^{j+1}x_1^2}{s}}} : x_1 \in \Omega_j\right\},
\end{equation}
where $\Omega_0=[-\sqrt{t},\sqrt{t}]$ and $\Omega_1=\mathbb{R}$. See figure \ref{fig0}, for an example elliptic and hyperbolic arc.
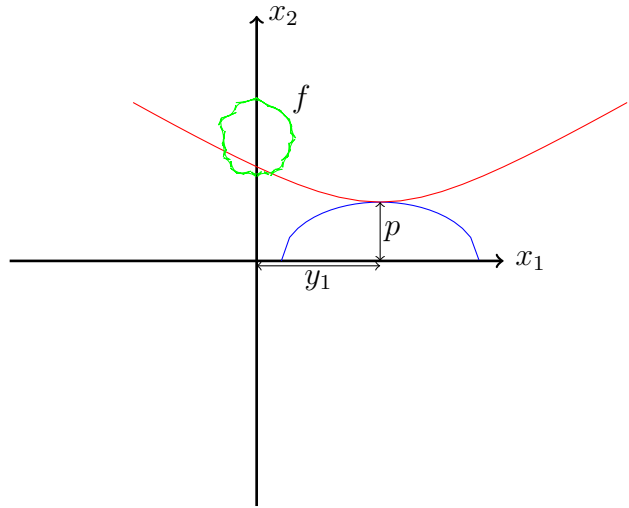
\begin{figure}[!h]
\centering
\begin{tikzpicture}[scale=0.65]
\draw [->,line width=1pt] (-5,0)--(5,0)node[right] {$x_1$};
\draw [->,line width=1pt] (0,-5)--(0,5)node[right] {$x_2$};
\draw [blue,domain=0.5:4.5] plot(\x, {0.6*sqrt(4-pow((\x-2.5),2))});
\draw [red,domain=-2.5:7.5] plot(\x, {0.6*sqrt(4+pow((\x-2.5),2))});
\draw [<->] (0,-0.1)--(2.5,-0.1);
\node at (1.25,-0.4) {$y_1$};
\draw [<->] (2.5,0)--(2.5,1.2);
\node at (2.75,0.6) {$p$};
\coordinate (c) at (0,2.5);
\draw[green,rounded corners=1mm] (c) \irregularcircle{0.75cm}{1mm};
\node at (0.9,3.3) {$f$};
\end{tikzpicture}
\caption{Elliptic arc ($\mathcal{C}_0$) and hyperbola curve ($\mathcal{C}_1$) examples. \tred{Red} - hyperbola. \tblue{Blue} - Elliptic arc. \textcolor{green}{Green} - $f$. $p = \sqrt{t/s}$.}
\label{fig0}
\end{figure}

When $n=2$, we define the Radon transform $\mathcal{R}_j$, for $j=0,1$, as
\begin{equation}
\begin{split}
\mathcal{R}_jf(t,y_1) &= \int_{\mathcal{C}_j\paren{t,y_1}} f \mathrm{d}\mathcal{C}_j \\
&= \int_{\Omega_j}\frac{\sqrt{t+\paren{\frac{1}{s}+(-1)^{j+1}}x_1^2}}{\sqrt{t+(-1)^{j+1}x_1^2}}f\paren{x_1+y_1,\sqrt{\frac{t+(-1)^{j+1}x_1^2}{s}}}\mathrm{d}x_1,
\end{split}
\end{equation}
where $\Omega_0=[-\sqrt{t},\sqrt{t}]$ and $\Omega_1=\mathbb{R}$. 

When $n\geq 3$, the surface of revolution of $\mathcal{C}_j$, with central axis $x_n$ and center $(\vy,0)\in\mathbb{R}^{n}$, is
\begin{equation}
\mathcal{S}_j(t,\vy) = \left\{\paren{r\Theta+\vy,\sqrt{\frac{t+(-1)^{j+1}r^2}{s}}} : \Theta \in S^{n-2}, r\in\Omega^+_j\right\},
\end{equation}
where $\Omega^+_0=[0,\sqrt{t}]$ and $\Omega^+_1=[0,\infty)$. See figure \ref{F3}, some some example $\mathcal{S}_j$ in $n=3$ dimensions. We define
\begin{equation}
\begin{split}
\mathcal{R}_jf(t,\vy) &= \int_{\mathcal{S}_j(t,\vy)} f \mathrm{d}\mathcal{S}_j \\
&= \int_{\Omega^+_j}r^{n-2}\int_{S^{n-2}}\frac{\sqrt{t+\paren{\frac{1}{s}+(-1)^{j+1}}r^2}}{\sqrt{t+(-1)^{j+1}r^2}}f\paren{r\Theta+\vy,\sqrt{\frac{t+(-1)^{j+1}r^2}{s}}}\mathrm{d}S^{n-2}\mathrm{d}r,
\end{split}
\end{equation}
when $n\geq 3$, where $\mathrm{d}S^{n-2}$ denotes the surface measure on $S^{n-2}$. In $n=2$ dimensions, $\mathcal{R}_j$ defines the integrals of $f$ over elliptic and hyperbolic curves with centers on the line $\{x_2=0\}$, where $j=0$ corresponds to ellipses and $j=1$ to hyperbola. In $n\geq 3$ dimensions, $\mathcal{R}_j$ defines the integrals of $f$ over the surfaces of revolution of elliptic and hyperbolic arcs, with axis of revolution parallel to $x_n$, and with centers on the plane $\{x_n=0\}$.

\begin{figure}[!h]
\centering
\begin{subfigure}{0.32\textwidth}
\includegraphics[width=0.9\linewidth, height=4cm, keepaspectratio]{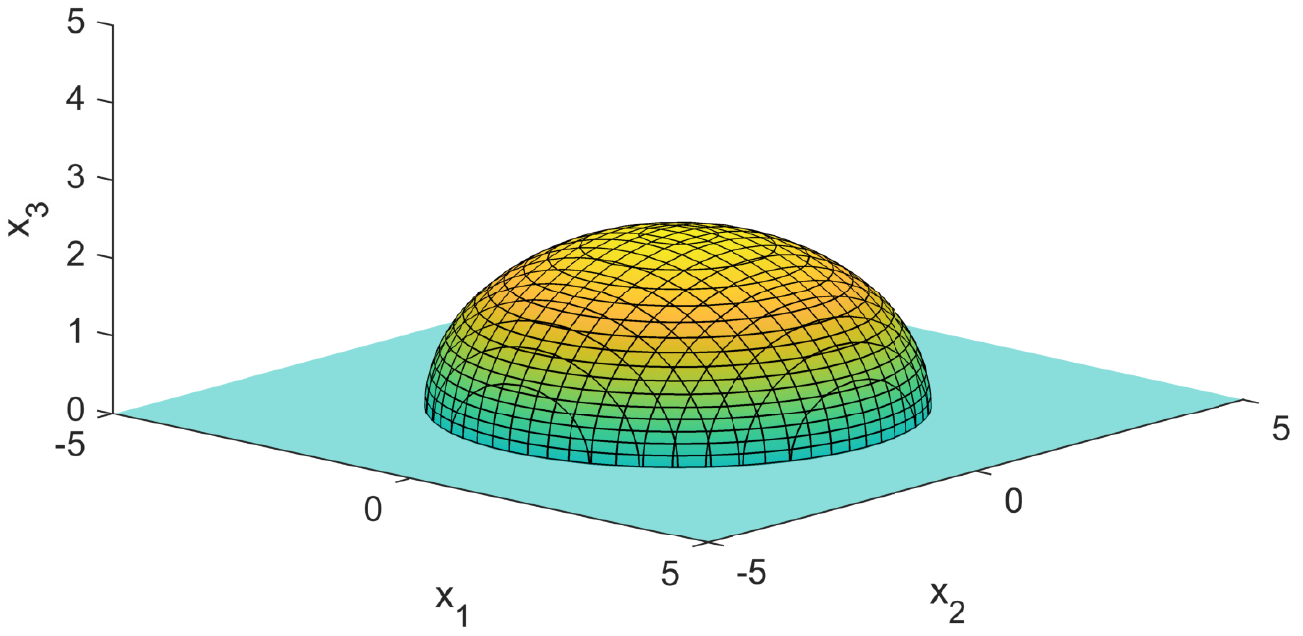}
\subcaption{elliptic arc of revolution ($\mathcal{S}_0$)} \label{F3a}
\end{subfigure}
\hspace{1cm}
\begin{subfigure}{0.32\textwidth}
\includegraphics[width=0.9\linewidth, height=4cm, keepaspectratio]{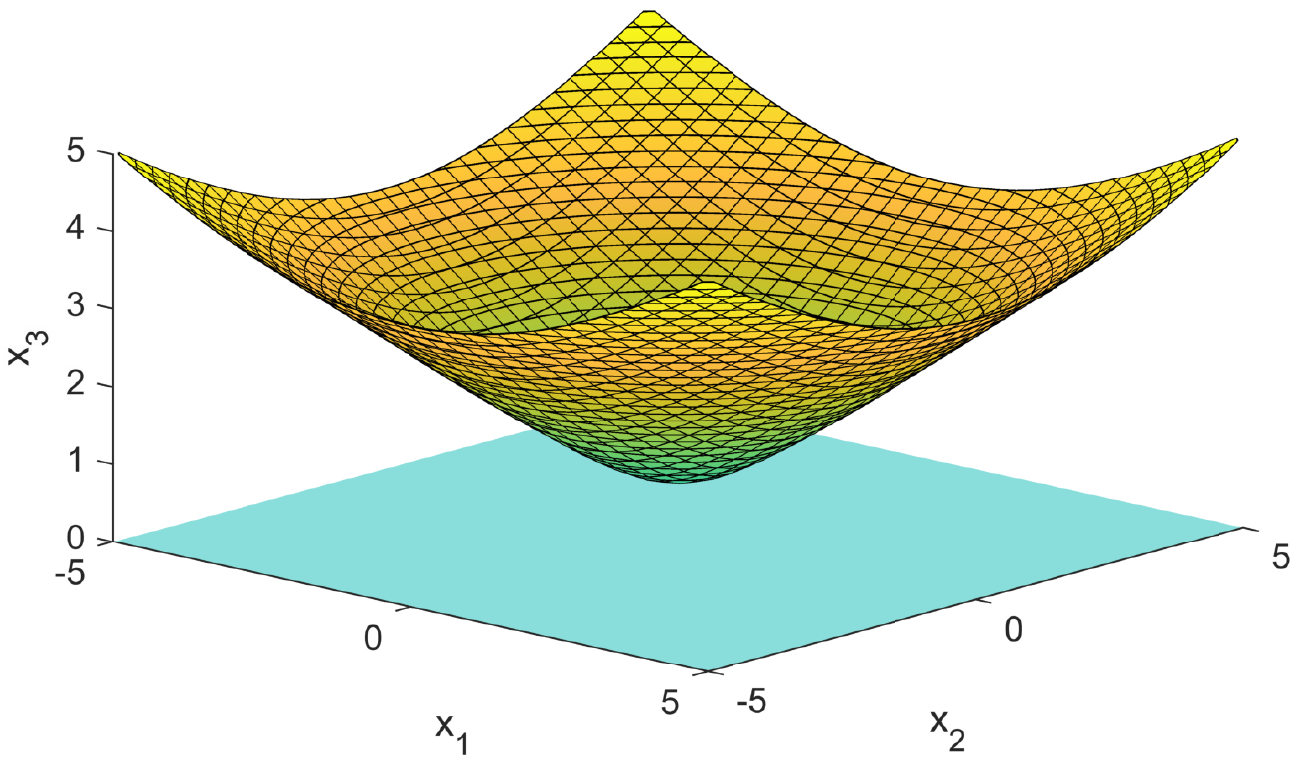} 
\subcaption{hyperboloid ($\mathcal{S}_1$)} \label{F3b}
\end{subfigure}
\caption{Example elliptic arc of revolution and hyperboloid surfaces in $n=3$ dimensions, with centers at the origin. }
\label{F3}
\end{figure}

\noindent We now have our second main theorem.

\begin{theorem}
\label{trans_inv_thm}
$\mathcal{R}_j$ is injective on domain $L^2_c\paren{\{a\leq x_n\leq b\}}$,  for $j=0,1$, where $0<a<b$.
\end{theorem}
\begin{proof}
We split the proof into two parts. First, we prove the $n=2$ case using standard Volterra equation theory. Then, we generalize to $n\geq 3$ dimensions, and use Theorem \ref{anti_volt_thm} to prove the result.\\
\\
\textbf{The $n=2$ case} - 
For this part of the proof, we consider the Radon transform
\begin{equation}
\mathcal{R}_jf(t,y_1)=\int_{\Omega_j}\frac{\sqrt{t+\paren{\frac{1}{s}+(-1)^{j+1}}x_1^2}}{\sqrt{t+(-1)^{j+1}x_1^2}}f\paren{x_1+y_1,\sqrt{\frac{t+(-1)^{j+1}x_1^2}{s}}}\mathrm{d}x_1,
\end{equation}
where $\Omega_0=[-\sqrt{t},\sqrt{t}]$ and $\Omega_1=\mathbb{R}$. Taking the Fourier transform in $y_1$ yields
\begin{equation}
\widehat{\mathcal{R}_jf}(t,\xi)=\int_{\Omega^+_j}\cos(x_1\xi)\frac{\sqrt{t+\paren{\frac{1}{s}+(-1)^{j+1}}x_1^2}}{\sqrt{t+(-1)^{j+1}x_1^2}}\hat{f}\paren{\xi,\sqrt{\frac{t+(-1)^{j+1}x_1^2}{s}}}\mathrm{d}x_1,
\end{equation}
where $\xi$ is dual to $x_1$, $\Omega^+_0=[0,\sqrt{t}]$ and $\Omega^+_1=[0,\infty)$. Making the substitution $x_2=\sqrt{\frac{t+(-1)^{j+1}x_1^2}{s}}$, yields
\begin{equation}
\widehat{\mathcal{R}_jf}(t,\xi)=\int_{B_j}\cos\paren{\xi\sqrt{(-1)^{j}(t-s x_2^2)}}\sqrt{(-1)^{j}\frac{st}{t-sx_2^2}+\paren{1+(-1)^{j+1}s}}\hat{f}\paren{\xi,x_2}\mathrm{d}x_2,
\end{equation}
where $B_0=[0,\sqrt{t/s}]$ and $B_1=[\sqrt{t/s},b]$. Substituting $p=\sqrt{t/s}$ gives
\begin{equation}
\label{volt}
\begin{split}
&\widehat{\mathcal{R}_jf}(p,\xi)=\\
&\int_{B_j}\cos\paren{\xi\sqrt{s}\sqrt{(-1)^{j}\paren{p^2-x_2^2}}}\sqrt{(-1)^{j}s\frac{p^2}{p^2-x_2^2}+\paren{1+(-1)^{j+1}s}}\hat{f}\paren{\xi,x_2}\mathrm{d}x_2\\
&=\int_{B_j}K_j(\xi,p,x_2)\hat{f}\paren{\xi,x_2}\mathrm{d}x_2,
\end{split}
\end{equation}
a weakly singular Volterra equation. Now we can apply Theorem \ref{weak_volt_thm} to recover $f$. To see this, let
$$K_j(\xi,p,x_2)=\frac{H_j(\xi,p,x_2)}{\sqrt{(-1)^{j}(p-x_2)}},$$
with
\begin{equation}
\begin{split}
&H_{j}(\xi,p,x_2)=\\
&\left[\cos\paren{s\xi\sqrt{(-1)^{j}\paren{p^2-x_2^2}}}\right]\left[\sqrt{s\frac{p^2}{p+x_2}+(-1)^{j}\paren{1+(-1)^{j+1}s}(p-x_2)}\right]\\
&= \left[\cos\paren{s\xi\sqrt{(-1)^{j}\paren{p^2-x_2^2}}}\right] \left[ \frac{\sqrt{(-1)^j(p^2-x_2^2)+ sx_2^2}}{\sqrt{p+x_2}} \right]\\
&=h_j(\xi,p,x_2)g_j(p,x_2).
\end{split}
\end{equation}
Let $T_j$ be defined as in Theorem \ref{weak_volt_thm}. We have
$$H_j(\xi,p,p)=\sqrt{\frac{s}{2}}\cdot p\neq 0,$$ for $p\in[a,b]$, $\frac{\mathrm{d}}{\mathrm{d}p}g_j\in C(T_j)$, and
\begin{equation}
\begin{split}
\frac{\mathrm{d}}{\mathrm{d}p}h_{j}(\xi,p,x_2)&=\frac{(-1)^{j+1}\xi p \sqrt{s}}{\sqrt{(-1)^{j}(p^2-x_2^2)}}\sin\paren{\xi \sqrt{s}\sqrt{(-1)^{j}\paren{p^2-x_2^2}}}\\
&=(-1)^{j+1}p s\xi^2\sinc\paren{\xi \sqrt{s}\sqrt{(-1)^{j}\paren{p^2-x_2^2}}}.
\end{split}
\end{equation}
Thus, $\frac{\mathrm{d}}{\mathrm{d}p}h_j(\xi,\cdot,\cdot)\in C(T_j)$ and $\frac{\mathrm{d}}{\mathrm{d}p}H_j(\xi,\cdot,\cdot)\in C(T_j)$, for any fixed $\xi\in\mathbb{R}$, and for $j=0,1$. The result follows from Theorem \ref{weak_volt_thm}.\\
\\
\textbf{The $n\geq 3$ case} - 
For this part of the proof, we consider
\begin{equation}
\begin{split}
&\mathcal{R}_jf(t,\vy)=\\
&\int_{\Omega^+_j}r^{n-2}\int_{S^{n-2}}\frac{\sqrt{t+\paren{\frac{1}{s}+(-1)^{j+1}}r^2}}{\sqrt{t+(-1)^{j+1}r^2}}f\paren{r\Theta+\vy,\sqrt{\frac{t+(-1)^{j+1}r^2}{s}}}\mathrm{d}S^{n-2}\mathrm{d}r,
\end{split}
\end{equation}
where $\mathrm{d}S^{n-2}$ is the surface measure on $S^{n-2}$ and $\Theta\in S^{n-2}$. Taking the Fourier transform in $\vy$ yields
\begin{equation}
\begin{split}
&\widehat{\mathcal{R}_j}(t,\vxi)=\\
&\int_{\Omega^+_j}r^{n-2}\frac{\sqrt{t+\paren{\frac{1}{s}+(-1)^{j+1}}r^2}}{\sqrt{t+(-1)^{j+1}r^2}}\left[\int_{S^{n-2}}e^{i r(\Theta\cdot\vxi)}\mathrm{d}S^{n-2}\right]\hat{f}\paren{\vxi,\sqrt{\frac{t+(-1)^{j+1}r^2}{s}}}\mathrm{d}r.
\end{split}
\end{equation}
Substituting $\omega=\sqrt{\frac{t+(-1)^{j+1}r^2}{s}}$ gives
\begin{equation}
\label{volty_2}
\begin{split}
\widehat{\mathcal{R}_jf}(p,\vxi)&=s^{\frac{n-2}{2}}\int_{B_j}\left[(-1)^{j}\paren{p^2-\omega^2}\right]^{\frac{n-3}{2}}\sqrt{(-1)^j(p^2 -\omega^2) + s\omega^2}\\
&\times\left[\int_{S^{n-2}}e^{i \sqrt{s} \sqrt{(-1)^{j}\paren{p^2-\omega^2}}(\Theta\cdot\vxi)}\mathrm{d}S^{n-2}\right]\hat{f}\paren{\vxi,\omega}\mathrm{d}\omega \\
&=s^{\alpha+\frac{1}{2}}\int_{B_j}\left[(-1)^{j}\paren{p-\omega}\right]^{\alpha}\left[\prod_{i=1}^3K_i(p,\omega)\right]\hat{f}\paren{\vxi,\omega}\mathrm{d}\omega,
\end{split}
\end{equation}
a generalized Abel equation, where $\alpha = \frac{n-3}{2} \geq 0$, $B_0=[0,p]$, $B_1=[p,b]$, and $p=\sqrt{t/s}$. The kernels $K_i$, for $i=1,2,3$, are defined
\begin{equation}
\begin{split}
K_1(p,\omega) &= \int_{S^{n-2}}e^{i \sqrt{s} \sqrt{(-1)^{j}\paren{p^2-\omega^2}}(\Theta\cdot\vxi)}\mathrm{d}S^{n-2}\\
&= \alpha_n\int_{0}^{\pi} e^{i \sqrt{s} |\vxi| \sqrt{(-1)^{j}\paren{p^2-\omega^2}}\cos \varphi_{n-2}}\sin^{n-3}(\varphi_{n-2})\mathrm{d}\varphi_{n-2}\\
&= \alpha_n\int_{0}^{\pi} \cos\paren{\sqrt{s} |\vxi| \sqrt{(-1)^{j}\paren{p^2-\omega^2}}\cos \varphi_{n-2}}\sin^{n-3}(\varphi_{n-2})\mathrm{d}\varphi_{n-2},
\end{split}
\end{equation}
where $\alpha_3=2$, $\alpha_4=2\pi$, and
\begin{equation}
\begin{split}
\alpha_n&=\int_{0}^{2\pi}\left[\int_{[0,\pi]^{n-4}}\prod_{i=1}^{n-4}\sin^{i}(\varphi_i)\mathrm{d}\varphi_1\ldots\mathrm{d}\varphi_{n-4}\right]\mathrm{d}\varphi_{n-3}\\
&=2\pi\int_{[0,\pi]^{n-4}}\prod_{i=1}^{n-4}\sin^{i}(\varphi_i)\mathrm{d}\varphi_1\ldots\mathrm{d}\varphi_{n-4},
\end{split}
\end{equation}
for $n\geq 5$, where $\varphi_1,\ldots,\varphi_{n-2}$ are standard spherical coordinates, with north pole $\frac{\vxi}{|\vxi|}$, which parametrize the unit $(n-2)$-sphere. The kernel $K_2$ is defined by
\begin{equation}
K_2(p,\omega) = (p + \omega)^{\alpha},
\end{equation}
and
\begin{equation}
\begin{split}
K_3(p,\omega) 
&= \sqrt{(-1)^j(p^2 -\omega^2) + s\omega^2}.
\end{split}
\end{equation}
Now,
$$ K(p,p) = \left[\int_{S^{n-2}}\mathrm{d}S^{n-2}\right] (2p)^{\alpha} \paren{p \sqrt{s}} = \paren{2^{\alpha}A_{n-2}\sqrt{s}} p^{\alpha + 1}  \neq 0, \ \text{for all}\ p\in[a,b],$$
where $A_{n-2}$ is the surface area of $S^{n-2}$. Let $T_j$ be defined as is Theorem \ref{weak_volt_thm}. Then $K^{(i)}_2, K^{(i)}_3 \in C(T_j)$, for any $i\geq 0$.

We have $K_1 \in C(T_j)$, and 
\begin{equation}
\begin{split}
&K^{(1)}_1(p,\omega) = \\
& \frac{\alpha_n \sqrt{s}|\vxi| p}{ (-1)^{j+1}}  \int_{0}^{\pi} \frac{ \sin\paren{\sqrt{s} |\vxi| \sqrt{(-1)^{j}\paren{p^2-\omega^2}}\cos \varphi_{n-2}} }{ \sqrt{(-1)^{j}\paren{p^2-\omega^2}} } \cos \varphi_{n-2} \sin^{n-3}(\varphi_{n-2})\mathrm{d}\varphi_{n-2} \\
&= \frac{\alpha_n s|\vxi|^2 p }{ (-1)^{j+1}} \int_{0}^{\pi} \sinc\paren{\sqrt{s} |\vxi| \sqrt{(-1)^{j}\paren{p^2-\omega^2}}\cos \varphi_{n-2}} \cos^2 \varphi_{n-2} \sin^{n-3}(\varphi_{n-2})\mathrm{d}\varphi_{n-2}.
\end{split}
\end{equation}
Define $J \in C^{\infty} (\mathbb{R})$ as
$$J(u) = \alpha_n (-1)^{j+1}s|\vxi|^2  \int_{0}^{\pi} \sinc\paren{\sqrt{s} |\vxi| u \cos \varphi_{n-2}} \cos^2 \varphi_{n-2} \sin^{n-3}(\varphi_{n-2})\mathrm{d}\varphi_{n-2}.$$
Then,
$$G_{1,1}(p,\omega) = K^{(1)}_1(p,\omega) = p \cdot J\paren{\sqrt{(-1)^{j}\paren{p^2-\omega^2}}} \in C(T_j).$$
Letting $g(p,\omega) = \sqrt{(-1)^{j}\paren{p^2-\omega^2}}$, we have, by Theorem \ref{di_bruno_thm}
\begin{equation}
\begin{split}
\label{Bruno}
\paren{J \circ g}^{(i)}(p,\omega) &= \frac{\mathrm{d}^i}{\mathrm{d}p^i} \paren{J \circ g}(p,\omega)\\
&= \sum_{\gamma \in \Pi} J^{(|\gamma|)}\paren{ g(p,\omega) } \cdot \prod_{\tau \in \gamma} g^{(|\tau|)}(p,\omega),
\end{split}
\end{equation}
where $\Pi$ is the set of partitions of $\{1,\ldots,i\}$, and $i \geq 1$. 

We have
\begin{equation}
g^{(q)}(p,\omega) = P_{j,q}(p,\omega) \left[(-1)^{j}\paren{p^2-\omega^2}\right]^{\frac{1}{2} - q},
\end{equation}
for $q\geq 0$, where the $P_{j,q} \in C^{\infty}(\mathbb{R}^2)$ are polynomials in $p$ and $\omega$. We can define the $P_{j,q}$ recursively
$$ P_{j,q}(p, \omega) = (p^2 - \omega^2) P^{(1)}_{j,q-1}(p, \omega) - \paren{2q-3}p \cdot P_{j,q-1}(p, \omega),\ \ \text{for}\ \ j\geq 2,$$
with $P_{j,1}(p,\omega) = (-1)^j p$, and $P_{j,0}(p,\omega) = 1$. 
It follows that, for $\gamma \in \Pi$,
\begin{equation}
\begin{split}
\prod_{\tau \in \gamma} g^{(|\tau|)}(p,\omega) &= \prod_{\tau \in \gamma} P_{j, |\tau|}(p,\omega) \left[(-1)^{j}\paren{p^2-\omega^2}\right]^{\frac{1}{2} - |\tau|}\\
&= \left[(-1)^{j}\paren{p^2-\omega^2}\right]^{\frac{|\gamma|}{2} - i} \prod_{\tau \in \gamma} P_{j, |\tau|}(p,\omega).
\end{split}
\end{equation}
Now, \eqref{Bruno} becomes
\begin{equation}
\begin{split}
\label{Bruno_1}
&\paren{J \circ g}^{(i)}(p,\omega) = \\
&\ \ \ \  \frac{1}{ \left[(-1)^{j}\paren{p^2-\omega^2}\right]^{i} }\sum_{\gamma \in \Pi} J^{(|\gamma|)}\paren{ g(p,\omega) } \cdot \left[(-1)^{j}\paren{p^2-\omega^2}\right]^{\frac{|\gamma|}{2}} \prod_{\tau \in \gamma} P_{j, |\tau|}(p,\omega).
\end{split}
\end{equation}
Hence,
$$H_{i,q}(p,\omega) = \left[(-1)^{j}\paren{p^2-\omega^2}\right]^{i} \paren{J \circ g}^{(q)}(p,\omega) \in C(T_j),$$
for all $1\leq q\leq i$.

Now, using Theorem \ref{general_leibniz}, we calculate the higher order derivatives of $K_1$,
\begin{equation}
K^{(i)}_1(p,\omega) = p \cdot \paren{J \circ g}^{(i-1)}(p,\omega) + (i-1) \paren{J \circ g}^{(i-2)}(p,\omega),
\end{equation}
for $i \geq 2$. It follows that
\begin{equation}
\begin{split}
G_{1,i}(p,\omega) &= \left[(-1)^j(p-\omega)\right]^{i-1}K_1^{(i)}(p,\omega)\\
&= p H_{i-1,i-1}(p,\omega) + (i-1) H_{i-1,i-2}(p,\omega) \in C(T_j),
\end{split}
\end{equation}
for $i\geq 2$. Thus, $K$ satisfies the conditions of Corollary \ref{corr2}, and we can solve \eqref{volty_2} for $\hat{f}(\vxi,\cdot) \in L^2_c([a,b])$, for any fixed $\vxi \in \mathbb{R}^{n-1}$. After which, $f \in L^2_c\paren{ \{ a \leq x_n \leq b\} }$ can be recovered uniquely by inverse Fourier transform.
\end{proof}

\begin{remark}
Let $f \in L^2_c\paren{ [-c,c]^{n-1} \times [a,b] }$, where $c>0$, and $0<a<b$, and let us consider the special case when $s=1$ and $j=0$. In this case, $\mathcal{R}_0f(t,\vy)$ defines the integrals of $f$ over spheres, radius $\sqrt{t}$, center $\vy$. By Theorem \ref{trans_inv_thm}, to reconstruct $f$ uniquely from $\mathcal{R}_0f$ when $s=1$, it is sufficient to know $\mathcal{R}_0f(t,\vy)$ for $\vy \in \left[ -(c+b),(c+b) \right]^{n-1}$, and for $\sqrt{t} \in [a,b]$. We have thus derived an inversion method for the spherical Radon transform with centers on a plane in $n$-dimensions, which uses only the sphere centers on a finite volume cube, and radii on the bound $[a,b]$. As discussed in the introduction, the formulae of \cite{p26} uses all sphere centers and radii, and in \cite{p28},  the authors present support theorems, but no inversion method. In \cite{p34}, the authors present an inversion formula for $\mathcal{R}_0$ in the $s=1$ case, which applies to distributions, $f$, and only uses limited sphere radii and centers. However, their inversion formulae only apply to the continuous part of $f$. Theorem \ref{trans_inv_thm} applies to $L^2_c$ functions. Thus, Theorem \ref{trans_inv_thm} derives important inversion methods for the spherical Radon transform, and for more general ellipsoid and hyperboloid transforms.
\end{remark}

Let $\mathcal{C}^+_j\paren{t,y_1} = \mathcal{C}_j\paren{t,y_1}$, and let $\mathcal{C}^-_j\paren{t,y_1}$ be the reflection of $\mathcal{C}^+_j\paren{t,y_1}$ in $\{x_2 = 0\}$. Let $\mathcal{S}^+_j(t,\vy) = \mathcal{S}_j(t,\vy)$ and $\mathcal{S}^-_j(t,\vy)$ be the reflections of one another in $\{x_n = 0\}$. Then, we define
\begin{equation}
\mathcal{E}_jf(t,y_1) = \int_{\mathcal{C}^-_j\paren{t,y_1}} f \mathrm{d}\mathcal{C}^-_j + \int_{\mathcal{C}^+_j\paren{t,y_1}} f \mathrm{d}\mathcal{C}^+_j,
\end{equation}
when $n = 2$, and
\begin{equation}
\mathcal{E}_jf(t,\vy) = \int_{\mathcal{S}^-_j(t,\vy)} f \mathrm{d}\mathcal{S}^-_j + \int_{\mathcal{S}^+_j(t,\vy)} f \mathrm{d}\mathcal{S}^+_j,
\end{equation}
when $n\geq 3$. When $s>1$, $\mathcal{E}_0f$ defines the integrals of $f$ over two-sided elliptic arcs of revolution. When $0< s\leq 1$, $\mathcal{E}_0f$ defines the integrals of $f$ over ellipsoids of revolution. $\mathcal{E}_1$ defines integrals over two-sheeted hyperboloids. We now have the following corollary.

\begin{corollary}
\label{corr_null}
The null space of $\mathcal{E}_j$, on domain $L^2_c\paren{\Delta_{a,b}}$, where $0<a<b<\infty$ and $\Delta_{a,b} = \{a\leq x_n \leq b\} \cup \{-b\leq x_n \leq -a\}$, is completed composed of odd functions in $x_n$
\end{corollary}
\begin{proof}
Let $f\in L^2_c\paren{\Delta_{a,b}}$ be such that $\mathcal{E}_jf = 0$. Let $f = f_1 + f_2$, where $f_i = f$ on $\{(-1)^ix_n<0\}$, and $f_i=0$ on $\mathbb{R}^n\backslash \{(-1)^ix_n<0\}$. Then, $\mathcal{E}_jf=\mathcal{R}_j\paren{f_1+\tilde{f_2}}=0$, where $\tilde{f_2}(\vx',x_n)=f_2(\vx',-x_n)$, and $\vx' = \paren{x_1,\ldots,x_{n-1}}$. Thus, by Theorem \ref{trans_inv_thm}, $f_1(\vx',x_n)=-f_2(\vx',-x_n)$, and $f(\vx',x_n)=-f_2(\vx',-x_n)+f_2(\vx',x_n)$ is odd in $x_n$. 
\end{proof}


\begin{discussion}
\label{dis_1}
In \cite{p13}, the authors analyze the stability of $\mathcal{E}_j$ from a microlocal perspective. They show that $\mathcal{E}_j$ is a Fourier Integral Operator (FIO) which satisfies the Bolker condition, if $f$ is supported on one side of $\{x_n = 0\}$. Thus, for such $f$, we would not expect to see any additional (unwanted) image singularities in the reconstruction.

Let us now break down some of the key instabilities in the inversion steps needed to solve the Volterra equations \eqref{volt} and \eqref{volty_2}, which are from the proof of Theorem \ref{trans_inv_thm}. When $n$ is even, we require $m = n/2$ differentiation steps and one Abel operation (which removes the singularity as in \eqref{weak_volt_1}) to solve \eqref{volt} and \eqref{volty_2}. The Abel operator is order $1/2$ smoothing on Sobolev scale, so in total this means that, roughly speaking, $(n-1)/2$ differentiation steps are needed to invert \eqref{volt} and \eqref{volty_2}, when $n$ is even. When $n$ is odd, no Abel operation is needed, as there is no singularity on the kernel diagonal, and, as in the even case, $(n-1)/2$ differentiation steps are needed to invert \eqref{volty_2}. The remaining inversion operations, i.e., Fourier transformation, change of variables, and inverting a second type Volterra equation with bounded kernel, do not require any further differentiation steps, and are more stable. As the dimension increases, so too does the order of smoothing and the ill-posedness, and the increase in smoothing appears to grow at the same rate as the classical hyperplane Radon transform, which is also order $(n-1)/2$ smoothing on Sobolev scale \cite[Theorem 5.1]{p31}. This discussion does not provide proof of this, however, and Sobolev space theory should be explored to address this. We leave this for future work as it is beyond the scope of this paper.
\end{discussion}

\subsection{A generalization}
\label{gen_section}
In this subsection, we present a generalization of the results of Theorem \ref{trans_inv_thm}. Let $f\in L^2_c\paren{ \{a\leq x_n\leq b\} }$, where $0\leq a<b$, Let $r = r(p,\omega) \in C\paren{T_0}$, where $T_0 = \{(p,\omega) : a\leq p\leq b, a\leq \omega \leq p\}$. Then, we define the generalized Radon transform, for $n\geq 2$,
\begin{equation}
\label{gen_radon}
\mathcal{R}f(p,\vy) = \int_{S^{n-2}}\int_{a}^p\sqrt{1+\paren{\frac{\mathrm{d}r}{\mathrm{d}\omega}}^2}r(p,\omega)^{n-2}f\paren{ r(p,\omega)\Theta+\vy,\omega }\mathrm{d}\omega\mathrm{d}S^{n-2}.
\end{equation}
When $n=2$, $\Theta = \pm 1$, and the integral $\int_{S^{n-2}} = \sum_{\Theta = \pm 1}$ becomes a sum over two points, and $\mathcal{R}f$ defines the integral of $f$ over the curve $\mathcal{C} = \{ \paren{r(p,\omega)+y_2,\omega} : \omega \in [0,p]\}$, and its reflection in the line $\{(y_2,x_2) : x_2\in \mathbb{R}\}$. When $n\geq 3$, $\mathcal{R}f$ defines the integrals of $f$ over the surfaces of revolution of $\mathcal{C}$, with central axis $\{(\vy,x_n) : x_n \in \mathbb{R}\}$.

We now have the following theorem, which is a generalization of Theorem \ref{trans_inv_thm}.

\begin{theorem}
\label{gen_thm}
Let $f\in L^2_c\paren{ \{a\leq x_n\leq b\} }$, where $0\leq a<b$. Let $r \in C\paren{T_0}$, be of the form
\begin{equation}
r(p,\omega) = (p-w)^{\frac{q}{2}} \cdot \nu(p,\omega),
\end{equation}
where $q \geq 1$, is an integer, $\nu \in C^{\infty}\paren{T_0}$, and $T_0 = \{(p,\omega) : a\leq p\leq b, a\leq \omega \leq p\}$. Further assume that $\nu(p,p) \neq 0$, for $p \in [a,b]$.  Then, \eqref{gen_radon} can be solved uniquely for $f$, if $\mathcal{R}f(p,\vy)$ is known for all $\vy \in \mathbb{R}^{n-1}$, and $p\in[a,b]$. 
\end{theorem}

\begin{remark}
Note, for bounded $\nu$, we only need the $\vy$ in some compact subset of $\mathbb{R}^{n-1}$, as $f$ is compactly supported. 
\end{remark}

\begin{proof}
Taking the Fourier transform in $\vy$ on both sides of \eqref{gen_radon} yields
\begin{equation}
\label{gen_anti}
\begin{split}
\widehat{\mathcal{R}f}(p,\vxi) &= \int_{a}^p (p-\omega)^{\alpha}\left[\prod_{i=1}^3 K_i(p,\omega)\right] \cdot  \hat{f}(\vxi,\omega)\mathrm{d}\omega \\
& = \int_{a}^p (p-\omega)^{\alpha} K(p,\omega) \cdot  \hat{f}(\vxi,\omega)\mathrm{d}\omega,
\end{split}
\end{equation}
a set of generalized Abel equations. 

Let $\beta = q/2$. We define 
$$
\begin{cases}
   \alpha = (n-3)/2 \geq -1/2, & \text{if $q=1$}.\\
    \alpha = \beta(n-2) \geq 0, & \text{if $q\geq2$}.
  \end{cases}
$$
As $\alpha > -1$, we aim to apply Corollary \ref{corr2} to solve \eqref{gen_anti}. The kernels $K_i$ are defined
\begin{equation}
K_1(p,\omega)=\begin{cases}
    2 \cos \paren{ |\vxi| (p-\omega)^{\beta} \nu(p,\omega) }, & \text{if $n=2$}.\\
    \alpha_n\int_{0}^{\pi} \cos\paren{  |\vxi| (p-\omega)^{\beta} \nu(p,\omega) \cos \varphi}\sin^{n-3}\varphi\mathrm{d}\varphi, & \text{if $n\geq 3$},
  \end{cases}
\end{equation}
where the $\alpha_n$ are as defined in Theorem \ref{trans_inv_thm}, $K_2(p,\omega) = \nu(p,\omega)^{n-2}$, and 
\begin{equation}
K_3(p,\omega)=\begin{cases}
    \paren{ (p-\omega) + \paren{ (p-\omega)\frac{\mathrm{d}\nu}{\mathrm{d}\omega}(p,\omega) - \frac{1}{2} \nu(p,\omega) }^2 }^{\frac{1}{2}}, & \text{if $q=1$}.\\
    \paren{ 1 + (p-\omega)^{q-2}\paren{ (p-\omega)\frac{\mathrm{d}\nu}{\mathrm{d}\omega}(p,\omega) - \beta \nu(p,\omega) }^2 }^{\frac{1}{2}}, & \text{if $q\geq2$}.
  \end{cases}
\end{equation}
Let 
$$g(p,\omega) = (p-\omega) + \paren{ (p-\omega)\frac{\mathrm{d}\nu}{\mathrm{d}\omega}(p,\omega) - \frac{1}{2} \nu(p,\omega) }^2 \in C^{\infty}(T_0).$$
When $p=\omega$, $g(p,p) = \frac{1}{4}\nu(p,p)^2 >0$, for $p \in[a,b]$, by assumption. When $p>\omega$ and $(p,\omega) \in T_0$, then clearly $g(p,\omega)>0$. Thus, $g>0$ everywhere on $T_0$, and must also be bounded away from zero on $T_0$, by smoothness of $g$ and compactness of $T_0$. Therefore, $K_3(p,\omega) = \sqrt{g(p,\omega)} \in C^{\infty}\paren{T_0}$, when $q=1$. If $q\geq 2$, then $K_3$ is clearly smooth on $T_0$. Further, $K_2 \in C^{\infty}\paren{T_0}$, since $\nu \in C^{\infty}\paren{T_0}$, by assumption.

Now, we have
\begin{equation}
K(p,p)=\begin{cases}
   \left|\nu(p,p)\right|, & \text{if $q=1$ and $n=2$}.\\
    \frac{\alpha_n}{2}\paren{\int_0^{\pi}\sin^{n-3}\varphi \mathrm{d}\varphi} \cdot \left|\nu(p,p)\right| \cdot \nu(p,p)^{n-2}, & \text{if $q=1$ and $n\geq 3$}. \\
  2\mu(p), & \text{if $q\geq 2$ and $n=2$}.\\
\alpha_n\paren{\int_0^{\pi}\sin^{n-3}\varphi \mathrm{d}\varphi} \cdot \mu(p)  \cdot \nu(p,p)^{n-2}, & \text{if $q\geq 2$ and $n\geq 3$},
  \end{cases}
\end{equation}
where 
\begin{equation}
\mu(p) =\begin{cases}
    \sqrt{1+\paren{\beta\nu(p,p)}^2}, & \text{if $q=2$}.\\
    1 , & \text{if $q\geq 3$}.
  \end{cases}
\end{equation}
Therefore, $K(p,p) \neq 0$, for $p\in [a,b]$, by assumption that $\nu(p,p)\neq 0$, for $p \in [a,b]$. It remains to be shown that $K_1$ satisfies the conditions of Corollary \ref{corr2}. 

Let $g_1(p,\omega) = |\vxi| (p-\omega)^{\beta} \nu(p,\omega) \in C(T_0)$. Let $n=2$. Then, we have
\begin{equation}
\frac{\mathrm{d}K_1}{\mathrm{d}p}(p,\omega) = 2 g_2(p,\omega) \sinc\paren{ g_1(p,\omega)} \in C(T_0),
\end{equation}
where
\begin{equation}
\begin{split}
g_2(p,\omega)  &= - g_1(p,\omega) \cdot \frac{\mathrm{d}g_1}{\mathrm{d} p}(p,\omega)\\
&= - |\vxi|^2(p-\omega)^{2\beta - 1}\nu(p,\omega) \left[ \beta \nu(p,\omega) + (p-\omega) \frac{\mathrm{d}\nu}{\mathrm{d} p}(p,\omega) \right],
\end{split}
\end{equation}
which is smooth on $T_0$. We can now apply Theorem \ref{weak_volt_thm} to prove the result in the $n=2$ case.

Let  $n\geq 3$. Then,
\begin{equation}
\frac{\mathrm{d}K_1}{\mathrm{d}p}(p,\omega) = \alpha_n\cdot g_2(p,\omega) \int_{0}^{\pi}\sinc\paren{ g_1(p,\omega) \cos\varphi} \sigma(\varphi) \mathrm{d}\varphi \in C(T_0),
\end{equation}
where $\sigma(\varphi) = \cos^2(\varphi)\sin^{n-3}(\varphi)$. Let $J \in C^{\infty}(\mathbb{R})$ be defined
$$J(u) = \int_{0}^{\pi}\sinc\paren{ u \cos\varphi} \sigma(\varphi) \mathrm{d}\varphi.$$
To finish the proof, we aim to show 
\begin{equation}
\label{hk}
h_k(p,\omega) = (p-\omega)^k \frac{\mathrm{d}^k}{\mathrm{d}p^k}J \paren{ g_1(p,\omega) }\in C(T_0),
\end{equation}
for any $k \geq 1$. When $q$ is even, \eqref{hk} clearly holds. We now focus on the case when $q$ is odd. Using Theorem \ref{di_bruno_thm} and Theorem \ref{general_leibniz}, we have, for $k\geq1$,
\begin{equation}
\label{cont}
\begin{split}
h_k(p,\omega) &= (p-\omega)^k \cdot \sum_{\gamma\in\Pi}J^{(|\gamma|)}\paren{ g_1(p,\omega) }\cdot \prod_{\tau \in\gamma}g^{(|\tau|)}_1(p,\omega) \\
& = \sum_{\gamma\in\Pi}J^{(|\gamma|)}\paren{ g_1(p,\omega) }\cdot (p-\omega)^k \prod_{\tau \in\gamma}\left[\sum_{i=0}^{|\tau|}c_{i,|\tau|}(p-\omega)^{\beta-(|\tau|-i)}\frac{\mathrm{d}^i\nu}{\mathrm{d}p^i}(p,\omega) \right] \\
& = \sum_{\gamma\in\Pi}J^{(|\gamma|)}\paren{ g_1(p,\omega) }\cdot \prod_{\tau \in\gamma}(p-\omega)^{|\tau|}\left[\sum_{i=0}^{|\tau|}c_{i,|\tau|}(p-\omega)^{\beta-(|\tau|-i)}\frac{\mathrm{d}^i\nu}{\mathrm{d}p^i}(p,\omega) \right] \\
& =  \sum_{\gamma\in\Pi}J^{(|\gamma|)}\paren{g_1(p,\omega) }\cdot \prod_{\tau \in\gamma}\left[\sum_{i=0}^{|\tau|}c_{i,|\tau|}(p-\omega)^{(\beta+i)}\frac{\mathrm{d}^i\nu}{\mathrm{d}p^i}(p,\omega) \right] ,
\end{split}
\end{equation}
where $\Pi$ is the set of partitions of $\{1,\ldots,k\}$, $J^{(|\gamma|)}$ denotes derivative $|\gamma|$ of $J$, $g_1^{(|\tau|)}$ denotes derivative $|\tau|$ of $g_1$, $c_{|\tau|,|\tau|} = |\vxi|$, and
\begin{equation}
c_{i,|\tau|}=|\vxi|\cdot \begin{pmatrix} |\tau| \\ i \end{pmatrix} \paren{\prod_{l=0}^{|\tau|-1-i}(\beta - l)},
\end{equation}
for $0\leq i \leq |\tau| - 1$. On the third line of \eqref{cont}, we use the fact that 
$$k = \sum_ {\tau\in\gamma} |\tau| \implies (p-\omega)^k = \prod_{\tau\in\gamma}(p-\omega)^{|\tau|}.$$
In the exponent in the sum on the fourth line of \eqref{cont}, $\beta+i \geq 1/2,$
since $i\geq 0$, and $\beta \geq 1/2$. Thus, $h_k \in C(T_0)$, as desired, for all $k\geq 1$. This finishes the proof.
\end{proof}

If we set $q=1$, and $\nu(p,\omega) = \sqrt{s} \sqrt{p+\omega}$, then $\mathcal{R}f = \mathcal{R}_0f$, and thus Theorem \ref{gen_thm} is a direct generalization of the results of Theorem \ref{trans_inv_thm}. In figure \ref{fig_surfaces}, we give some additional examples of curves and surfaces in 2-D and 3-D which satisfy the conditions of Theorem \ref{gen_thm}.
\begin{figure}[!h]
\centering
\begin{subfigure}{0.24\textwidth}
\includegraphics[width=0.9\linewidth, height=3.2cm, keepaspectratio]{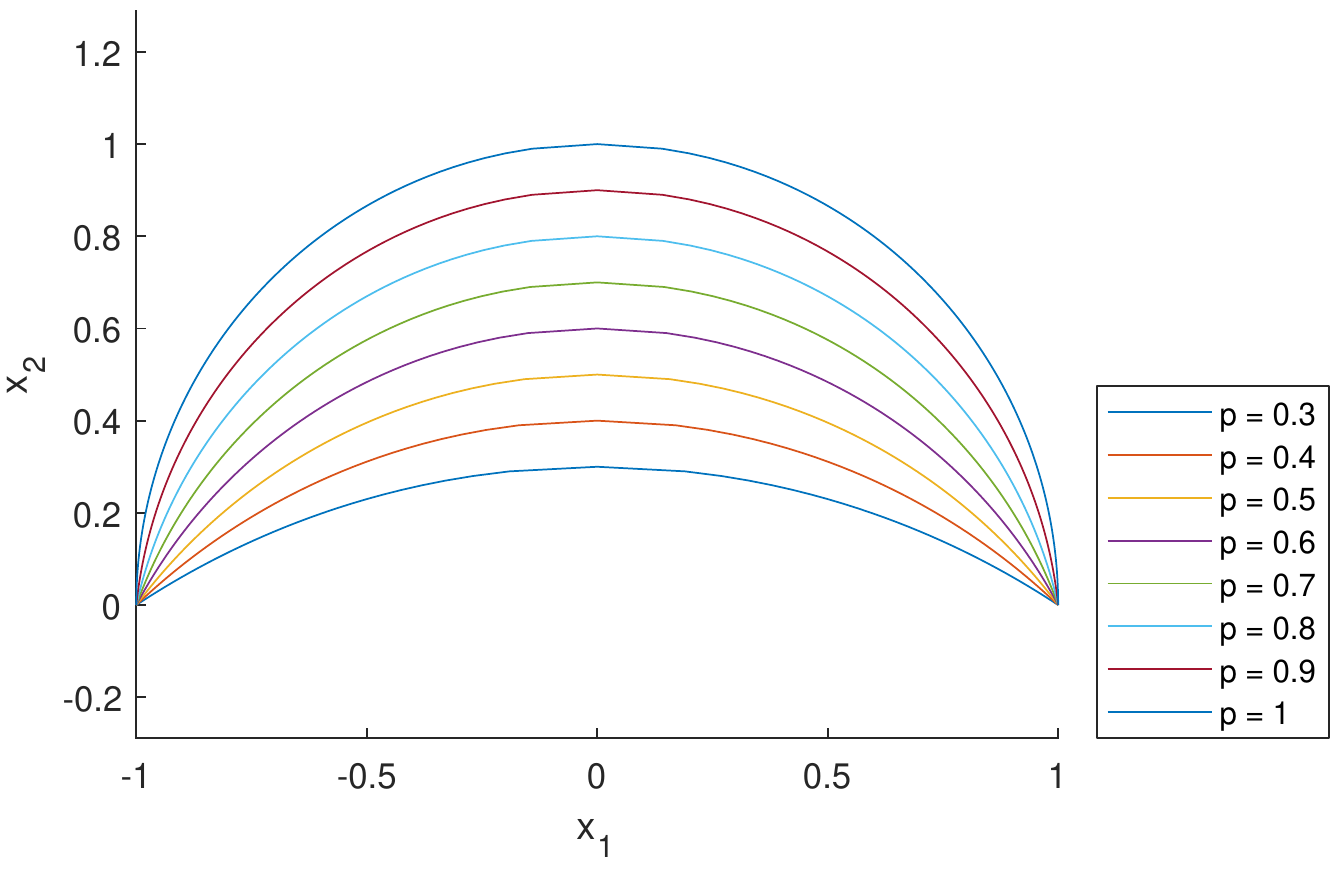}
\subcaption{\scalebox{0.65}{  $r = \sqrt{p-\omega}\sqrt{\frac{p\omega+1}{p}}$ }} \label{s1F}
\end{subfigure}
\begin{subfigure}{0.24\textwidth}
\includegraphics[width=0.9\linewidth, height=3.2cm, keepaspectratio]{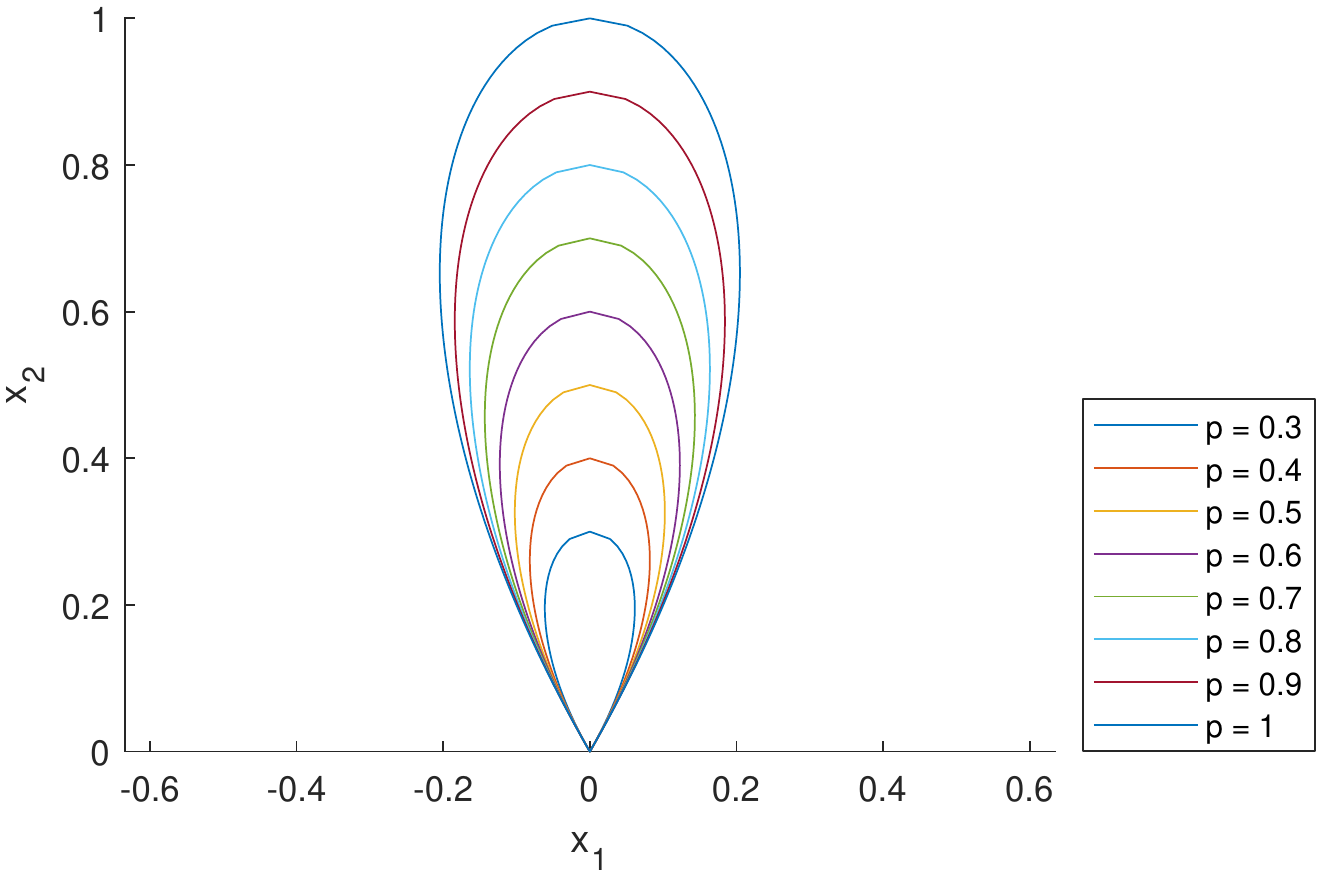}
\subcaption{\scalebox{0.6}{ $r = \sqrt{\frac{3}{4}}\sqrt{p^2-\omega^2} - \sqrt{p-\omega}\sqrt{\frac{3}{4}p-\frac{1}{4}\omega}$ }} \label{tear12D}
\end{subfigure}
\begin{subfigure}{0.24\textwidth}
\includegraphics[width=0.9\linewidth, height=3.2cm, keepaspectratio]{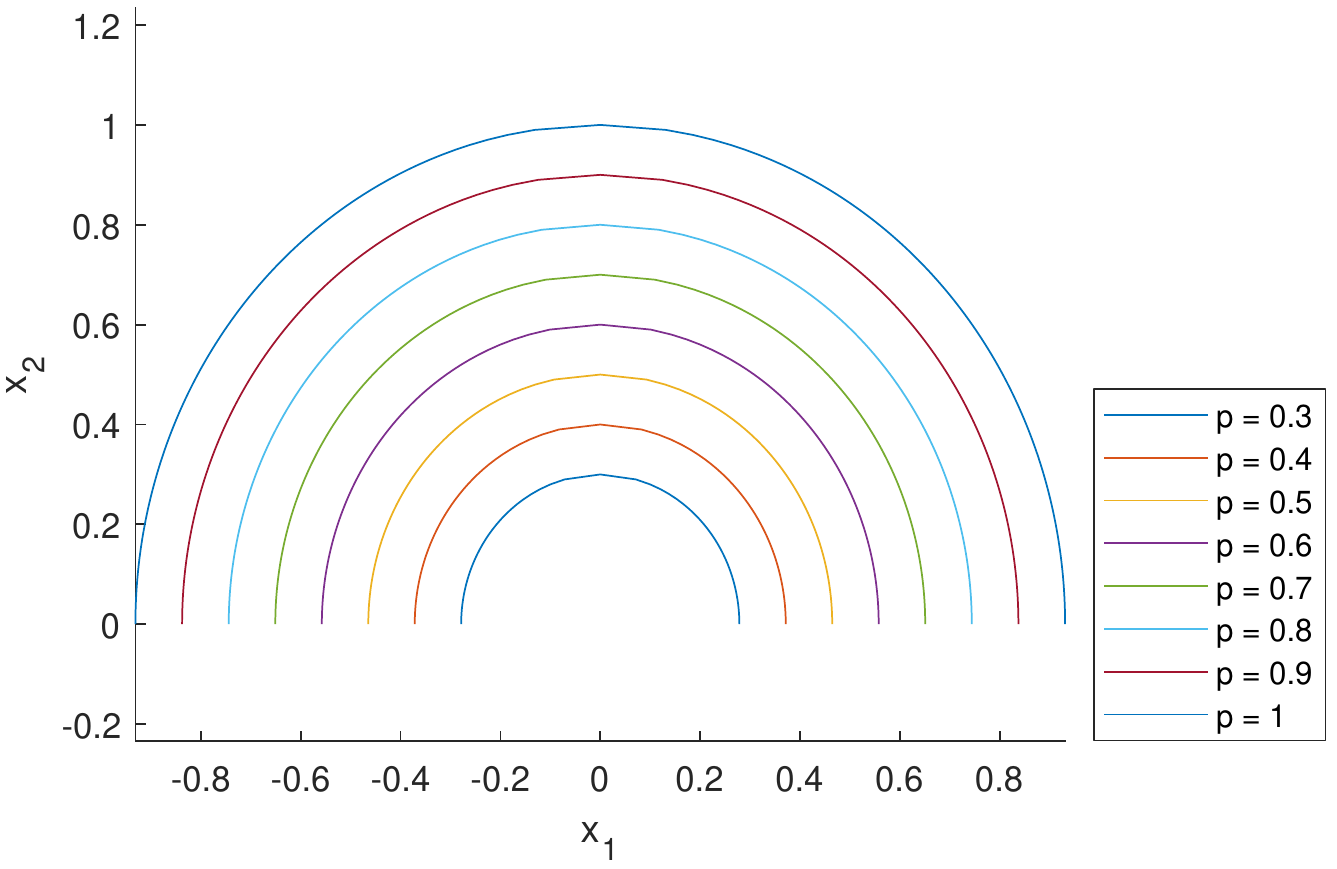}
\subcaption{ \scalebox{0.65}{ $r = \sqrt{p^2-\omega^2}\sqrt{\frac{p^2+h^2}{p^2+h^2+d^2}}$ } } \label{gauss2D}
\end{subfigure}
\begin{subfigure}{0.24\textwidth}
\includegraphics[width=0.9\linewidth, height=3.2cm, keepaspectratio]{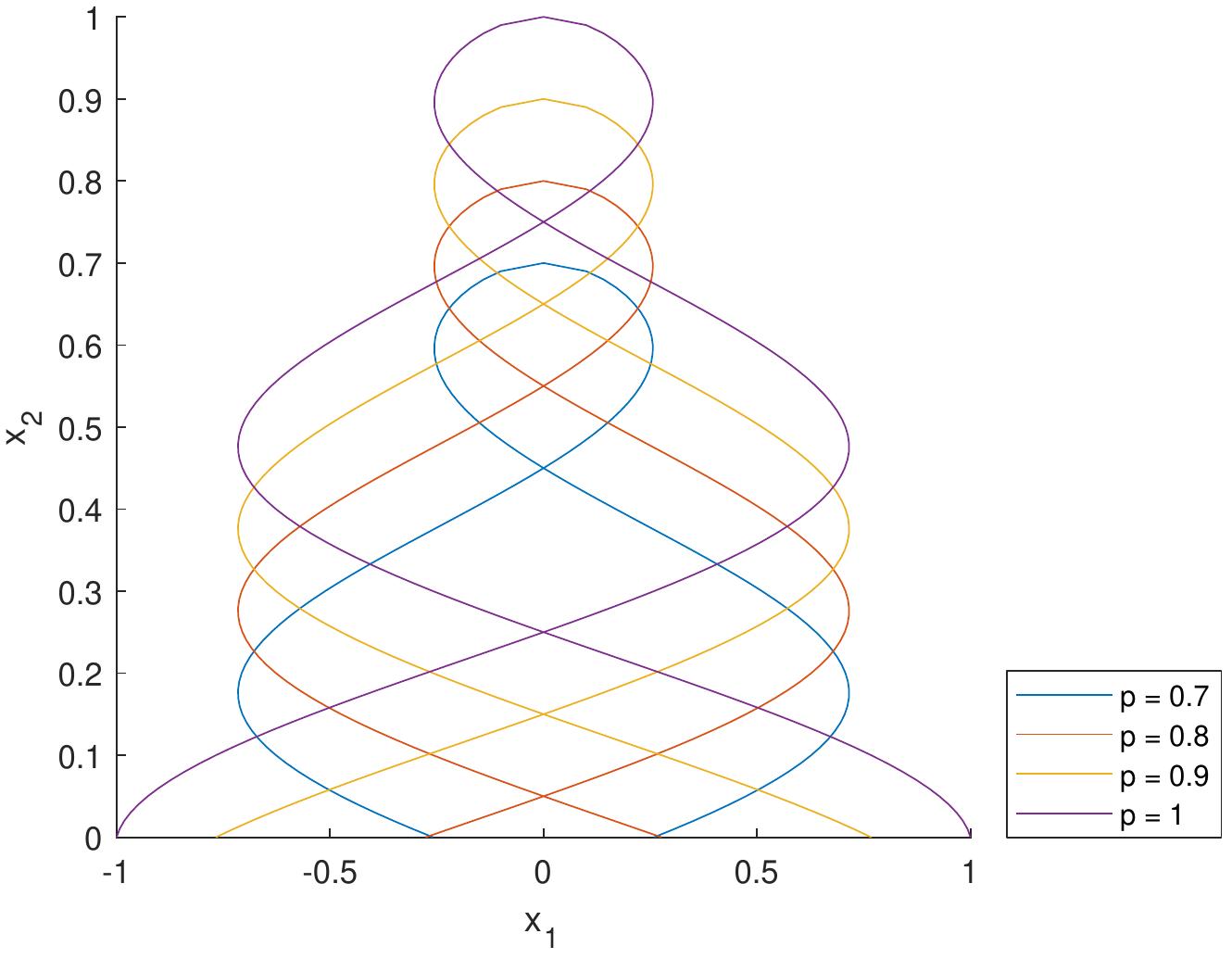}
\subcaption{ \scalebox{0.65}{ $r = \sqrt{p-\omega}\cos\paren{2\pi(p-\omega)}$ } } \label{cos2D}
\end{subfigure}
\begin{subfigure}{0.24\textwidth}
\includegraphics[width=0.9\linewidth, height=3.2cm, keepaspectratio]{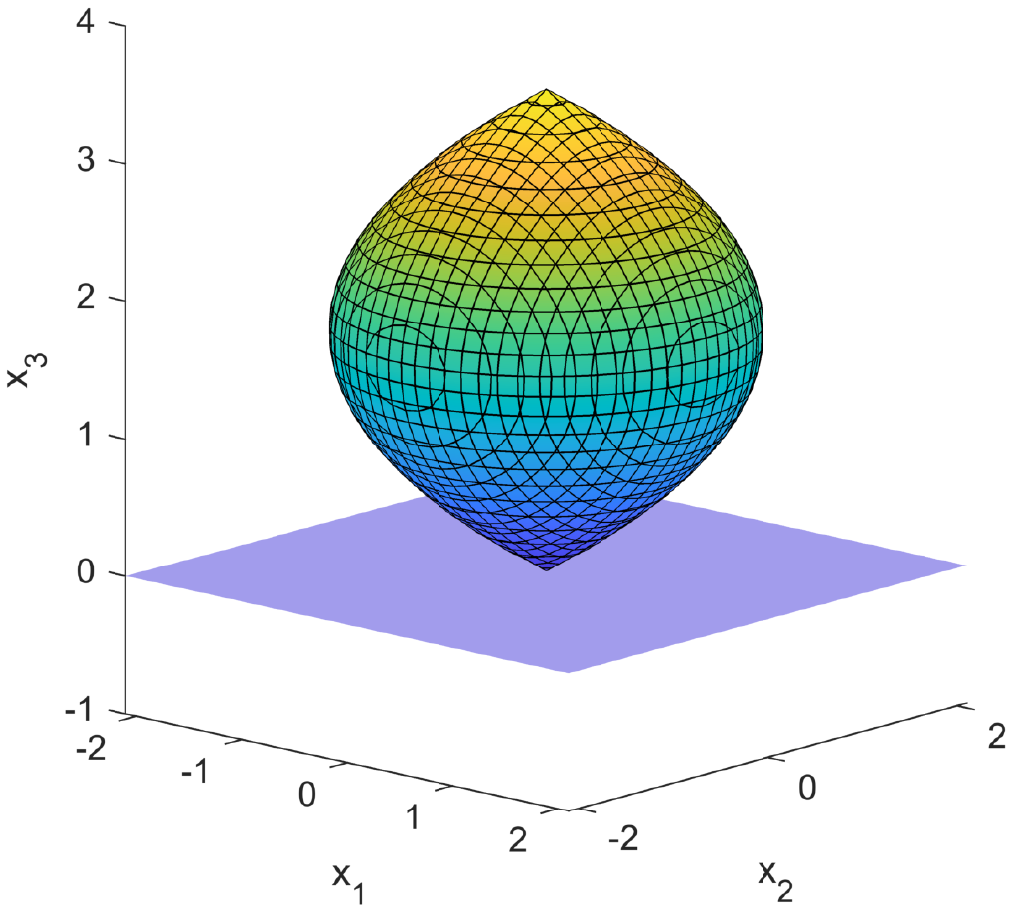}
\subcaption{ \scalebox{0.65}{ $r = p\omega (p-\omega)$ } }
\end{subfigure}
\begin{subfigure}{0.24\textwidth}
\includegraphics[width=0.9\linewidth, height=3.2cm, keepaspectratio]{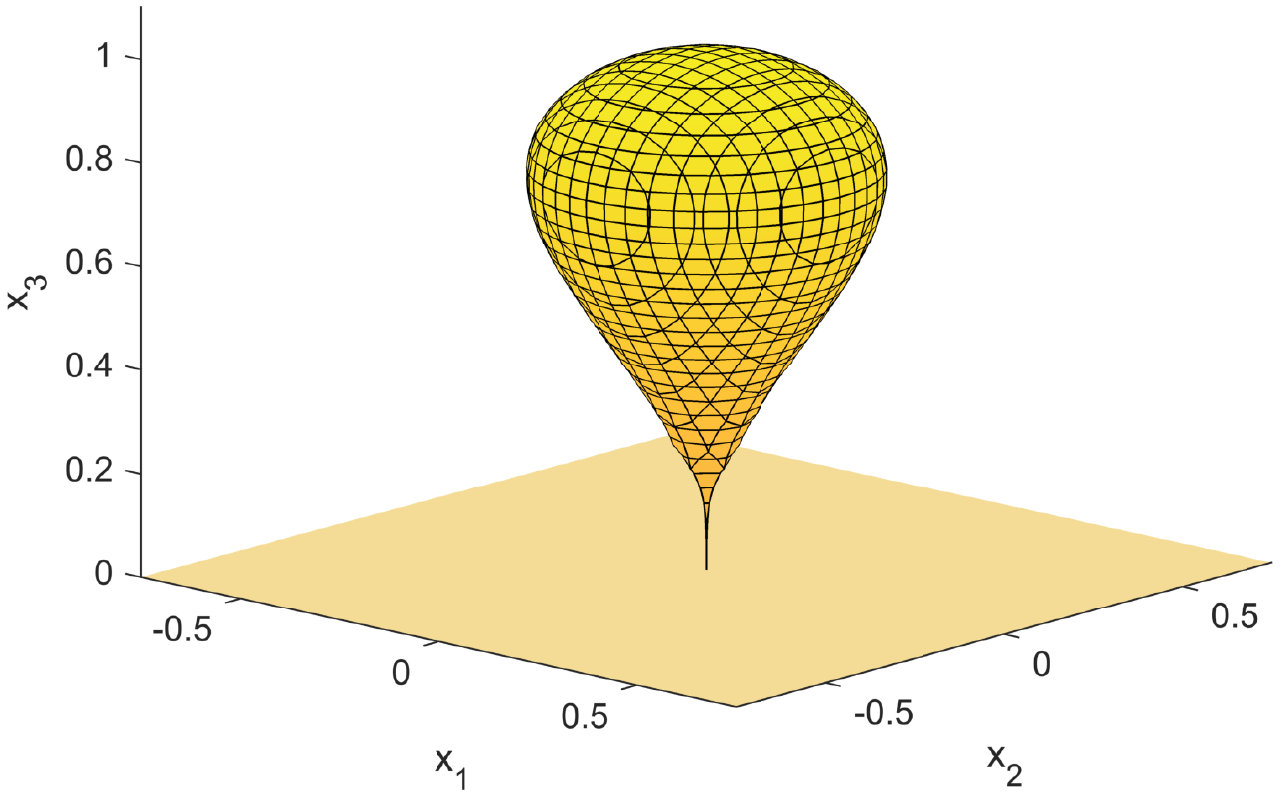}
\subcaption{ \scalebox{0.65}{ $r = \omega \sinc(p-\omega) \sqrt{p-\omega}$ } } \label{tear3D}
\end{subfigure}
\begin{subfigure}{0.24\textwidth}
\includegraphics[width=0.9\linewidth, height=3.2cm, keepaspectratio]{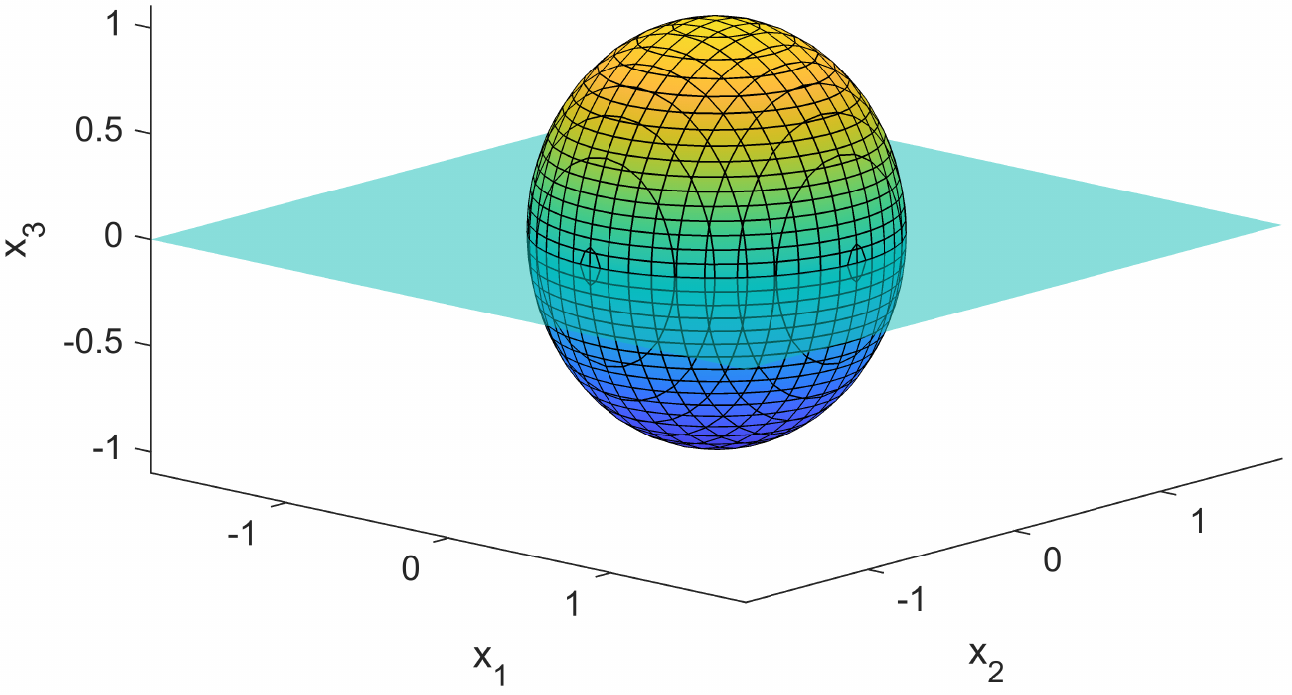}
\subcaption{ \scalebox{0.65}{ $r = \sqrt{\frac{p^2-c^2}{p^2}} \cdot \sqrt{p^2-\omega^2}$ }} \label{gauss3D}
\end{subfigure}
\begin{subfigure}{0.24\textwidth}
\includegraphics[width=0.9\linewidth, height=3.2cm, keepaspectratio]{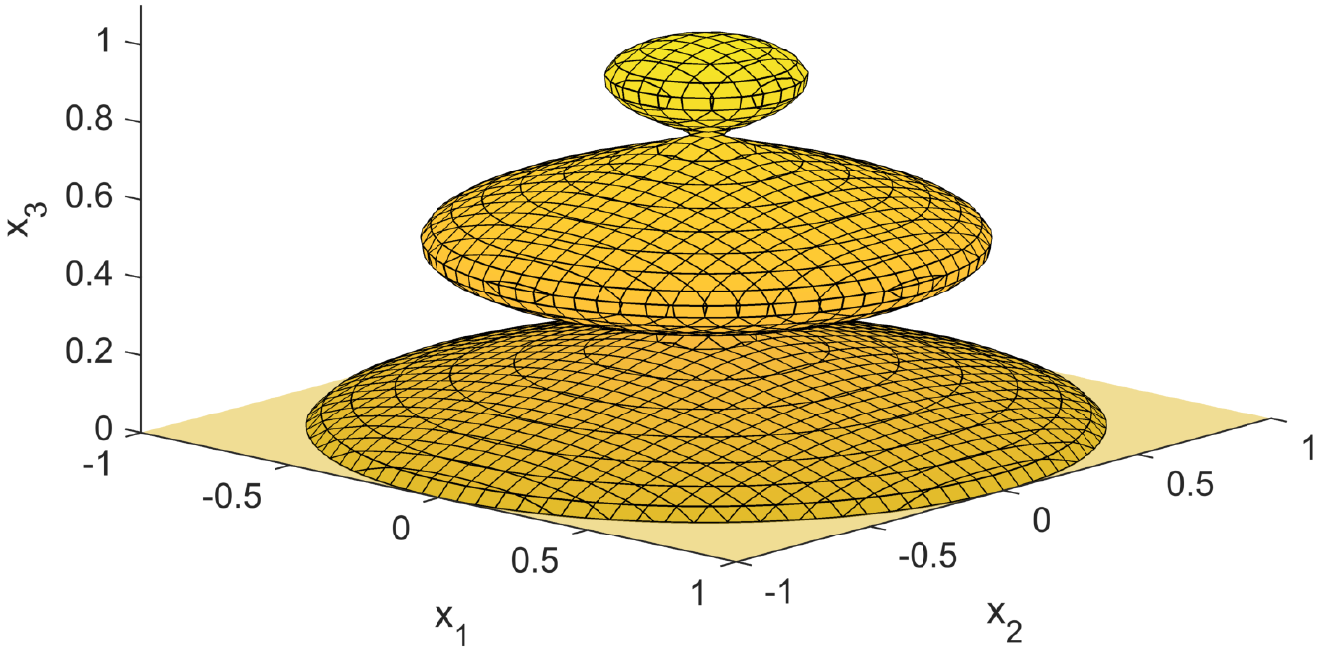}
\subcaption{\scalebox{0.65}{ $r = \sqrt{p-\omega} \cos\paren{2\pi(p-\omega)}$ }} \label{h13}
\end{subfigure}
\caption{Example curves and surfaces which satisfy the conditions of Theorem \ref{gen_thm}. In each case, the expression for $r = r(p,\omega)$ is given in the subfigure caption. In figures (A), (B), (C), (E), and (F), $a$ has to be stricly greater than zero in order for Theorem \ref{gen_thm} to hold. In figure (G), $a \geq c$, where $c = 0.5$ is the fixed linear eccentricity of the spheroid, with foci $(0,0,\pm c)$. In the remaining figures, $a\geq 0$. In (C), we plot the elliptic curves of \cite[equation (4)]{p1}.}
\label{fig_surfaces}
\end{figure}

The circular arc curves of figure \ref{s1F} are of particular interest in CST \cite{p35}. The elliptic arc curves of figure \ref{gauss2D} have applications in multistatic Synthetic Aperture Radar (SAR) \cite{p1}. The variables $h$ and $d$ in the figure \ref{gauss2D} caption correspond to height and base distance from the origin of transmitter and receiver platforms in multistatic SAR. See \cite[figure 1]{p1}. We set $h=5$, and $d=2$, as in the examples of \cite[figure 3]{p1}. The authors in \cite{p1} present an approximate inversion method for a Radon transform, $R_x$, which integrates $f$ over ellipses as in figure \ref{gauss2D} with centers on a line, but injectivity of $R_x$ is not proven. Theorem \eqref{gen_thm} proves injectivity for $R_x$, for $f \in L^2_c\paren{ \{a\leq x_2 \leq b\} }$, $0<a<b$, after setting $q=1$, and $\nu(p,\omega) = \sqrt{p+\omega}\sqrt{\frac{p^2+h^2}{p^2+h^2+d^2}}$ in Theorem \ref{gen_thm}. 

When $r = \sqrt{\frac{p^2-c^2}{p^2}} \cdot \sqrt{p^2-\omega^2}$, as in figure \ref{gauss3D}, the integral surfaces are spheroids with foci on $\{x_n = \pm c\}$. In this case, we set $q=1$, and $\nu(p,\omega) = \sqrt{p+\omega}\sqrt{\frac{p^2-c^2}{p^2}}$ in Theorem \ref{gen_thm}. Then,  $f \in L^2_c\paren{ \{a\leq x_2 \leq b\} }$, with $c<a<b$, can be recovered uniquely if $\mathcal{R}f(p,\vy)$ for $p \in [a,b]$ and $\vy \in \mathbb{R}^{n-1}$ is known. This example has applications in URT. For example, we could place a sound wave emitter and receiver at $(0,0\pm c)$, in $n=3$ dimensions, and image $f \in L^2_c\paren{ \{a\leq x_2 \leq b\} }$. In this case, translating $f$ parallel to the $(x_1,x_2)$ plane and varying the major spheroid axis, $p>c$, determines $\mathcal{R}f(p,\vy)$. In the context of URT, the time delay between sending and receiving signals determines $p$. 

We can also recover $f$ uniquely from its integrals over teardrop curves and surfaces as in figures \ref{tear12D}, and \ref{tear3D}, and from integrals over oscillating curves as in figure \ref{cos2D}. 

We have discussed in this section several key applications of Theorem \ref{anti_volt_thm} to translation invariant Radon transforms, and discussed practical applications in the context of URT, SAR, and CST. In appendix \ref{appA}, we consider additional applications of Theorem \ref{anti_volt_thm} to rotation invariant Radon transforms.

\section{Image reconstructions}
\label{results}
In this section, we present image reconstructions from $\mathcal{E}_jf$ data in $n=2$ dimensions, where $f$ is a simulated image phantom. To reconstruct $f$, we use algebraic methods. Specifically, let $A_j$ denote the discretized form of $\mathcal{E}_j$, let $\vu$ be the discretized and vectorized form of $f$ on a uniform grid, and let $\vb$ denote the measured data. Then, we aim to find
\begin{equation}
\label{obj}
\argmin_{\vx \in \mathcal{X}} \|A\vx - \vb\|_2^2 + \lambda G(\vx),
\end{equation}
where $G$ is a regularization function (e.g., Total Variation (TV) or Tikhonov), and $\mathcal{X}$ is our solution space (e.g., non-negative functions).

\subsection{Data simulation}
\label{data_sim}
We discretize $f$ on an $m\times m$ uniform grid, with $m$ odd,
$$\left\{(x_1,x_2) : x_1 = -m/2 + \frac{m}{m-1}(i-1), x_2 = -m/2 + \frac{m}{m-1}(j-1), 1\leq i,j \leq m-1\right\},$$
and $\mathcal{E}_jf(t,y_1)$ is sampled for $y_1 \in [-m,m]$ with $2m$ steps at even intervals, and for $t = sp^2$, where $p \in \{1,\ldots \frac{m+1}{2} \}$, and $s = 2$ is fixed. We choose $s=2$ so the ellipse foci lie on $\{x_2=0\}$, and this is motivated by applications in URT. For example, the ellipse foci could represent sound wave emitters and receivers.

We simulate data with noise in the following way
\begin{equation}
\vb = A_{\epsilon}\vx + \gamma \times \frac{\|A_{\epsilon}\vx\|_2}{\sqrt{l}} \mathbf{\eta},
\end{equation}
where $l$ is the length of $A_{\epsilon}\vx$, $\mathbf{\eta} \sim \mathcal{N}(0,1)$ is a set of draws from a standard Gaussian, and $\gamma$ is a parameter which controls the noise level. $A_{\epsilon}$ is a perturbed $A$, which is used to avoid inverse crime. Specifically, $A_{\epsilon}$ is generated by multiplying every non-zero element of $A$ by $u \sim U(1-\epsilon,1+\epsilon)$, where $U(1-\epsilon,1+\epsilon)$ is the uniform distribution on $[1-\epsilon,1+\epsilon]$. Thus, we perturb the weights of $A$ uniformly and at random, and, on top of that, we add additional Gaussian noise to the perturbed data.

Throughout this section, we fix $m = 257$, $\epsilon = 0.05$, and we vary $\gamma \in [0.01,0.05]$, i.e., between $1 \%$ and $5 \%$ added Gaussian noise. Note, the $(p,y_1)$ sampling described above is sufficient to satisfy the conditions of Theorem \ref{trans_inv_thm}, and we can recover $f$, with compact support on $[-m/2,m/2]\times [0,m/2]$, uniquely from $\mathcal{E}_jf$.

\subsection{Image phantoms}
We consider the image phantoms displayed in figure \ref{F_phan}, one of which is a characteristic function on an annulus, and the other is the sum of 20 characteristic functions on ellipse interiors. The centers and radii of the ellipses are selected at random. Both phantoms are supported in the upper half space $\{x_2 > 0\}$, in line with Theorem \ref{trans_inv_thm}. 
\begin{figure}[!h]
\centering
\begin{subfigure}{0.24\textwidth}
\includegraphics[width=0.9\linewidth, height=3.2cm, keepaspectratio]{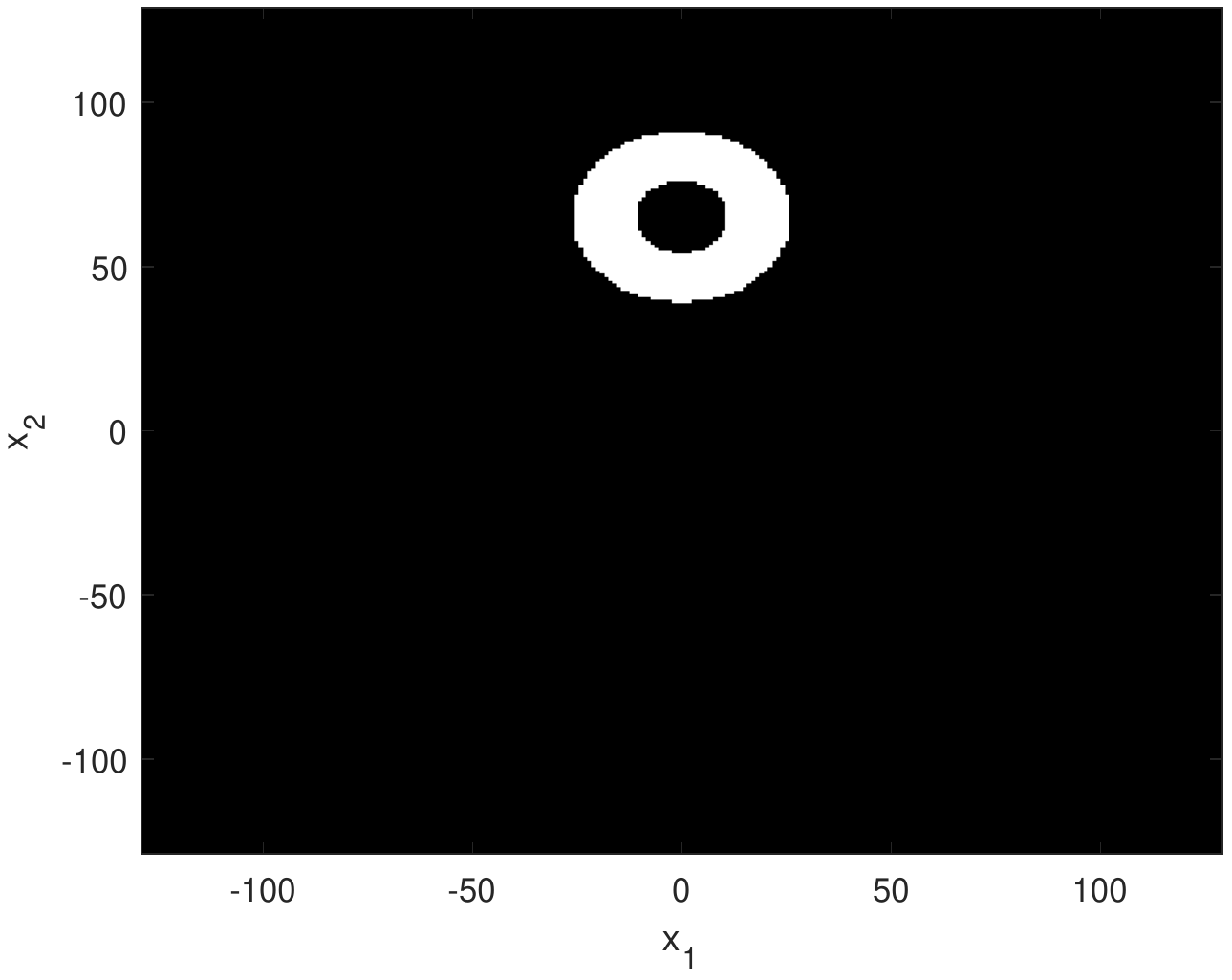}
\subcaption{annulus phantom} \label{F7a}
\end{subfigure}
\begin{subfigure}{0.24\textwidth}
\includegraphics[width=0.9\linewidth, height=3.2cm, keepaspectratio]{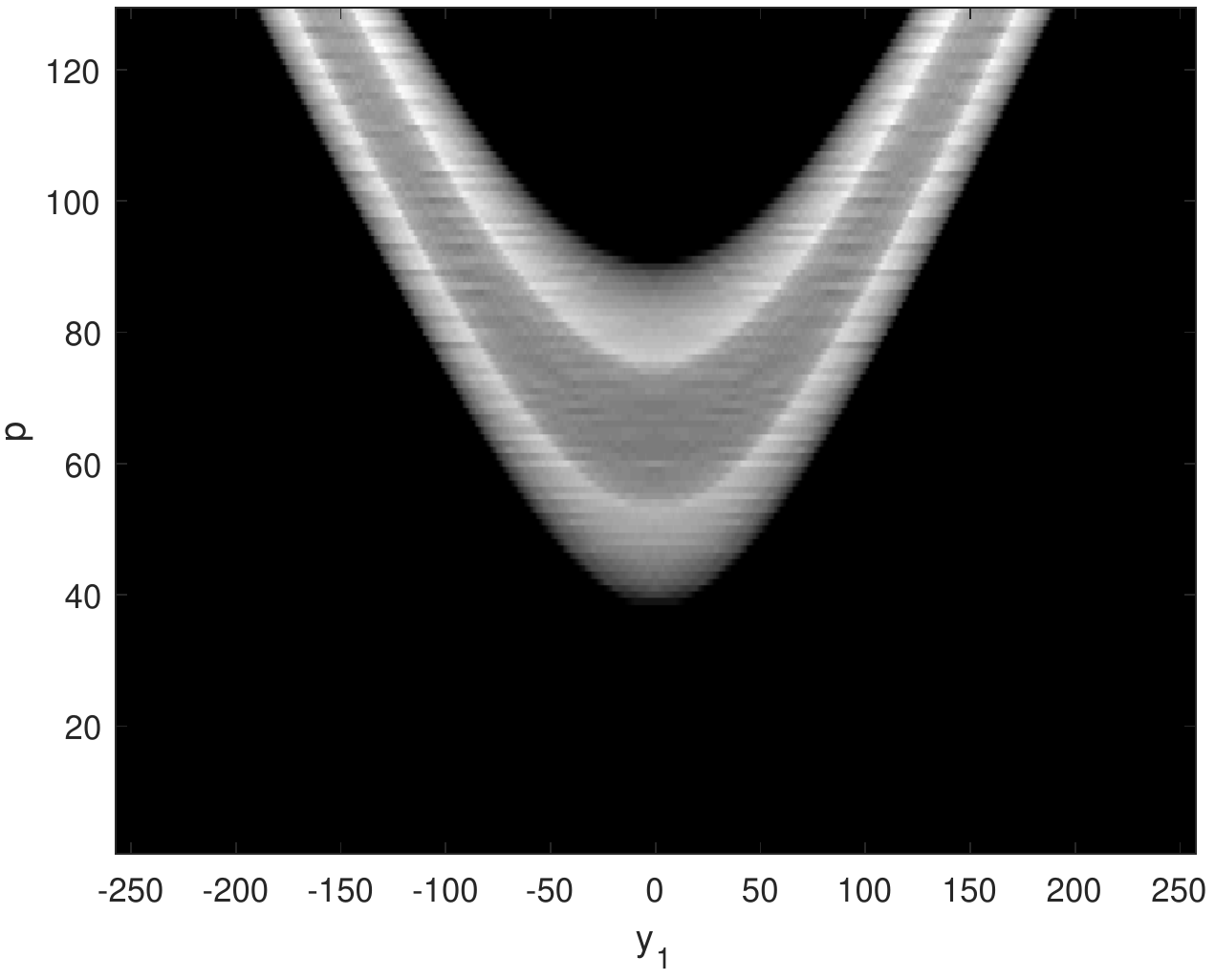} 
\subcaption{$\mathcal{E}_0 f$} \label{F7b}
\end{subfigure}
\begin{subfigure}{0.24\textwidth}
\includegraphics[width=0.9\linewidth, height=3.2cm, keepaspectratio]{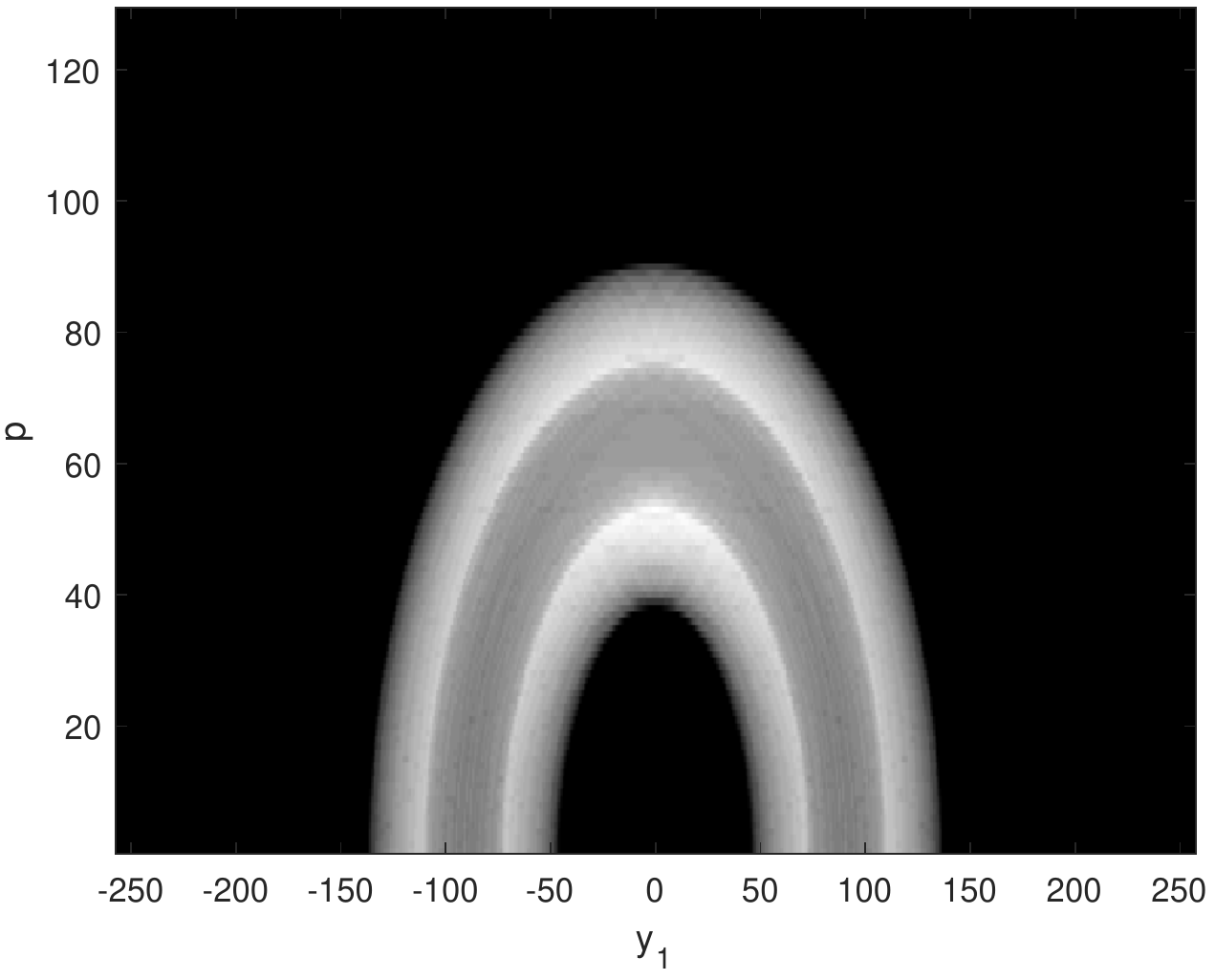}
\subcaption{$\mathcal{E}_1 f$} \label{F7c}
\end{subfigure}
\\
\begin{subfigure}{0.24\textwidth}
\includegraphics[width=0.9\linewidth, height=3.2cm, keepaspectratio]{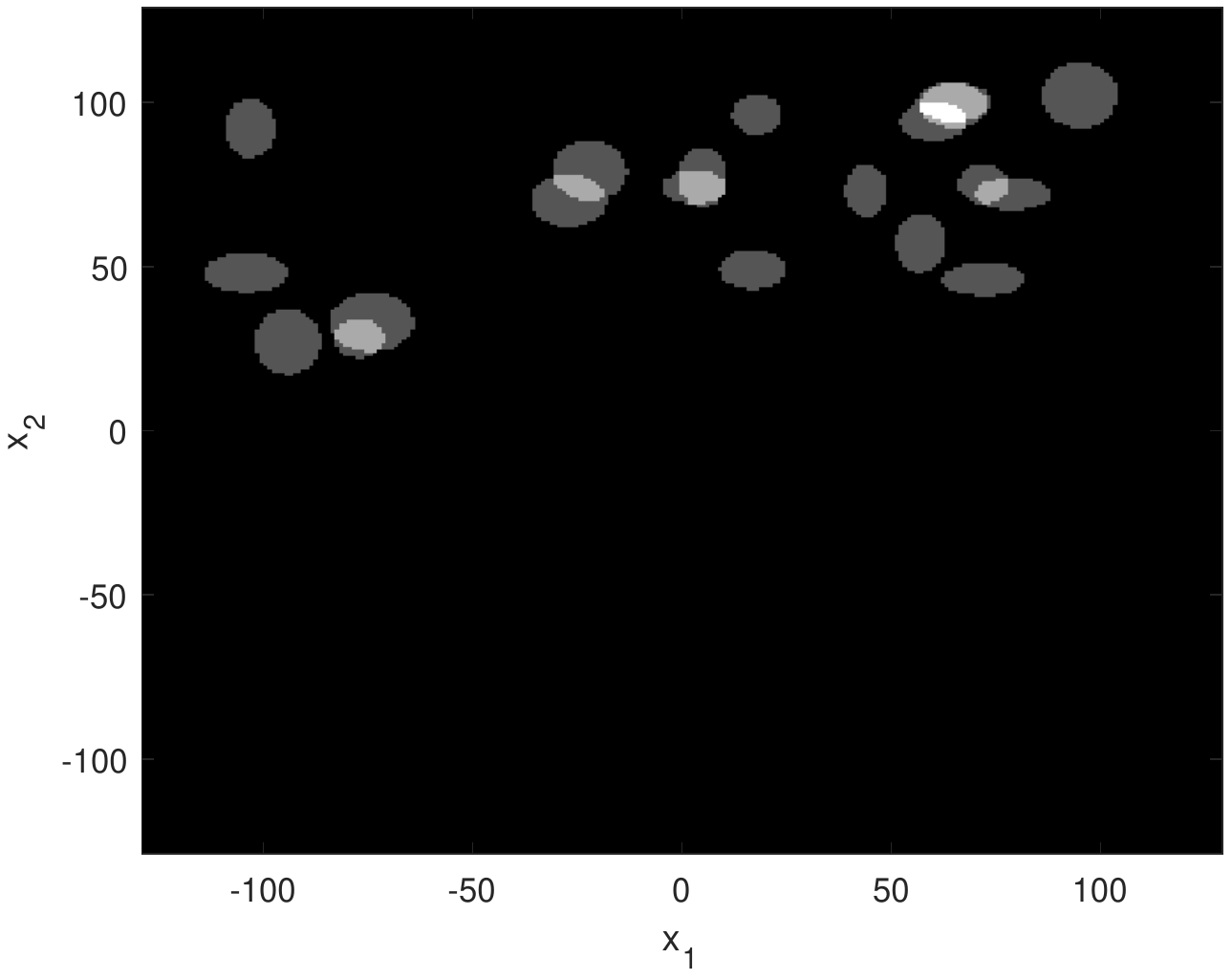}
\subcaption{ellipse phantom} \label{F7d}
\end{subfigure}
\begin{subfigure}{0.24\textwidth}
\includegraphics[width=0.9\linewidth, height=3.2cm, keepaspectratio]{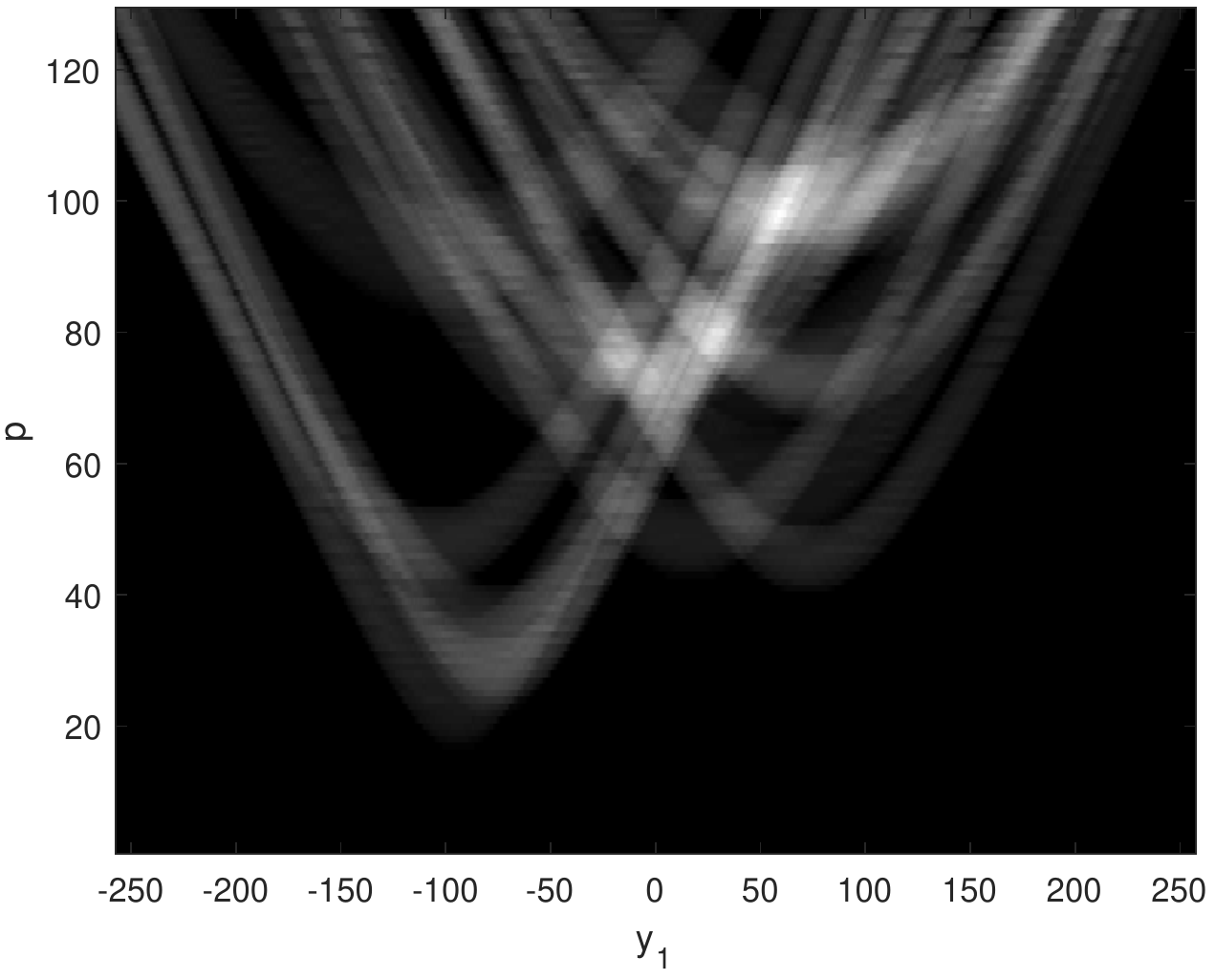} 
\subcaption{$\mathcal{E}_0 f$} \label{F7e}
\end{subfigure}
\begin{subfigure}{0.24\textwidth}
\includegraphics[width=0.9\linewidth, height=3.2cm, keepaspectratio]{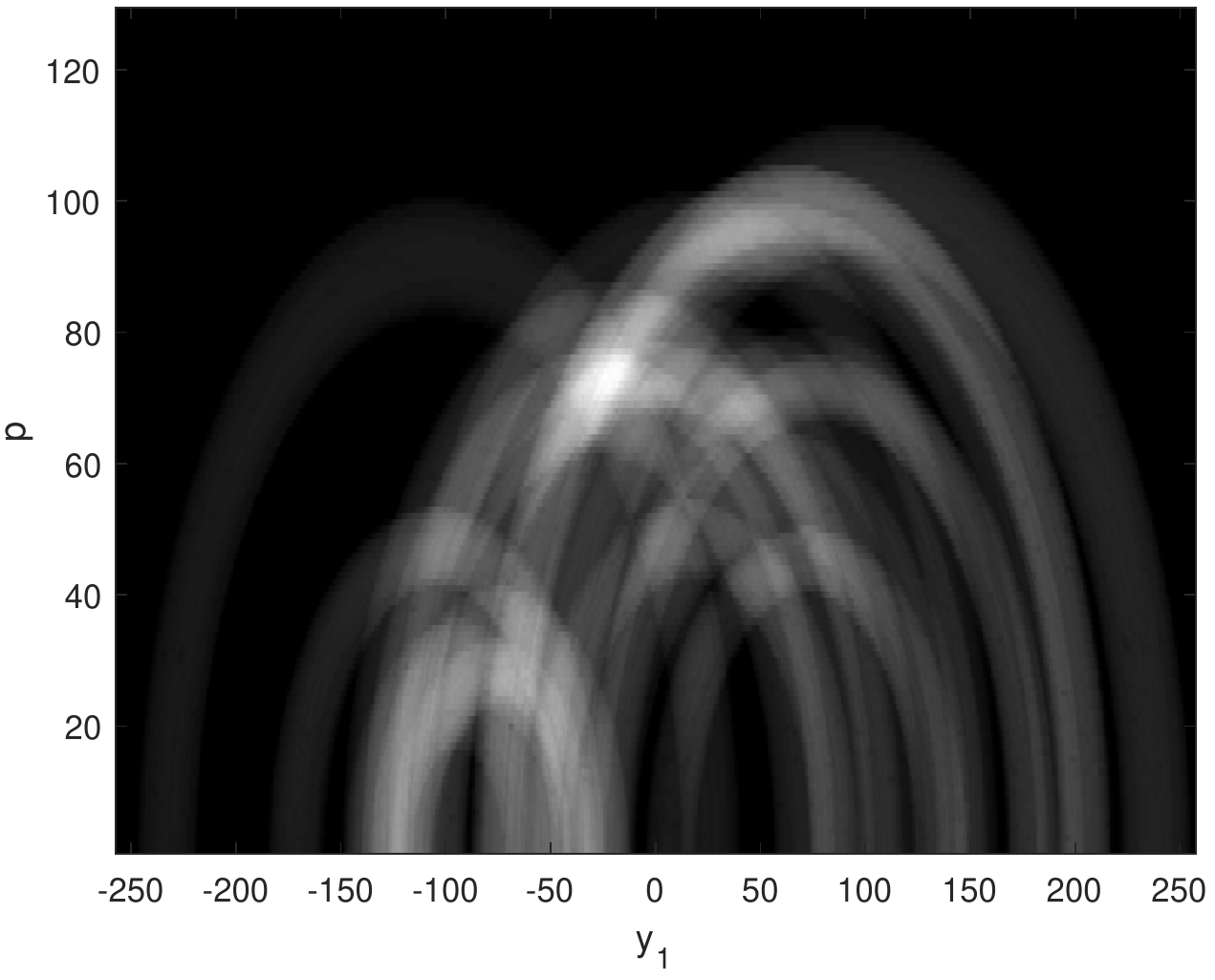}
\subcaption{$\mathcal{E}_1 f$} \label{F7f}
\end{subfigure}
\caption{Image phantoms and $\mathcal{E}_j$ sinograms.}
\label{F_phan}
\end{figure}

As a quantitative measure of image reconstruction performance, we use the least squares error
$$\delta = \frac{\|\vx_{\epsilon} - \vx\|_2}{\|\vx\|_2},$$
where $\vx$ is the true value, and $\vx_{\epsilon}$ is a reconstruction. Before calculating $\delta$, $\vx$ and $\vx_{\epsilon}$ are normalized to have max value 1, and $\delta$ is calculated on $\{x_2 > 0\}$.

\subsection{Reconstruction methods}
\label{recon_methods}
In this section, we discuss our methods. We consider the imaging methods listed below:
\begin{enumerate}
\item Conjugate Gradient Least Squares (CGLS) with Tikhonov regularization \cite{p32} - In this case, $G(\vx) = \|\vx\|_2^2$, $\mathcal{X} = \mathbb{R}^{m^2}$, and \eqref{obj} is a quadratic objective. 
\item Total Variation (TV) \cite{p33} - We employ the regularization function 
$$G(\vx) = \sqrt{\|\nabla \vx\|^2_2+\beta^2},$$
where $\beta>0$ is an additional smoothing parameter, which is included so the gradient is defined at zero and techniques from convex optimization may be applied. We restrict the solution space $\mathcal{X} = \mathbb{R}_+^{m^2}$ to vectors with non-negative entries, as, in many imaging applications, we are imaging physical quantities (e.g., densities), which are non-negative. 
The objective \eqref{obj}, although convex, is non-linear, and we solve \eqref{obj} using gradient based solvers and the code provided in \cite{p33}.
\end{enumerate}
CGLS with Tikhonov regularization, while fast and efficient, offers modest regularization, and is included in our simulations mainly to highlight some of the artifacts we expect to see if the solution is not sufficiently regularized. TV is included to show the effectiveness of a stronger regularizer in removing the artifacts.

\subsection{Results}
In this subsection, we present our results. See figure \ref{Fr1}, where we present reconstructions of the annulus and ellipse phantom using CGLS and Tikhonov regularization. In table \ref{Tr1}, we give the least squares error corresponding to the reconstructions in figure \ref{Fr1}.
\begin{figure}[!h]
\centering
\begin{subfigure}{0.24\textwidth}
\includegraphics[width=0.9\linewidth, height=3.2cm, keepaspectratio]{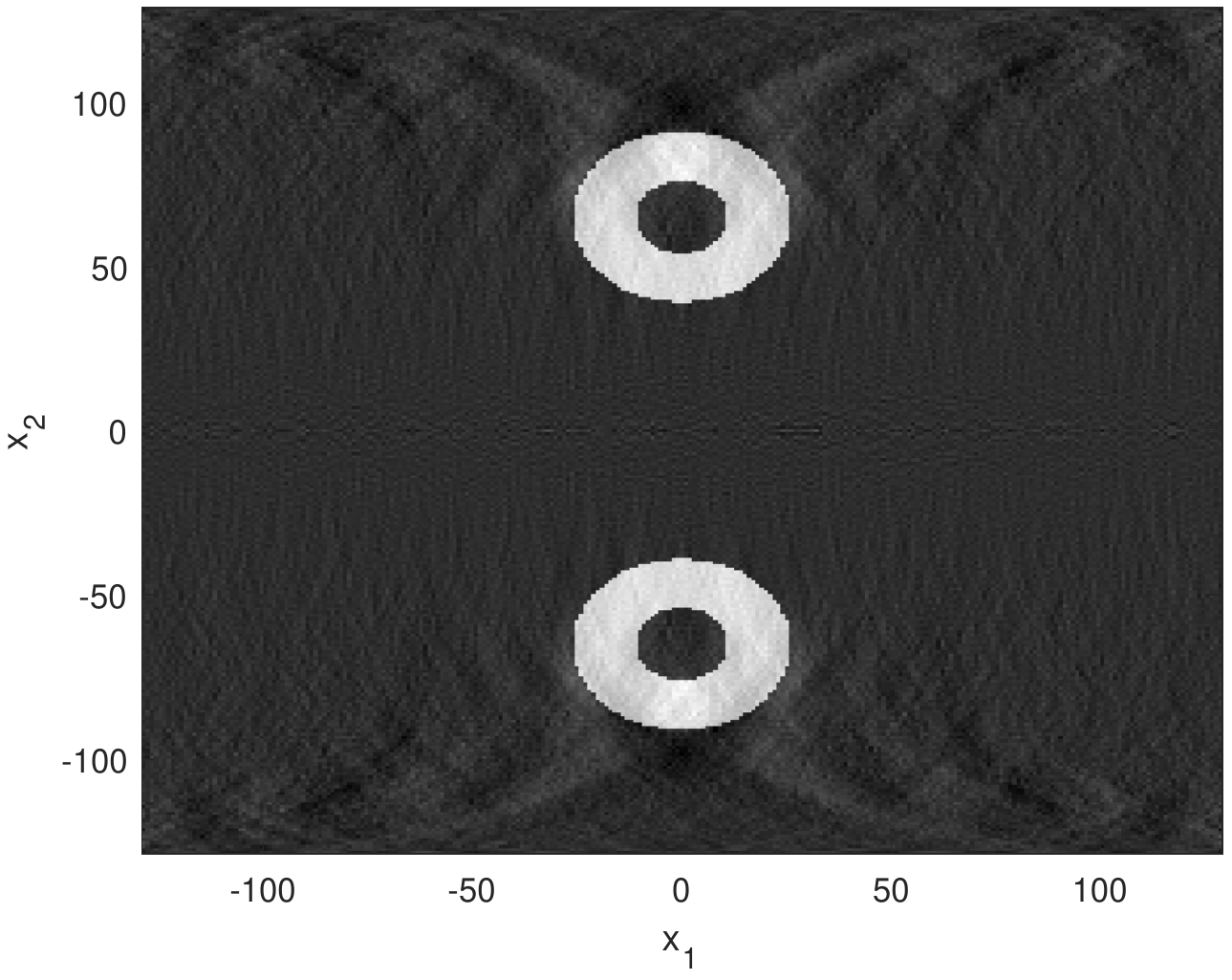}
\end{subfigure}
\begin{subfigure}{0.24\textwidth}
\includegraphics[width=0.9\linewidth, height=3.2cm, keepaspectratio]{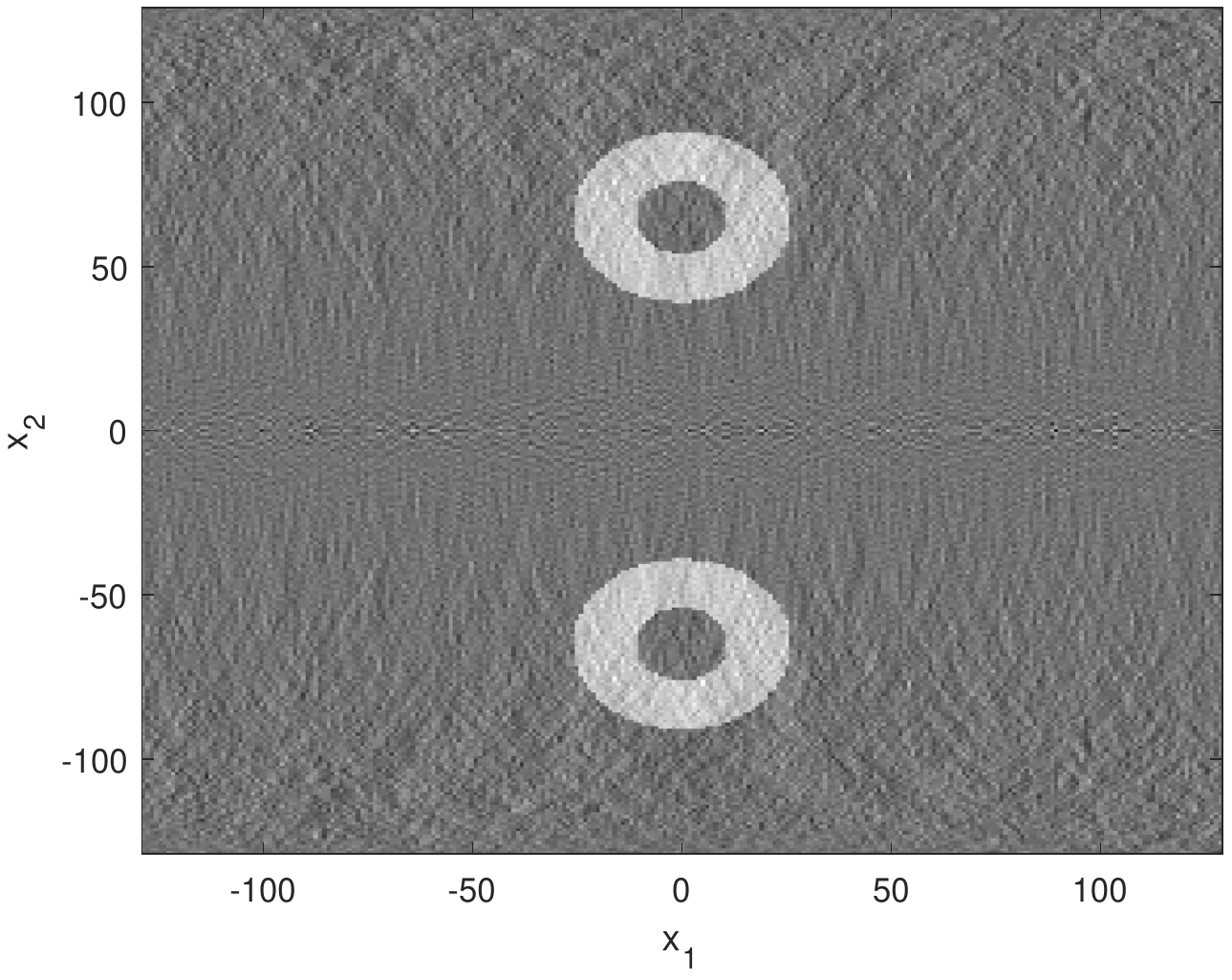}
\end{subfigure}
\begin{subfigure}{0.24\textwidth}
\includegraphics[width=0.9\linewidth, height=3.2cm, keepaspectratio]{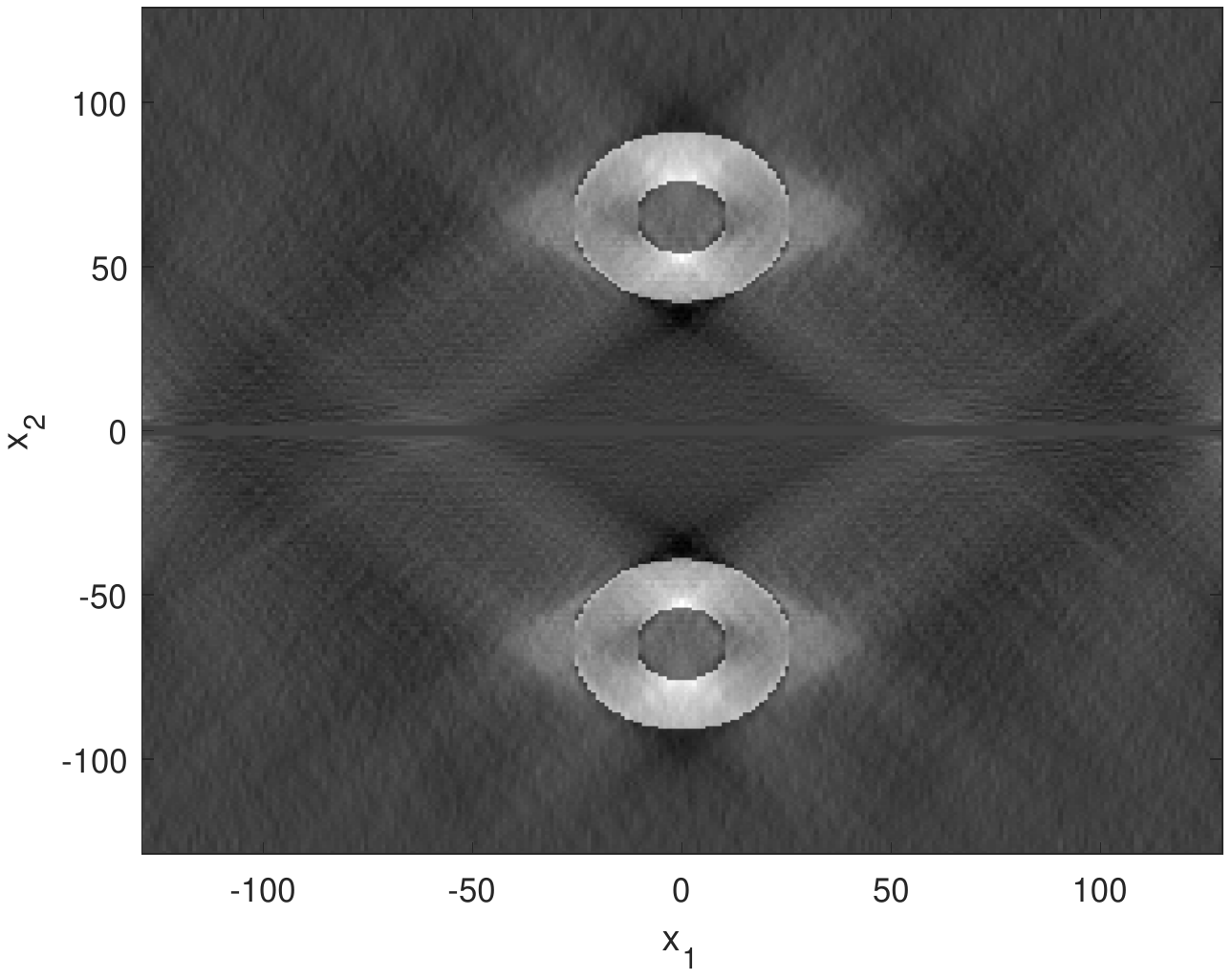}
\end{subfigure}
\begin{subfigure}{0.24\textwidth}
\includegraphics[width=0.9\linewidth, height=3.2cm, keepaspectratio]{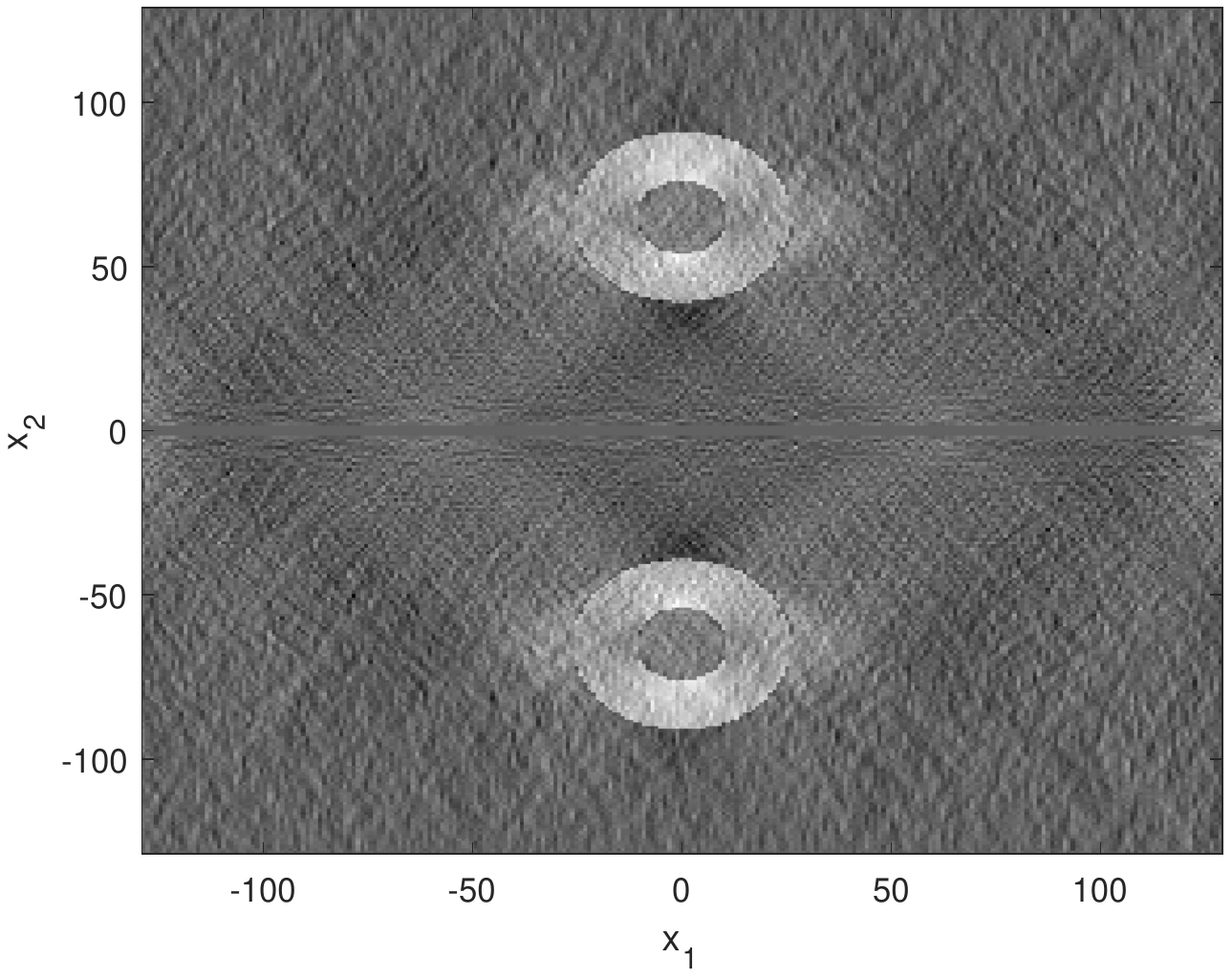}
\end{subfigure}
\begin{subfigure}{0.24\textwidth}
\includegraphics[width=0.9\linewidth, height=3.2cm, keepaspectratio]{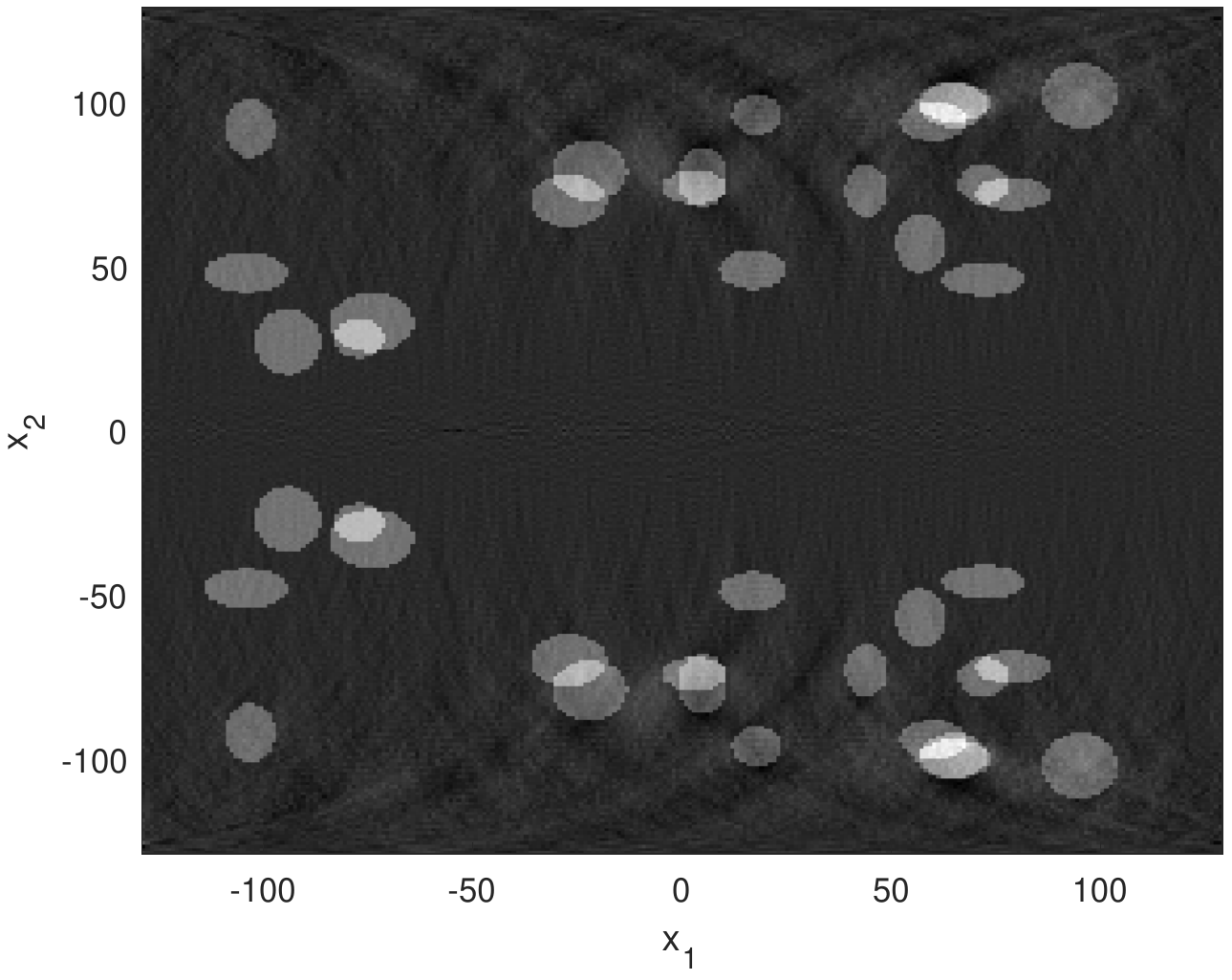}
\subcaption*{$\mathcal{E}_0f$ recon, $1\%$ noise}
\end{subfigure}
\begin{subfigure}{0.24\textwidth}
\includegraphics[width=0.9\linewidth, height=3.2cm, keepaspectratio]{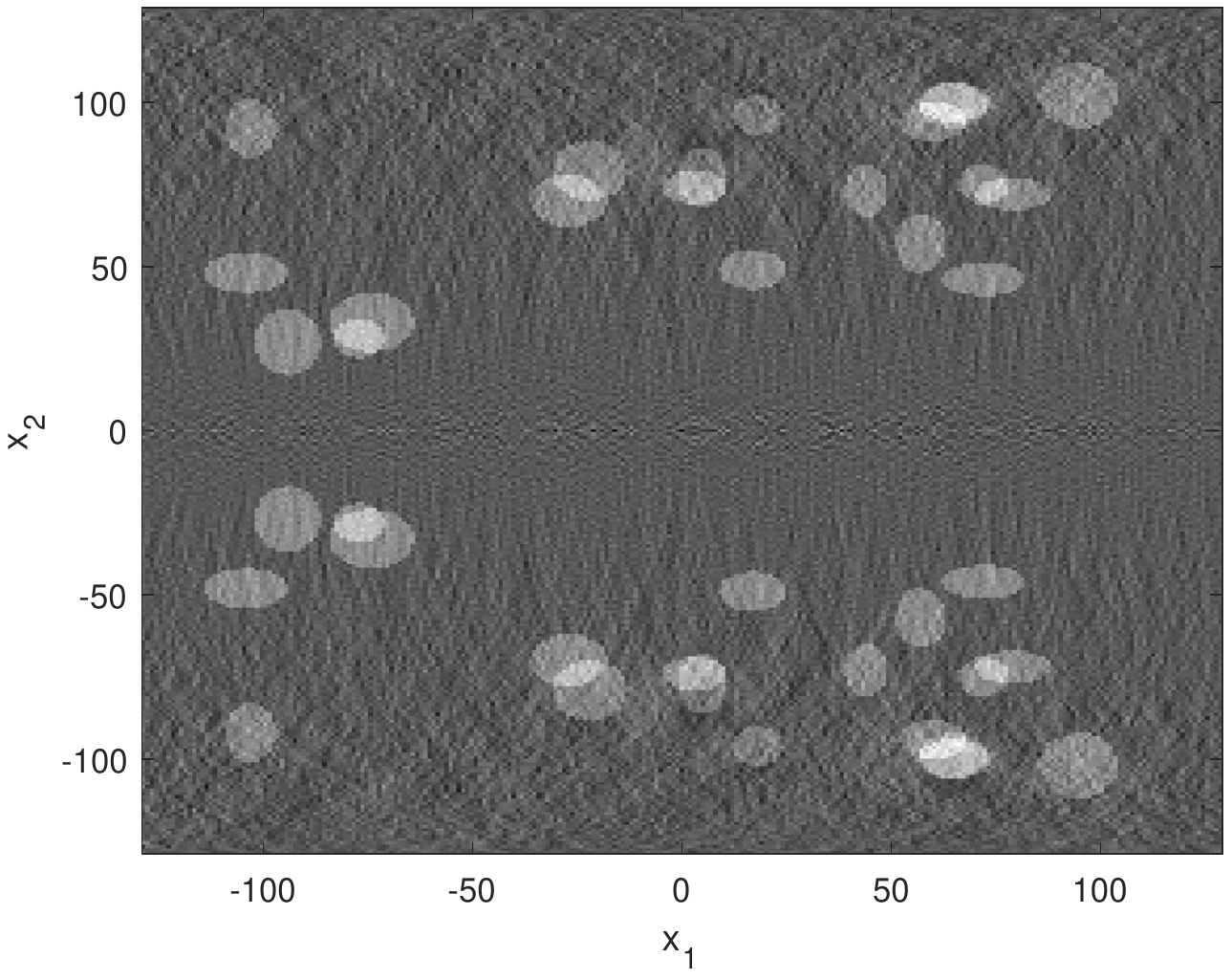}
\subcaption*{$\mathcal{E}_0f$ recon, $5\%$ noise}
\end{subfigure}
\begin{subfigure}{0.24\textwidth}
\includegraphics[width=0.9\linewidth, height=3.2cm, keepaspectratio]{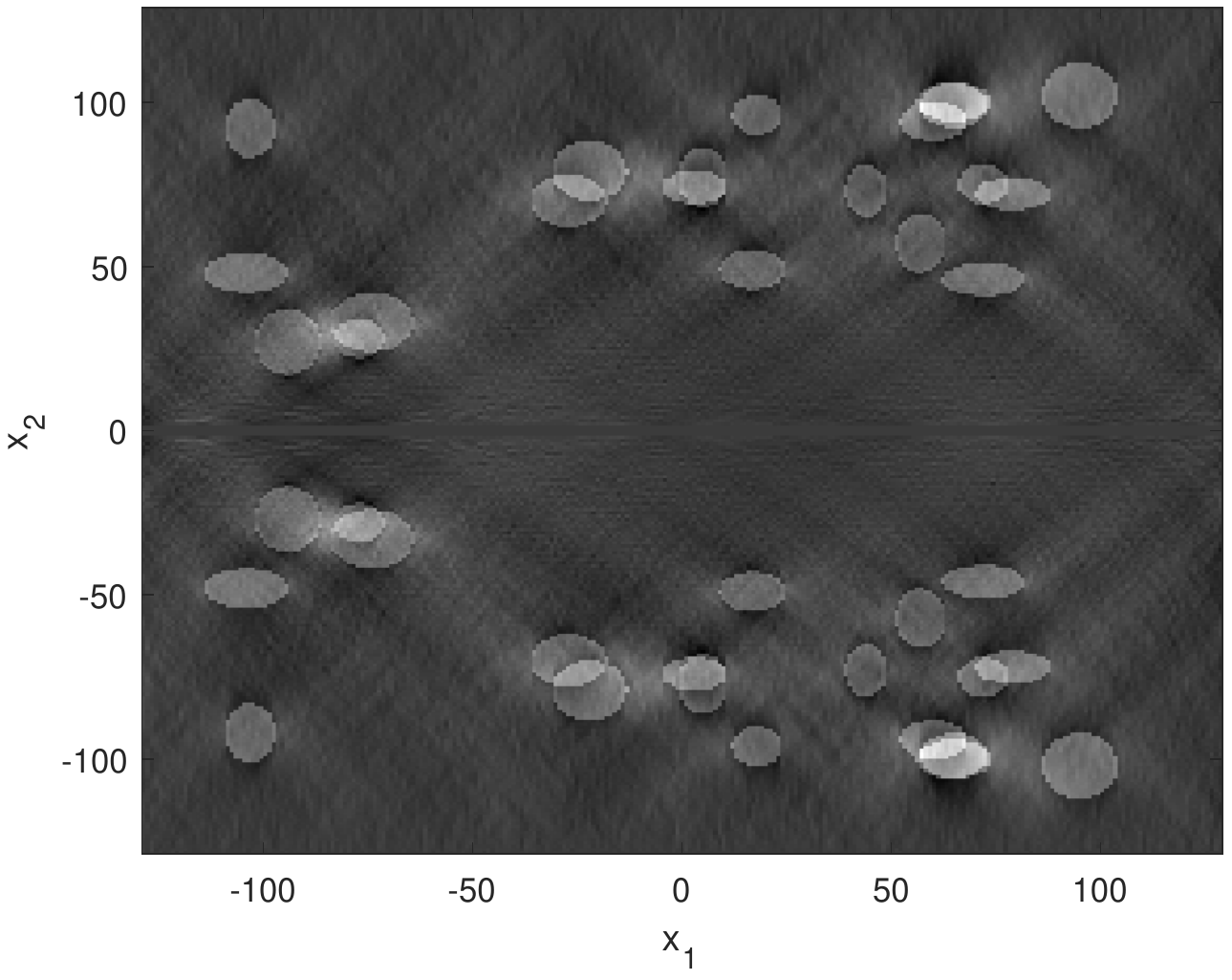}
\subcaption*{$\mathcal{E}_1f$ recon, $1\%$ noise}
\end{subfigure}
\begin{subfigure}{0.24\textwidth}
\includegraphics[width=0.9\linewidth, height=3.2cm, keepaspectratio]{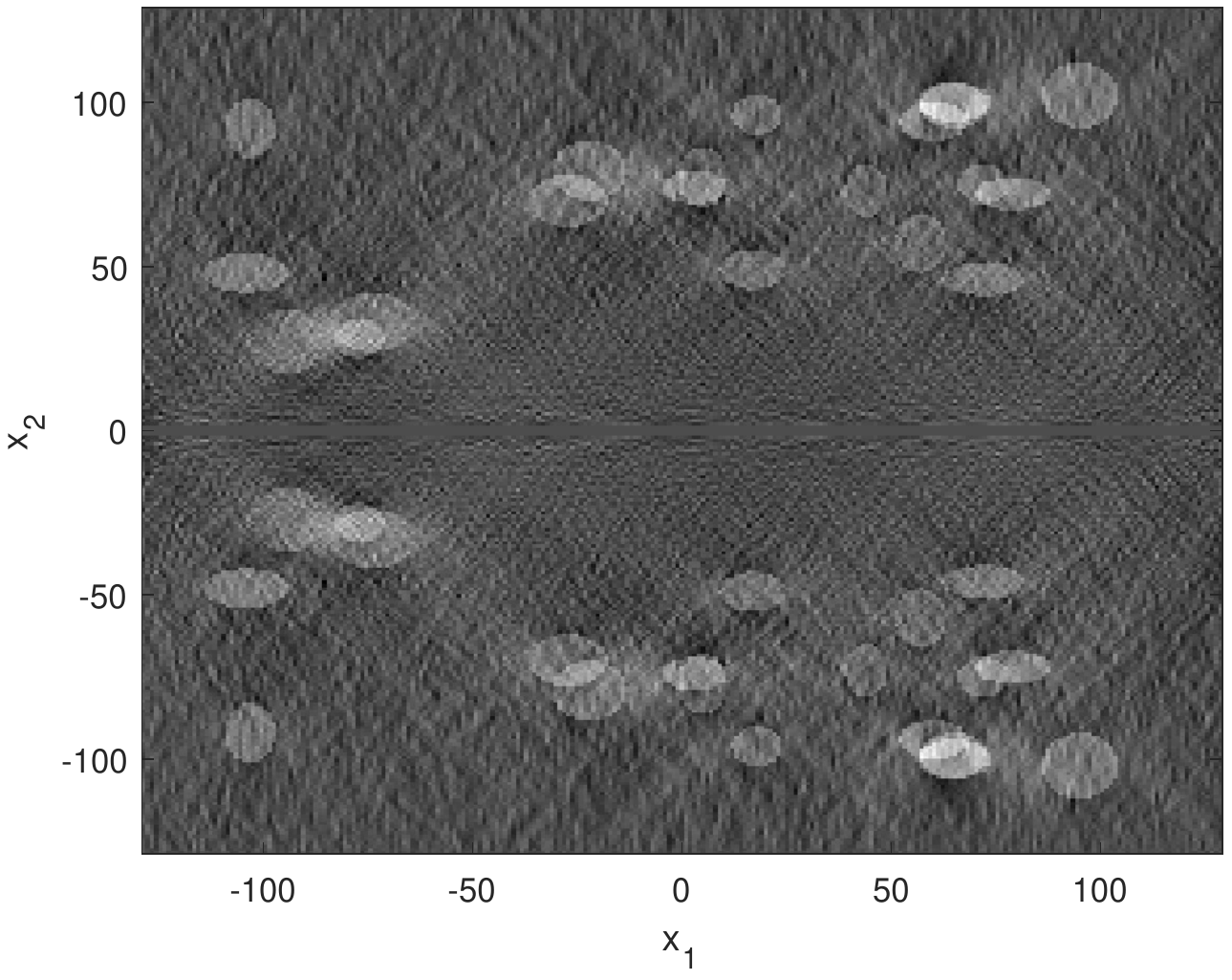}
\subcaption*{$\mathcal{E}_1f$ recon, $5\%$ noise} \label{h13}
\end{subfigure}
\caption{CGLS reconstructions from ellipse and hyperbola integrals, with varying levels of noise. The annulus phantom reconstructions are on the top row, and the ellipse phantom reconstructions are on the bottom row.}
\label{Fr1}
\end{figure}
The regularization parameter, $\lambda$, was chosen experimentally to give the best results in terms of least squares error. With low level ($1\%$) noise, we see artifacts in the reconstructions from $\mathcal{E}_j f$ data. We see a blurring effect near jump singularities with direction parallel to $x_1$, and this effect is more pronounced in the $\mathcal{E}_1f$ reconstructions. This is because $\mathcal{E}_j$ cannot detect singularities in directions parallel to $x_1$, and with limited sinogram data, as considered here (i.e., the ellipse radii are bounded by $(m+1)/2$), edges with directions a small angle from $x_1$ cannot be resolved. This creates a blurring ``cone" effect, which is particularly highlighted in the $\mathcal{E}_1f$ reconstructions. See \cite{p13} for more discussion on the visible and invisible singularities in $\mathcal{E}_j f$ data. Overall, the $\mathcal{E}_0f$ data reconstructions are higher quality than the $\mathcal{E}_1f$ data reconstructions, both in terms of image quality and least squares error. The sampling of $(t,y_1)$, discussed in subsection \ref{data_sim}, was chosen to satisfy the conditions of Theorem \ref{trans_inv_thm}, so that $f$ can be recovered uniquely from $\mathcal{E}_jf$. With this sampling scheme, the specific shape of the ellipses allows them to detect more singularities of $f$, when compared to hyperbola, which helps to explain the higher inversion instability of $\mathcal{E}_1$, when compared to $\mathcal{E}_0$, in this example. 

The CGLS algorithm is not robust to noise, and the artifacts in both phantoms are significantly amplified with increased noise, e.g., as we move from $1\%$ to $5\%$ added noise. This is also reflected in the least squares error results of table \ref{Tr1}.
\begin{table}[!ht]
    \centering
    \begin{tabular}{|l|l|l|l|l|}
    \hline
        phantom & $\mathcal{E}_0f$, $1\%$ noise & $\mathcal{E}_0f$, $5\%$ noise & $\mathcal{E}_1f$, $1\%$ noise & $\mathcal{E}_0f$, $5\%$ noise \\ \hline
        annulus & 0.27 & 0.64 & 0.57 & 0.75 \\ \hline
        ellipses & 0.27 & 0.73 & 0.54 & 0.88  \\ \hline
    \end{tabular}
\caption{Least squares errors, $\delta$, corresponding to the reconstructions in figure \ref{Fr1}.}
\label{Tr1}
\end{table}

In all reconstructions in figure \ref{Fr1}, the original phantom image is reflected in the $x_1$ axis. This is to be expected, as we can only recover even functions in $x_2$ using $\mathcal{E}_jf$ data, by Corollary \ref{corr_null}. In \cite{p13}, the authors prove that, in reconstructions from $\mathcal{E}_jf$ data, there are artifacts which are reflections in the $x_1$ axis. More specifically, they showed that any singularity in the wavefront set of $f$, which is detected by $\mathcal{E}_jf$, would be reflected in the $x_1$ axis and appear as an additional (unwanted) singularity in the reconstruction. Thus, the image reconstructions of figure \ref{Fr1} are also in line with the theory of \cite{p13}.
\begin{figure}[!h]
\centering
\begin{subfigure}{0.24\textwidth}
\includegraphics[width=0.9\linewidth, height=3.2cm, keepaspectratio]{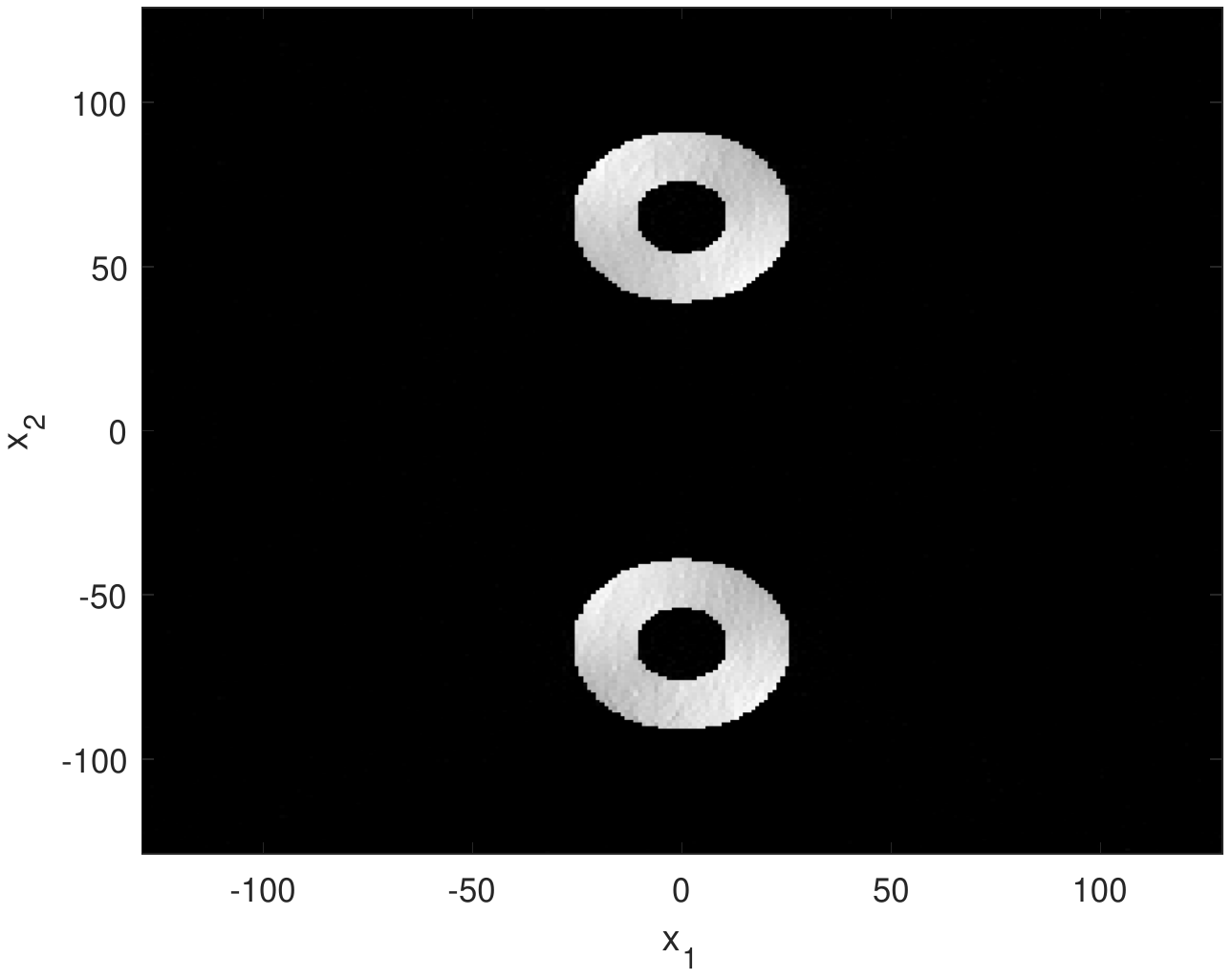}
\end{subfigure}
\begin{subfigure}{0.24\textwidth}
\includegraphics[width=0.9\linewidth, height=3.2cm, keepaspectratio]{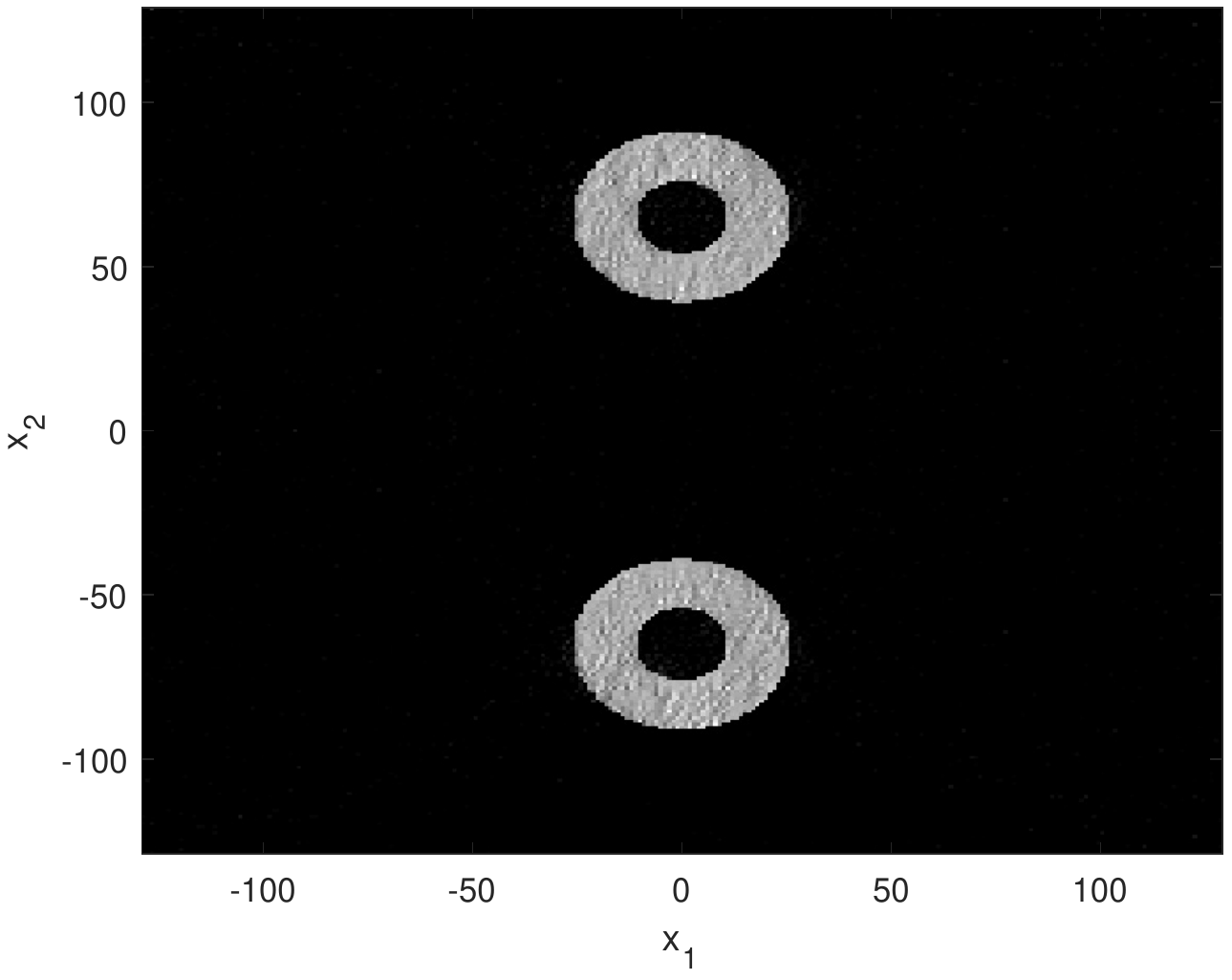}
\end{subfigure}
\begin{subfigure}{0.24\textwidth}
\includegraphics[width=0.9\linewidth, height=3.2cm, keepaspectratio]{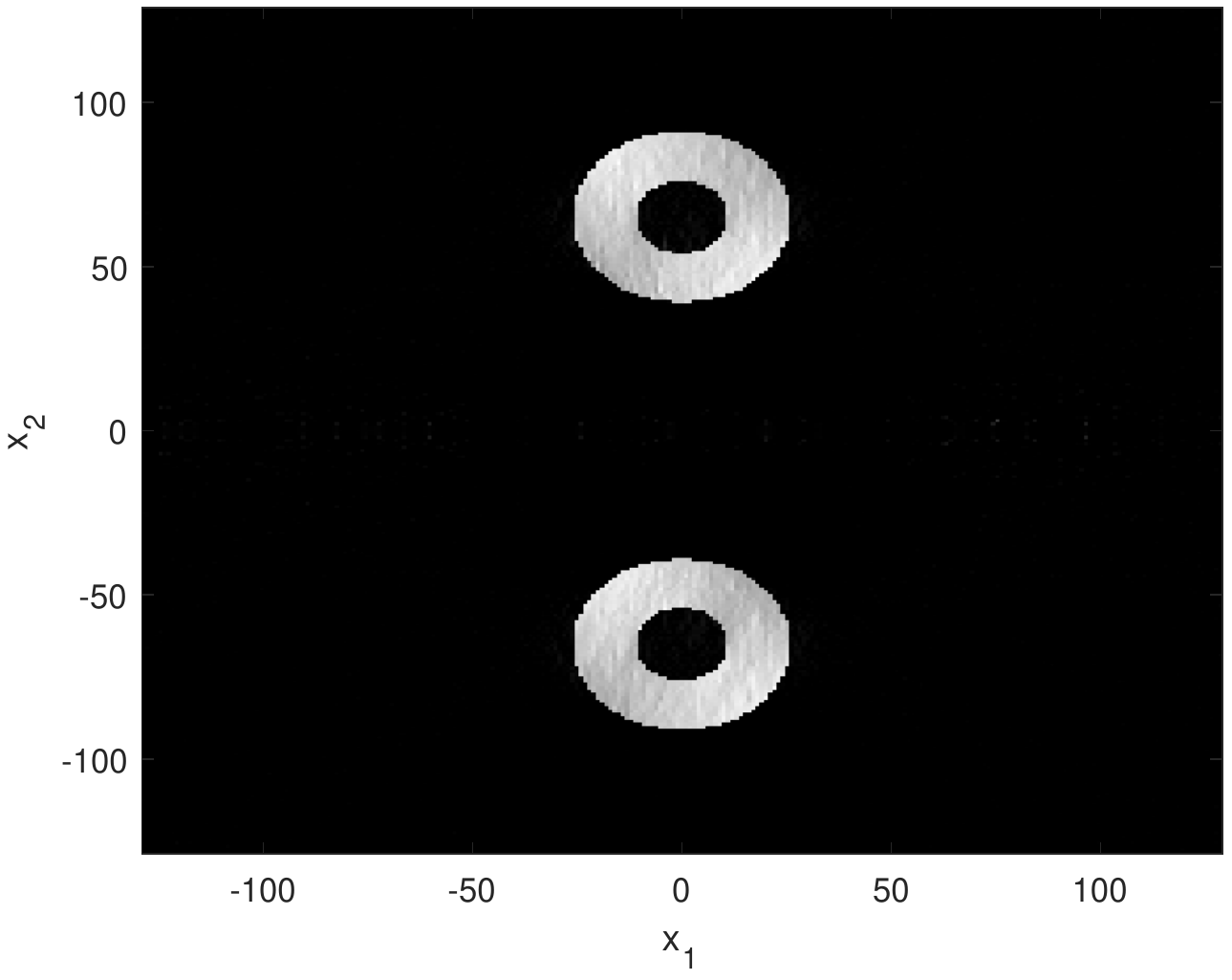}
\end{subfigure}
\begin{subfigure}{0.24\textwidth}
\includegraphics[width=0.9\linewidth, height=3.2cm, keepaspectratio]{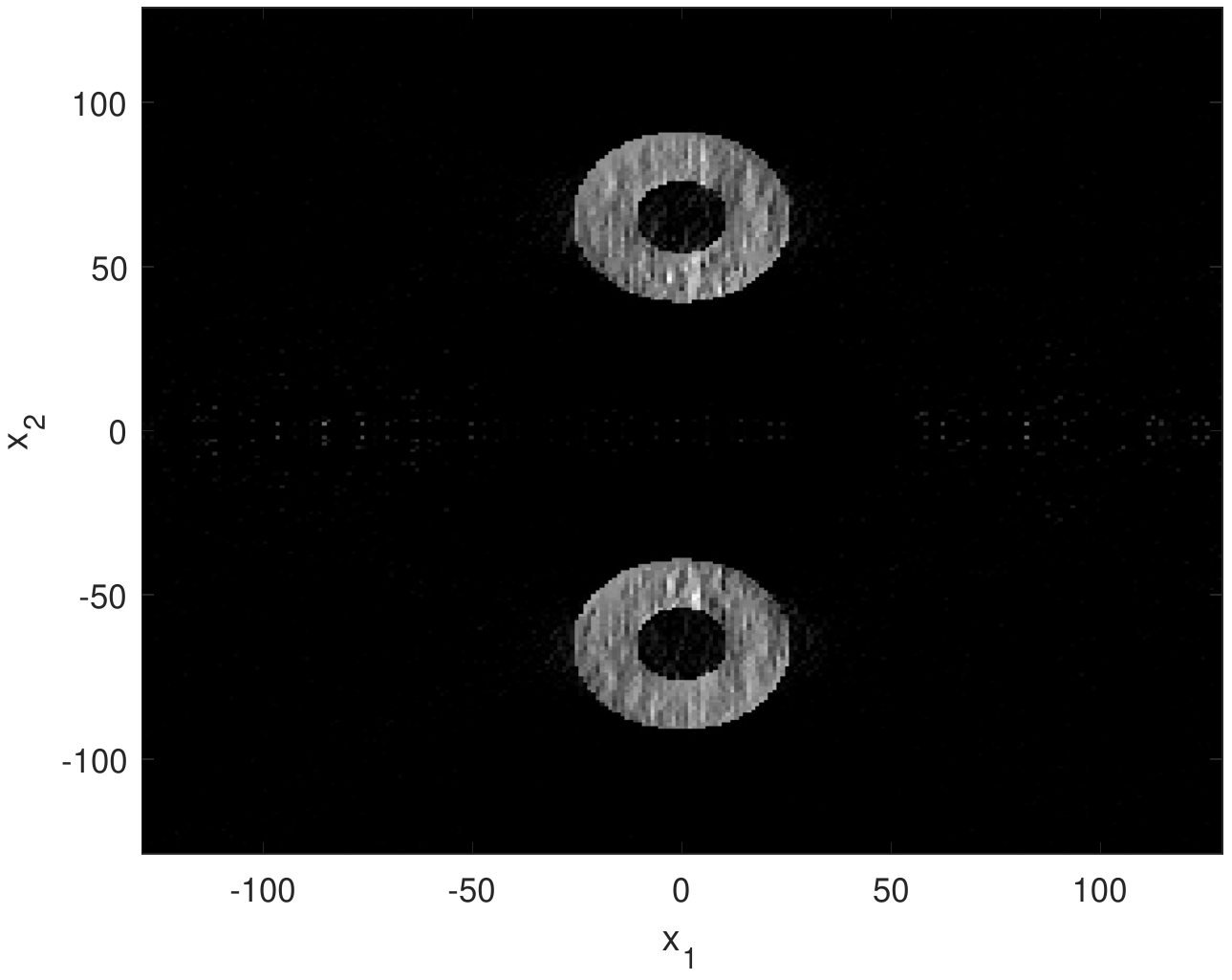}
\end{subfigure}
\begin{subfigure}{0.24\textwidth}
\includegraphics[width=0.9\linewidth, height=3.2cm, keepaspectratio]{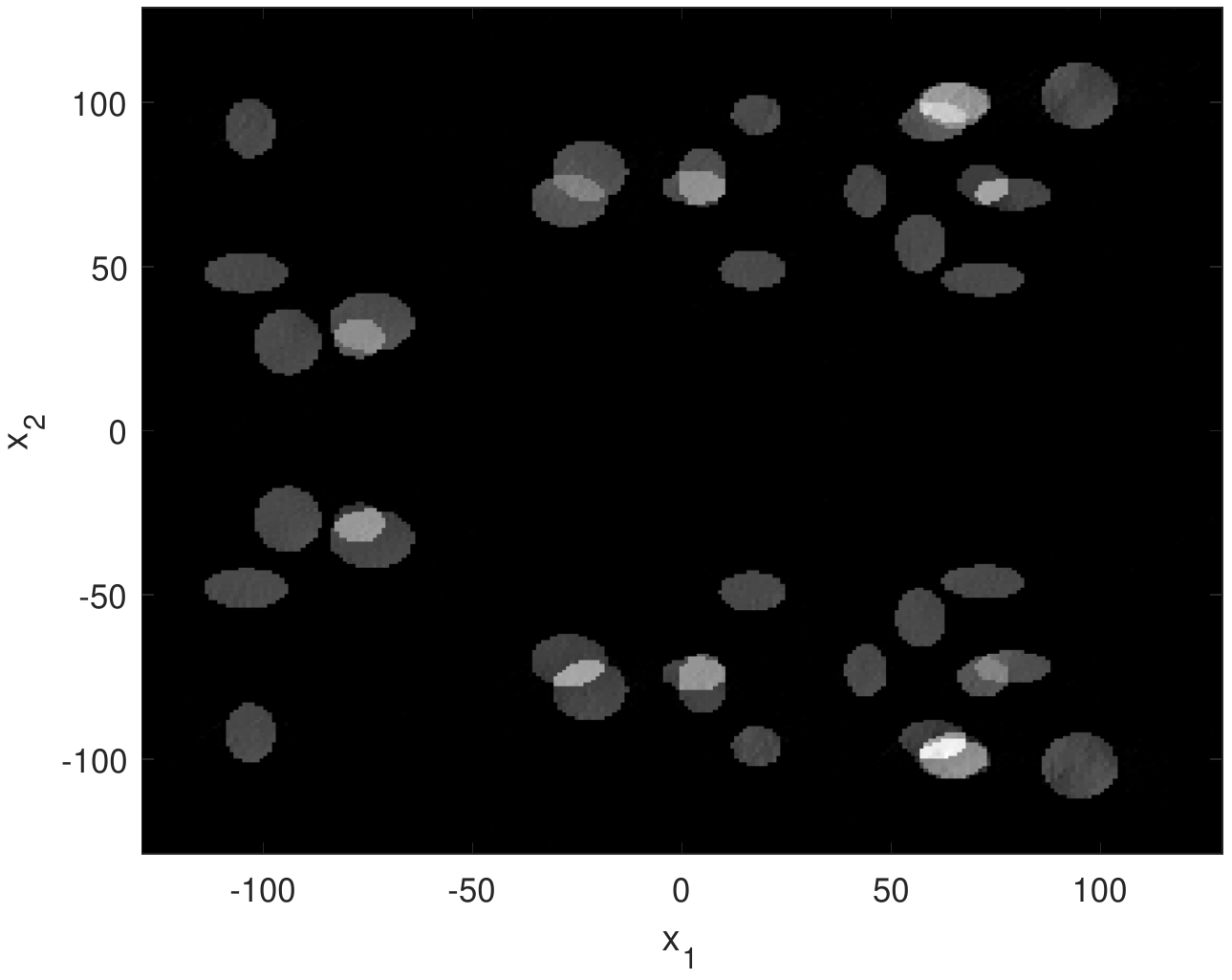}
\subcaption*{$\mathcal{E}_0f$ recon, $1\%$ noise}
\end{subfigure}
\begin{subfigure}{0.24\textwidth}
\includegraphics[width=0.9\linewidth, height=3.2cm, keepaspectratio]{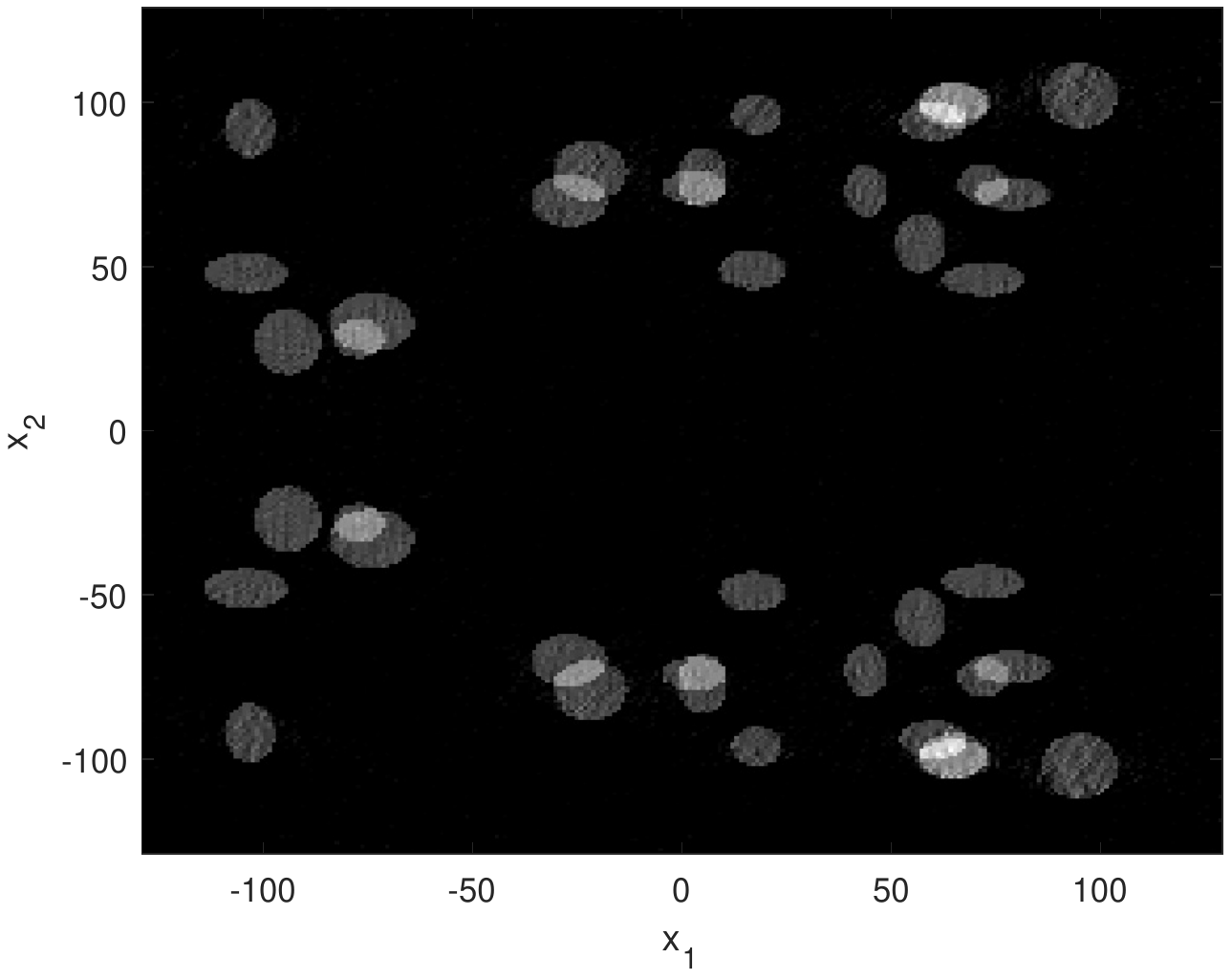}
\subcaption*{$\mathcal{E}_0f$ recon, $5\%$ noise}
\end{subfigure}
\begin{subfigure}{0.24\textwidth}
\includegraphics[width=0.9\linewidth, height=3.2cm, keepaspectratio]{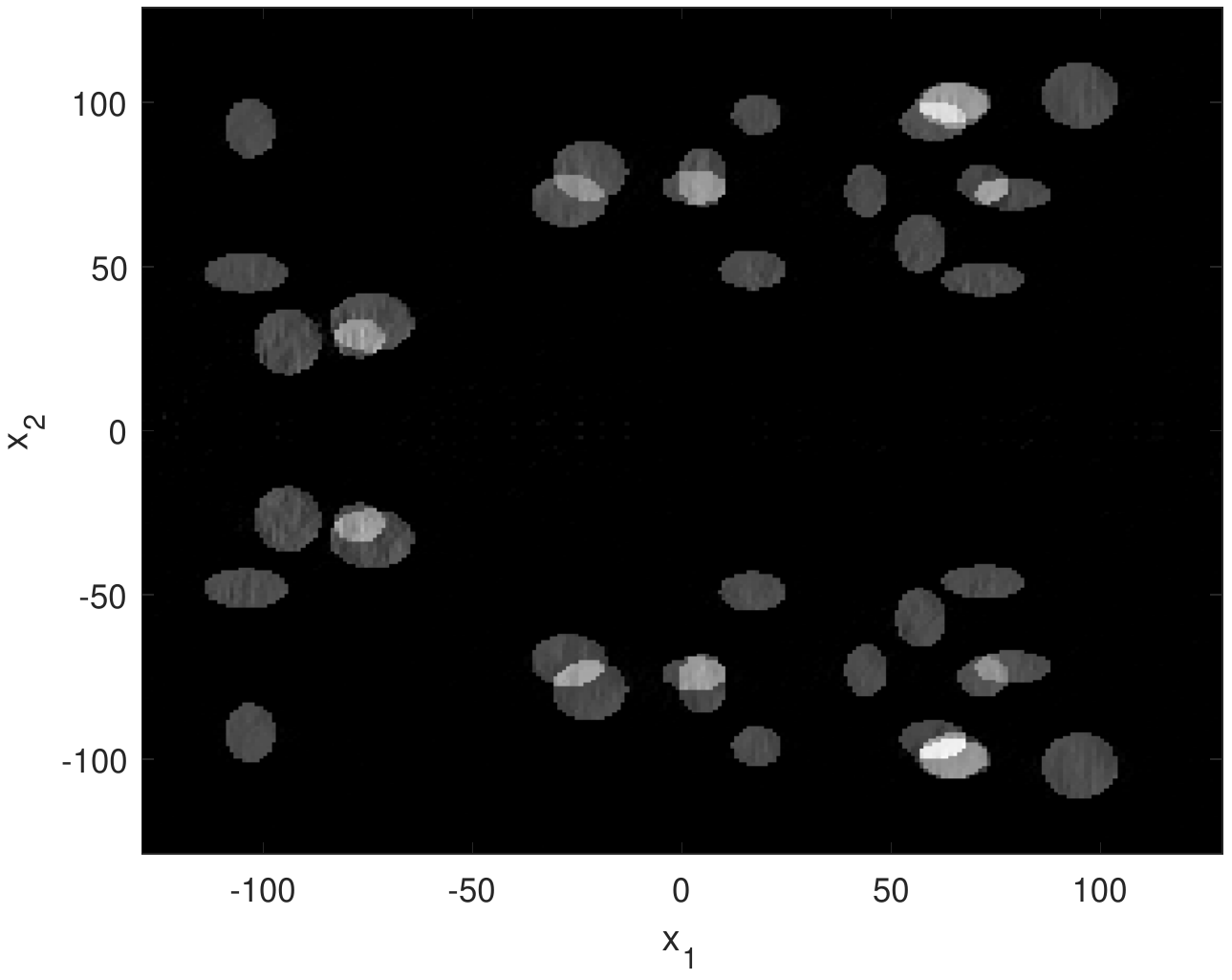}
\subcaption*{$\mathcal{E}_1f$ recon, $1\%$ noise}
\end{subfigure}
\begin{subfigure}{0.24\textwidth}
\includegraphics[width=0.9\linewidth, height=3.2cm, keepaspectratio]{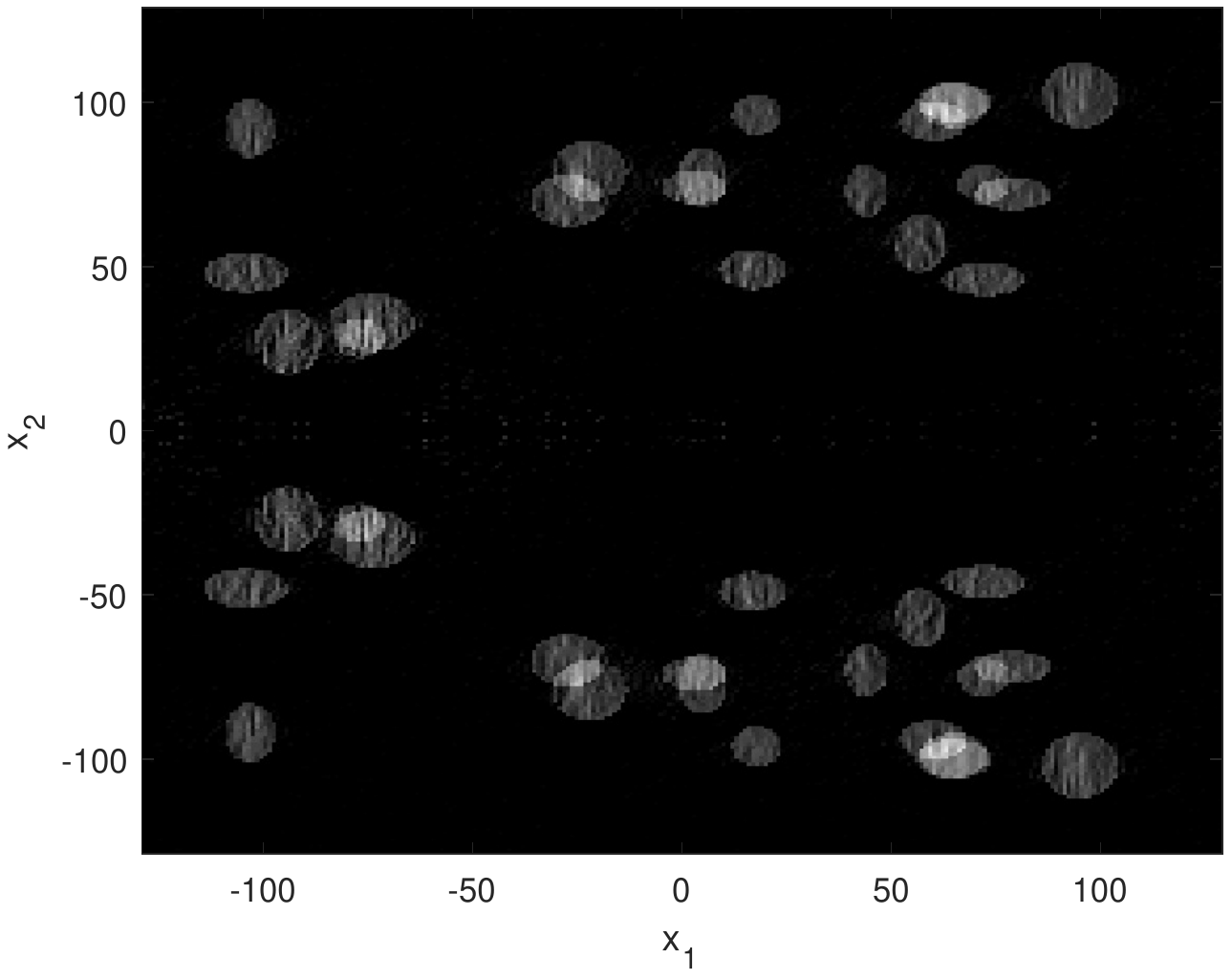}
\subcaption*{$\mathcal{E}_1f$ recon, $5\%$ noise} \label{h13}
\end{subfigure}
\caption{TV reconstructions from ellipse and hyperbola integrals, with varying levels of noise. The annulus phantom reconstructions are on the top row, and the ellipse phantom reconstructions are on the bottom row.}
\label{Fr2}
\end{figure}
\begin{table}[!ht]
    \centering
    \begin{tabular}{|l|l|l|l|l|}
    \hline
        phantom & $\mathcal{E}_0f$, $1\%$ noise & $\mathcal{E}_0f$, $5\%$ noise & $\mathcal{E}_1f$, $1\%$ noise & $\mathcal{E}_0f$, $5\%$ noise \\ \hline
        annulus & 0.17 & 0.34 & 0.21 & 0.54 \\ \hline
        ellipses & 0.09 & 0.23 & 0.09 & 0.36  \\ \hline
    \end{tabular}
\caption{Least squares errors, $\delta$, corresponding to the reconstructions in figure \ref{Fr2}.}
\label{Tr2}
\end{table}

In figure \ref{Fr2}, we present image reconstructions using the TV method outlined in point (2) of subsection \ref{recon_methods}. In table \ref{Tr2}, we give the least squares error values corresponding to the reconstructions in figure \ref{Fr2}. TV successfully suppresses the artifacts due to limited edge detection and is more robust to noise, when compared to CGLS. The reflection artifacts in the $x_1$ axis are still present however. This is because TV enforces sparse gradients. Since the reflected phantoms have sparse gradients, there is no reason for the TV penalty to suppress the reflected artifact. The original shape and size of both phantoms is well preserved at all noise levels considered, and there is less overall background noise, when compared to CGLS. There is a visible noise effect, however, on the non-zero parts of the phantoms, particularly when $5\%$ noise is added and we are reconstructing $f$ from $\mathcal{E}_1f$ data. The noise effect is also more present in the annulus phantom reconstructions, when compared to the ellipse phantom reconstructions. This is also reflected in the least squares error results of table \ref{Tr2}, as the $\delta$ values corresponding to the annulus phantom reconstruction are higher, when compared to the ellipse phantom reconstructions, across all examples considered here.

\section{Conclusion}
In this paper, we presented novel injectivity results and inversion methods for a class of generalized Abel operators. In Theorem \ref{anti_volt_thm}, we provided sufficient conditions on the Abel operator kernel for existence and uniqueness of solution in $L^2([a,b])$. In Theorem \ref{trans_inv_thm}, we applied our theory to ellipsoid and hyperboloid Radon transforms, $\mathcal{R}_j$, which have applications to URT when $j=0$. We proved injectivity of $\mathcal{R}_j$, and provided an inversion method based on Neumann series. Specifically, we showed, after Fourier transformation, that $\mathcal{R}_j$ could be reduced to a set of generalized Abel operators. The operator kernels were then shown to satisfy the conditions of Theorem \ref{anti_volt_thm} to prove injectivity of $\mathcal{R}_j$. We generalized the results of Theorem \ref{trans_inv_thm} in section \ref{gen_section}, and showed that $f$ could be recovered uniquely from its integrals over the surfaces of revolution of a more general set of continuous curves. In section \ref{results}, we presented simulated reconstructions of image phantoms in two-dimensions from ellipse and hyperbola integral data. We compared Tikhonov and TV regularization methods. Strong artifacts were observed in the Tikhonov reconstructions due to data limitations, and Tikhonov regularization was not sufficient to suppress the noise. TV faired much better, and successfully suppressed the more significant artifacts observed in the Tikhonov reconstructions.

In discussion \ref{dis_1}, we outlined the smoothing properties of $\mathcal{R}_j$, and discussed the inversion instabilities. In further work, we aim to formalize these ideas and derive stability estimates on Sobolev scale.

\section*{Acknowledgements:} The authors wish to acknowledge
funding support from The V
Foundation, Brigham Ovarian Cancer Research Fund, Abcam Inc., and Aspira Women's Health.

\appendix

\section{Additional applications of generalized Abel equations}
\label{appA}
In this section, we discuss additional applications of generalized Abel equations to Radon transforms which are rotation invariant.
\subsection{A rotation invariant spherical Radon transform}
In \cite{p10}, the authors consider the spherical means transform
\begin{equation}
\label{sph_means}
\hat{f}(\vy) = \frac{1}{\omega_n}\int_{\vxi\in S^{n-1}} f\paren{p(\vom +\vxi)}\mathrm{d}\Omega,
\end{equation}
where $\omega_n$ denotes the volume of $S^{n-1}$, $\mathrm{d}\Omega$ is the standard measure on $S^{n-1}$, and $\vy = 2p\vom$. $\hat{f}(\vy)$ is the mean value of $f$ over the sphere with center $\vy/2$ and radius $|\vy|/2$. We use the notation of \cite{p10} in this appendix. In \cite{p10}, the authors present inversion formulae for $\hat{f}$, and investigate its applications to the Darboux equation.

After spherical harmonic decomposition of $\hat{f}$, in \cite[Theorem 1]{p10}, the authors reduce $\hat{f}$ to a set of Volterra equations
\begin{equation}
\label{fl}
\hat{f}_l(s) = \paren{\frac{2}{s}}^{n-1}\frac{\omega_{n-1}}{\omega_n C_l^{\gamma}(1)} \int_{0}^s C_l^{\gamma}\paren{\frac{r}{s}} f_l(r)r^{2\gamma}\paren{1-\paren{\frac{r}{s}}^2}^{\gamma-\frac{1}{2}} \mathrm{d}r,
\end{equation}
where $\gamma = \frac{n-2}{2}$, and $C_l^{\gamma}$ is a Gegenbauer polynomial of  degree $l$. The equations \eqref{fl} are to be solved for $f_l$, for every $l\in \mathbb{Z}$.

Equation \eqref{fl} can be converted to a generalized Abel equation by multiplying both sides by $s^{2\gamma - 1}$
\begin{equation}
\begin{split}
\label{fl_1}
\frac{s^{n-1}}{\lambda_n}\hat{f}_l(s) &= \int_{0}^s \paren{s-r}^{\alpha} \left[r^{\alpha + \frac{1}{2}} \paren{s + r}^{\alpha} C_l^{\gamma}\paren{\frac{r}{s}}\right]f_l(r)\mathrm{d}r \\
&= \int_{0}^s \paren{s-r}^{\alpha} K_l(s,r) f_l(r)\mathrm{d}r
\end{split}
\end{equation}
where $\lambda_n = \frac{2 \omega_{n-1}}{\omega_n C_l^{\gamma}(1)} \neq 0$, and $\alpha = (n-3)/2$.  In \cite[Theorem 1]{p10}, the authors solve \eqref{fl} using a generalization of Cormack's formula \cite{p11}, under the assumption that the $f_l$ are smooth. We provide an $L^2$ solution to \eqref{fl} in the following theorem.

\begin{theorem}
\label{sph_inv_thm}
Let $f_l \in L^2([a,b])$, with $0<a<b<\infty$. Then \eqref{fl} can solved uniquely for $f_l$ if $\hat{f}_l$ is known for $s\in [a,b]$. 
\end{theorem}
\begin{proof}
Let $T_0$ be defined as in Theorem \ref{weak_volt_thm}. By \eqref{fl_1}, it is sufficient to show that $K_l$ satisfies the conditions of Corollary \ref{corr1} are satisfied. We have 
$$K_l(s,s) = 2^{\alpha}s^{2\alpha + \frac{1}{2}} C_l^{\gamma}(1) \neq 0,$$
for $s \in [a,b]$. Let 
$$g_1(s,r) = r^{\alpha + \frac{1}{2}} \paren{s + r}^{\alpha},$$
and 
$$g_2(s,r) = C_l^{\gamma}\paren{\frac{r}{s}}.$$ 
Then, $g_1 \in C^{\infty}(T_0)$. Further, since $g_2$ is the composition of a polynomial and $r/s$, which is smooth on $T_0$, $g_2 \in C^{\infty}(T_0)$. Thus, $K_l \in C^{\infty}(T_0)$, and the conditions of Corollary \ref{corr1} are satisfied. This finishes the proof.
\end{proof}

\begin{remark}
In \cite[Theorem 1]{p10}, an inversion formula is provided for \eqref{fl} for smooth functions, which cannot have singularities (edges). In many imaging applications, the $f$ of interest have sharp edges, and thus a solution in $L^2$ is desired. Theorem \ref{sph_inv_thm} provides such a solution. It is noted that equation \ref{fl_1}, given the smoothness of $K_l$, could also be solved for $L^2$ functions using the arguments of \cite[pages 515 and 516]{p4}.  This analysis and Theorem \ref{sph_inv_thm} are included as it gives another concrete example, in addition to Theorem \ref{trans_inv_thm} and the generalizations of section \ref{gen_section}, of the applications of generalized Abel equations and Theorem \ref{anti_volt_thm} to Radon transforms.
\end{remark}


\begin{thebibliography}{10}
\bibitem{p1}
Coker, Jonathan D., and Ahmed H. Tewfik. ``Multistatic SAR image reconstruction based on an elliptical-geometry Radon transform." In 2007 International Waveform Diversity and Design Conference, pp. 204-208. IEEE, 2007.
\bibitem{p2}
Ambartsoumian, Gaik, Jan Boman, Venkateswaran P. Krishnan, and Eric Todd Quinto. ``Microlocal analysis of an ultrasound transform with circular source and receiver trajectories." American Mathematical Society, series= Contemporary Mathematics 598 (2013): 45-58.
\bibitem{p3}
Cormack, Allan Macleod. "Representation of a function by its line integrals, with some radiological applications." Journal of applied physics 34, no. 9 (1963): 2722-2727.
\bibitem{p4}
Quinto, Eric Todd. "The invertibility of rotation invariant Radon transforms." Journal of Mathematical Analysis and Applications 91, no. 2 (1983): 510-522.
\bibitem{p6}
Agranovsky, Mark L., and Eric Todd Quinto. "Injectivity sets for the Radon transform over circles and complete systems of radial
 functions." Journal of Functional Analysis 139, no. 2 (1996): 383-414.
\bibitem{p7}
Davis, Harold Thayer. A survey of methods for the inversion of integrals of Volterra type. Vol. 14, no. 76. Indiana University, 1927.
\bibitem{p8}
Hirkawa, Nakagoro. "On a Simple Integral Equation." Tohoku Mathematical Journal, First Series 8 (1915): 38-41.
\bibitem{p9}
Schiefeneder, Daniela, and Markus Haltmeier. "The Radon transform over cones with vertices on the sphere and orthogonal axes." SIAM Journal on Applied Mathematics 77, no. 4 (2017): 1335-1351.
\bibitem{p10}
Cormack, A. M., and Eric Todd Quinto. "A Radon transform on spheres through the origin in $\mathbb{R}^n$ and applications to the Darboux equation." Transactions of the American Mathematical Society 260, no. 2 (1980): 575-581.
\bibitem{p11}
Cormack, Allan Macleod. "Representation of a function by its line integrals, with some radiological applications." Journal of applied physics 34, no. 9 (1963): 2722-2727.
\bibitem{p12}
Webber, James. "X-ray Compton scattering tomography." Inverse problems in science and engineering 24, no. 8 (2016): 1323-1346.
\bibitem{p13}
Webber, James W., Sean Holman, and Eric Todd Quinto. "Ellipsoidal and hyperbolic Radon transforms; microlocal properties and injectivity." arXiv preprint arXiv:2212.00243 (2022).
\bibitem{p15}
Olver, Peter J. Applications of Lie groups to differential equations. Vol. 107. Springer Science \& Business Media, 1993.
\bibitem{p16}
Roman, Steven. "The formula of Faa di Bruno." The American Mathematical Monthly 87, no. 10 (1980): 805-809.
\bibitem{p17}
Tricomi, Francesco Giacomo. Integral equations. Vol. 5. Courier corporation, 1985.
\bibitem{p18}
Ambartsoumian, Gaik, Rim Gouia-Zarrad, and Matthew A. Lewis. "Inversion of the circular Radon transform on an annulus." Inverse Problems 26, no. 10 (2010): 105015.
\bibitem{p19}
Webber, James, and Eric L. Miller. "Compton scattering tomography in translational geometries." Inverse Problems 36, no. 2 (2020): 025007.
\bibitem{p20}
Cebeiro, Javier, Cécilia Tarpau, Marcela A. Morvidone, Diana Rubio, and Maï K. Nguyen. "On a three-dimensional Compton scattering tomography system with fixed source." Inverse Problems 37, no. 5 (2021): 054001.
\bibitem{p21}
Ambartsoumian, Gaik, and Venkateswaran P. Krishnan. "Inversion of a class of circular and elliptical Radon transforms." Complex Analysis and Dynamical Systems VI. Part 1 (2015): 1-12.
\bibitem{p22}
Denecker, Koen, Jeroen Van Overloop, and Frank Sommen. "The general quadratic Radon transform." Inverse problems 14, no. 3 (1998): 615.
\bibitem{p23}
Grathwohl, Christine. "Seismic imaging with the elliptic Radon transform in 3D: analytical and numerical aspects." PhD diss., Karlsruher Institut für Technologie (KIT), 2020.
\bibitem{p24}
Moon, Sunghwan. "On the determination of a function from an elliptical Radon transform." Journal of Mathematical Analysis and Applications 416, no. 2 (2014): 724-734.
\bibitem{p25}
Grathwohl, Christine, Peer Christian Kunstmann, Eric Todd Quinto, and Andreas Rieder. "Imaging with the elliptic Radon transform in three dimensions from an analytical and numerical perspective." SIAM Journal on Imaging Sciences 13, no. 4 (2020): 2250-2280.
\bibitem{p26}
Narayanan, E. K. "Spherical means with centers on a hyperplane in even dimensions." Inverse Problems 26, no. 3 (2010): 035014.
\bibitem{p27}
Bukhgeim, Aleksandr L'vovich, and Viktor Borisovich Kardakov. "Solution of an inverse problem for an elastic wave equation by the method of spherical means." Sibirskii Matematicheskii Zhurnal 19, no. 4 (1978): 749-758.
\bibitem{p28}
Quinto, Eric Todd. "Support theorems for the spherical Radon transform on manifolds." International Mathematics Research Notices 2006, no. 9 (2006): 67205-67205.
\bibitem{p29}
Chihara, Hiroyuki. "Inversion of seismic-type Radon transforms on the plane." Integral Transforms and Special Functions 31, no. 12 (2020): 998-1009.
\bibitem{p30}
Della Valle, Cecile, and Camille Pouchol. "Solving Abel integral equations by regularisation in Hilbert scales." arXiv preprint arXiv:2107.12062 (2021).
\bibitem{p31}
Natterer, Frank. The mathematics of computerized tomography. Society for Industrial and Applied Mathematics, 2001.
\bibitem{p32}
P. C. Hansen, Regularization Tools Version 4.0 for Matlab 7.3, Numerical Algorithms, 46 (2007), pp. 189-194.
\bibitem{p33}
Ehrhardt, Matthias J., Kris Thielemans, Luis Pizarro, David Atkinson, Sébastien Ourselin, Brian F. Hutton, and Simon R. Arridge. "Joint reconstruction of PET-MRI by exploiting structural similarity." Inverse Problems 31, no. 1 (2014): 015001.
\bibitem{p34}
Andersson, Lars-Erik. "On the determination of a function from spherical averages." SIAM Journal on Mathematical Analysis 19, no. 1 (1988): 214-232.
\bibitem{p35}
Nguyen, M. K., and T. T. Truong. "Inversion of a new circular-arc Radon transform for Compton scattering tomography." Inverse Problems 26, no. 6 (2010): 065005.
\end{thebibliography}
\end{document}